\documentclass[12pt]{amsart}%
\usepackage{amsfonts}
\usepackage{graphicx}
\usepackage{amscd}
\usepackage{amsmath}
\usepackage{amssymb}%
\newtheorem{theorem}{Theorem}[section]
\theoremstyle{plain}

\newtheorem{corollary}[theorem]{Corollary}

\newtheorem{definition}[theorem]{Definition}
\newtheorem{lemma}[theorem]{Lemma}
\newtheorem{proposition}[theorem]{Proposition}
\theoremstyle{definition}
\newtheorem{example}[theorem]{Example}

\newtheorem{remark}[theorem]{Remark}

\numberwithin{equation}{section}

\begin{document}
\title[ ]{A General Integral}
\author{Ricardo Estrada}
\address{Department of Mathematics\\
Louisiana State University\\
Baton Rouge\\
LA 70803\\
USA}
\email{restrada@math.lsu.edu}
\author{Jasson Vindas}
\address{Department of Mathematics\\
Ghent University\\
Krijgslaan 281 Gebouw S22\\
B 9000 Gent\\
Belgium}
\email{jvindas@cage.Ugent.be}
\thanks{R. Estrada gratefully acknowledges support from NSF, through grant number 0968448.}
\thanks{J. Vindas gratefully acknowledges support by a Postdoctoral Fellowship of the
Research Foundation--Flanders (FWO, Belgium)}
\subjclass[2000]{Primary 26A39, 46F10. Secondary 26A24, 26A36}
\keywords{Distributions, \L ojasiewicz point values, distributional integration, general
integral, non-absolute integrals}

\begin{abstract}
We define an integral, the distributional integral of functions of one real
variable, that is more general than the Lebesgue and the
Denjoy-Perron-Henstock-Kurzweil integrals, and which allows the integration of
functions with distributional values everywhere or nearly everywhere.

Our integral has the property that if $f$ is locally distributionally
integrable over the real line and $\psi\in\mathcal{D}\left(  \mathbb{R}%
\right)  $ is a test function, then $f\psi$ is distributionally integrable,
and the formula%
\[
\left\langle \mathsf{f},\psi\right\rangle =\left(  \mathfrak{dist}\right)
\int_{-\infty}^{\infty}f\left(  x\right)  \psi\left(  x\right)  \,\mathrm{d}%
x\,,
\]
defines a distribution $\mathsf{f}\in\mathcal{D}^{\prime}\left(
\mathbb{R}\right)  $ that has distributional point values almost everywhere
and actually $\mathsf{f}\left(  x\right)  =f\left(  x\right)  $ almost everywhere.

The indefinite distributional integral $F\left(  x\right)  =\left(
\mathfrak{dist}\right)  \int_{a}^{x}f\left(  t\right)  \,\mathrm{d}t$
corresponds to a distribution with point values everywhere and whose
distributional derivative has point values almost everywhere equal to
$f\left(  x\right)  .$

The distributional integral is more general than the standard integrals, but
it still has many of the useful properties of those standard ones, including
integration by parts formulas, substitution formulas, even for infinite
intervals --in the Ces\`{a}ro sense--, mean value theorems, and convergence
theorems. The distributional integral satisfies a version of Hake's theorem.
Unlike general distributions, locally distributionally integrable functions
can be restricted to closed sets and can be multiplied by power functions with
real positive exponents.

\end{abstract}
\maketitle

\newpage

\tableofcontents

\newpage
\section{Introduction}\label{Sc Introduction}

In this article we construct and study the properties of a general integration
operator that can be applied to functions of one variable, $f:\left[
a,b\right]  \rightarrow\overline{\mathbb{R}}=\mathbb{R\cup}\left\{
-\infty,\infty\right\}  .$ We denote this integral as%
\begin{equation}
\left(  \mathfrak{dist}\right)  \int_{a}^{b}f\left(  x\right)  \,\mathrm{d}%
x\,, \label{1.1}%
\end{equation}
and call it the \emph{distributional integral} of $f.$ The space of
distributionally integrable functions is a vector space and the operator
(\ref{1.1}) is a linear functional in this space.

The construction gives an integral with the following properties:\smallskip

\noindent\textbf{1.} Any Denjoy-Perron-Henstock-Kurzweil integrable function is also distributionally
integrable and the integrals coincide.  In
particular any Lebesgue integrable function is distributionally
integrable and the integrals coincide. If the Denjoy-Perron-Henstock-Kurzweil integral
can be assigned the value $+\infty$ (or $-\infty$) then the distributional
integral can also be assigned the value $+\infty$ (or $-\infty$).\smallskip

\noindent\textbf{2.} If a distribution $\mathsf{f}\in\mathcal{D}^{\prime
}\left(  \mathbb{R}\right)  $ has distributional point values (as defined in
Subsection \ref{Sbs Point values}) at all points of $\left[  a,b\right]  $ and
if $f\left(  x\right)  =\mathsf{f}\left(  x\right)  $ is the function given by
those point values, then $f$ is distributionally integrable over $\left[
a,b\right]  .$\smallskip

\noindent\textbf{3.} If $f:\mathbb{R}\rightarrow\overline{\mathbb{R}}$ is
function that is distributionally integrable over any compact interval, and if
$\psi\in\mathcal{D}\left(  \mathbb{R}\right)  $ is a test function, then the
formula%
\begin{equation}
\left\langle \mathsf{f}\left(  x\right)  ,\psi\left(  x\right)  \right\rangle
=\left(  \mathfrak{dist}\right)  \int_{-\infty}^{\infty}f\left(  x\right)
\psi\left(  x\right)  \,\mathrm{d}x\,, \label{1.2}%
\end{equation}
where the integral on the right is meant as the distributional integral on any compact interval that contains the support of $\psi$, defines a distribution $\mathsf{f}\in\mathcal{D}^{\prime}\left(
\mathbb{R}\right)$. This distribution $\mathsf{f}$ has distributional point
values almost everywhere and%
\begin{equation}
\mathsf{f}\left(  x\right)  =f\left(  x\right)  \ \ \ \ \left(  \text{a.e.}%
\right)  \,. \label{1.3}%
\end{equation}
If we start with a distribution $\mathsf{f}_{0}\in\mathcal{D}^{\prime}\left(
\mathbb{R}\right)  $\ that has values everywhere, then construct the function
$f$ given by those values, and then define a distribution $\mathsf{f}%
\in\mathcal{D}^{\prime}\left(  \mathbb{R}\right)  $ by formula (\ref{1.2})
then we recover the initial distribution: $\mathsf{f=f}_{0}.$\smallskip

We call the integral a \emph{general} integral because of property \textbf{1},
which says that it is more general than the standard integrals. We call it the
\emph{distributional} integral because of \textbf{2 }and \textbf{3}, since
these properties say that functions integrable in this sense are related to
corresponding distributions in a very precise fashion.

In the same way that locally integrable Lebesgue functions $f$ give rise to
associated \textquotedblleft regular\textquotedblright\ distributions
$\mathsf{f},$ $f\leftrightarrow\mathsf{f,}$ locally integrable
distributionally functions have associated \textquotedblleft locally
integrable distributions.\textquotedblright\ Actually Denjoy-Perron-Henstock-Kurzweil
integrable functions also have canonically associated distributions
\cite{Mikusinski-Ostas}. Observe, however, that for the purposes of this
article is better to say that $f$ and $\mathsf{f}$ are \emph{associated} and
employ different notations for the function and the distribution, instead of
the standard practice of saying that $\mathsf{f}$ \textquotedblleft
is\textquotedblright\ $f.$ The question of whether a distribution can be
associated to a function or not was considered in the lecture \cite{estPol};
understanding that distributions, in general, are regularizations of
functions, and usually \emph{not} uniquely determined \cite{EF2001} allows one
to avoid common misunderstandings in the formulas used in Mathematical Physics
\cite{Ford-Oconnell}.

Our construction of the integral is based upon a characterization of positive
measures in terms of the properties of the $\phi-$transform \cite{
VEmeasures, EVraex,DroZav03,VPwavelet}, introduced in Section
\ref{Sc The phi transform}. Indeed, in the Theorems \ref{Theorem M3},
\ref{Theorem M3p}, and \ref{Theorem M4} we give conditions on the pointwise
extreme values of a distribution that guarantee that it is a positive measure,
and this allows us to consider the notions of major and minor distributional
pairs and then, in Definition \ref{Def IN3}, define the distributional
integral. In Section \ref{Sc Integral} we show that the integral is a linear
functional, that distributionally integrable functions are finite almost
everywhere and measurable, and that the integrals of functions that are equal $\left(
\text{a.e}\right)  $ coincide.

In Section \ref{Section: Indefinite integral} we study the indefinite integral%
\begin{equation}
F\left(  x\right)  =\left(  \mathfrak{dist}\right)  \int_{a}^{x}f\left(
t\right)  \,\mathrm{d}t\,, \label{1.4}%
\end{equation}
of a distributionally integrable function $f.$ We prove that $F$ is a
\L ojasiewicz function (Definition \ref{Def Loj func}), that is, it has point
values everywhere. In general $F$ will not be continuous but it will be
\textquotedblleft continuous in an average sense.\textquotedblright\ Other
integration processes have discontinuous indefinite integrals \cite[Sections
479--482]{Hobson}, but they are not even linear operations. Any \L ojasiewicz
function has associated a unique distribution $\mathsf{F,}$ $F\leftrightarrow
\mathsf{F,}$ and thus we may consider its derivative, $\mathsf{f=F}^{\prime}.$
We show that $\mathsf{F}^{\prime}$ has distributional values almost everywhere
and that actually $\mathsf{F}^{\prime}\left(  x\right)  =f\left(  x\right)  $
$\left(  \text{a.e}\right)  .$ This is a precise statement of the idea that $f$
is the derivative of $F$ almost everywhere. Later on, in Section
\ref{Section Distributions and Integration}, we are able to show that
$\mathsf{f=F}^{\prime}$ is the same distribution given by (\ref{1.2}).

In Section \ref{Section Comparison} we show that our integral is more general
than the Lebesgue integral and than the Denjoy-Perron-Henstock-Kurzweil integral. In
fact, more generally, our integral is capable of recovering a function from its
higher order differential quotients, a problem originally considered by Denjoy
in \cite{den}. We also show that \L ojasiewicz functions and distributionally
regulated functions \cite{VEStudia} are distributionally integrable, as are
the distributional derivatives of \L ojasiewicz distributions whose point values exist nearly everywhere.
The relationship between locally distributionally integrable functions and
distributions is studied in Section
\ref{Section Distributions and Integration}, not only in the space
$\mathcal{D}^{\prime}\left(  \mathbb{R}\right)  ,$ but in other spaces such as
$\mathcal{E}^{\prime}\left(  \mathbb{R}\right)  ,$ $\mathcal{S}^{\prime
}\left(  \mathbb{R}\right)  ,$ or $\mathcal{K}^{\prime}\left(  \mathbb{R}%
\right)  $ as well.

According to Hake's theorem \cite{Hake}, there are no improper
Denjoy-Perron-Henstock-Kurzweil integrals over finite intervals, since such integrals
are actually ordinary Denjoy-Perron-Henstock-Kurzweil integrals. We prove a
corresponding result, namely, if $f$ is distributionally integrable over
$\left[  a,x\right]  $ for any $x<b,$ and if $\left(  \mathfrak{dist}\right)
\int_{a}^{x}f\left(  t\right)  \,\mathrm{d}t$ has a distributional limit
$L$\ as $x\rightarrow b,$ then $f$ is integrable over $\left[  a,b\right]  $
and the integral is equal to $L.$ We apply this result to show that if $f$ is
distributionally integrable over $\left[  a,b\right]  $ then so are the
functions $\left(  x-a\right)  ^{\alpha}\left(  b-x\right)  ^{\beta}f\left(
x\right)  $ for any real numbers $\alpha>0$ and $\beta>0.$

We prove a bounded convergence theorem, a monotone convergence theorem, and a
version of Fatou's lemma in Section \ref{Section Convergence Thms}. We
examine changes of variables in Section \ref{Section Change of Variables},
showing, in particular, that distributional integrals become Ces\`{a}ro type
integrals when the change sends a finite interval to an infinite one. The
three mean value theorems of integral calculus are proved in Section
\ref{Section MVT}.

In the last section, Section \ref{Section Examples}, we provide several
examples that illustrate our ideas. We give examples of functions that are
distributionally integrable but not Denjoy-Perron-Henstock-Kurzweil integrable,
examples of distributionally integrable functions that are not \L oja\-siewicz
functions, and examples of \L ojasiewicz functions which are not indefinite
integrals. We consider the boundary values of the Poisson integral of a
distributionally integrable function. Moreover, we consider the Fourier series
of periodic locally distributionally integrable functions and the Fourier
transform of tempered locally distributionally integrable functions. We also
explain why the Cauchy representation formula
\begin{equation}
F\left(  z\right)  =\frac{1}{2\pi i}\left(  \mathfrak{dist}\right)
\int_{-\infty}^{\infty}\frac{f\left(  \xi\right)  }{\xi-z}\,\mathrm{d}\xi\,,
\label{1.5}%
\end{equation}
holds for certain functions $F$ analytic in $\Im m\,z>0$ whose boundary values
on $\mathbb{R}$ come from locally integrable distributions (as $f\left(
\xi\right)  =\xi^{-1}e^{-i/\xi},$ for instance), and why such a formula does
not hold, even in the principal value sense, for non distributionally
integrable functions (as $f\left(  \xi\right)  =\xi^{-1},$ for instance).

There have been several studies that involve distributions and integration. Let
us emphasize that our integral is a method to find the integral of
\emph{functions} as are, let us say, the Riemann or the Denjoy integrals. A
completely different question is the integration of \emph{distributions}.
Indeed, observe, first of all, that the fact that any distribution
$\mathsf{f}\in\mathcal{D}^{\prime}\left(  \mathbb{R}\right)  $ has a primitive
$\mathsf{F}\in\mathcal{D}^{\prime}\left(  \mathbb{R}\right)  ,$ $\mathsf{F}%
^{\prime}=\mathsf{f,}$ is trivial. If $\mathsf{F}$ has values at $x=a$ and at
$x=b$ then we say that $\mathsf{f}$ is integrable over $\left[  a,b\right]  $
and write \cite{CF}
\begin{equation}
\int_{a}^{b}\mathsf{f}\left(  x\right)  \,\mathrm{d}x=\mathsf{F}\left(
b\right)  -\mathsf{F}\left(  a\right)  \,. \label{1.6}%
\end{equation}
Hence $\int_{a}^{b}\mathsf{f}\left(  x\right)  \,\mathrm{d}x$ is a
\emph{number.} This notion is due to the Polish school \cite{AMS,loj} and has
several applications, as in the theory of sampling theorems \cite{Walter88}.
On the other hand, Silva and Sikorski, independently, used their definitions of
the integral of distributions to write Fourier transforms and convolutions of
distributions as integrals \cite{Silva, Sikorski}. Moreover, several authors
\cite{Baumer et al, Talvila}\ have considered the class of continuously
integrable distributions, that is, those distributions with a continuous
primitive; observe, however, that continuously integrable distributions may
not have values at any point, and thus are not really \emph{functions,} in general.

We should point out that one can devise a simple procedure for the
construction of primitives of \emph{functions} by using the fact that
\emph{distributions} are known to have primitives. Indeed, start with a
function $f,$ associate to it a distribution $\mathsf{f,}$ construct the
distributional primitive $\mathsf{F,}$ that is, $\mathsf{F}^{\prime
}=\mathsf{f,}$ and then construct the function $F$ associated to $\mathsf{F.}$
Then $F$ would be a primitive of $f.$ Unfortunately, this procedure fails, in
general, because there is no unique way to assign a distribution $\mathsf{f}$
to a given function $f,$ as follows from the Theorem \ref{Theorem DI 1}.
Interestingly, however, it does work sometimes, as we show, for instance, for
\L ojasiewicz\ functions, because in this case all the associations are unique
\cite{loj}.

\section{Preliminaries}
\label{Sc Prelim}
In this section we have collected several important ideas that will play a
role in our construction of a general distributional integral.

\subsection{Spaces}\label{Sbs spaces}
We use the term smooth function to mean a $C^{\infty}$ function.
The Schwartz spaces of test functions $\mathcal{D},$ $\mathcal{E},$ and $\mathcal{S}$
and the corresponding spaces of distributions are well known \cite{AMS, kan,
sch, Strichartz, Vladimirovbook}. Recall that $\mathcal{E}$ consists of all smooth functions, while $\mathcal{D}$ and $\mathcal{S}$ stand, respectively, for the spaces of smooth compactly supported and rapidly decreasing test functions. In general \cite{zem}, we call a topological
vector space $\mathcal{A}$ a space of test functions if $\mathcal{D}%
\subseteq\mathcal{A}\subseteq\mathcal{E},$ where the inclusions are
continuous and dense, and if $\frac{d}{dx}$ is a continuous operator on $\mathcal{A}$.
A useful space, particularly in the study of distributional asymptotic
expansions \cite{r43, est-kan, pil, vla} is $\mathcal{K}^{\prime}(\mathbb{R})$, the dual
of $\mathcal{K}(\mathbb{R})$. The test function space $\mathcal{K}(\mathbb{R})$ is given by $\mathcal{K}(\mathbb{R})=\bigcup_{\alpha\in\mathbb{R}}\mathcal{K}_{\alpha}(\mathbb{R})$, the union having topological meaning, where each $\mathcal{K}_{\alpha}(\mathbb{R})$ consists of those smooth functions $\phi$ that satisfy 
\begin{equation}
\label{strongasympt}
\phi^{(m)}(t)=O(\left|t\right|^{\alpha-m})\ \ \ \mbox{as} \left|t\right|\to\infty\: , \ \ \forall m\in\mathbb{N}\ ,
\end{equation}
and is provided with the topology generated by the family of seminorms
\begin{equation}
\max\{\sup_{\left|t\right|\leq1}|\phi^{(m)}(t)|, \sup_{\left|t\right|\geq1}\left|t\right|^{m-\alpha}|\phi^{(m)}(t)|\}\ .
\end{equation}
The space $\mathcal{K}^{\prime}(\mathbb{R})$
plays a fundamental role in the theory of summability of distributional
evaluations \cite{est2}.

We shall use the notation $\mathsf{f,}$ $\mathsf{g,}$ $\mathsf{F,}$ etc. to
denote distributions, while $f,$ $g,$ $F,$ etc. will denote functions. If $f$
is a locally Lebesgue integrable function and $\mathsf{f}$ is the
corresponding \emph{regular} distribution, given by $\left\langle
\mathsf{f},\phi\right\rangle =\int_{-\infty}^{\infty}f\left(  x\right)
\phi\left(  x\right)  \,\mathrm{d}x$ for $\phi\in\mathcal{D}(\mathbb{R}),$
then we shall use the notation $f\leftrightarrow\mathsf{f;}$ naturally $f$ is
not really a function but an equivalence class of functions equal almost
everywhere.

\subsection{Point values}\label{Sbs Point values}

In \cite{loj,loj2} \L ojasiewicz defined the value of a distribution
$\mathsf{f}\in\mathcal{D}^{\prime}(\mathbb{R})$ at the point $x_{0}$ as the
limit
\begin{equation}
\mathsf{f}(x_{0})=\lim_{\varepsilon\rightarrow0}\mathsf{f}(x_{0}+\varepsilon
x)\,,
\end{equation}
if the limit exists in $\mathcal{D}^{\prime}(\mathbb{R})$, that is if
\begin{equation}
\lim_{\varepsilon\rightarrow0}\left\langle \mathsf{f}(x_{0}+\varepsilon
x),\phi(x)\right\rangle =\mathsf{f}(x_{0})\int_{-\infty}^{\infty}%
\phi(x)\,\mathrm{d}x\,, \label{pe1}%
\end{equation}
for each $\phi\in\mathcal{D}(\mathbb{R})$. It was shown by \L ojasiewicz that
the existence of the distributional point value $\mathsf{f}(x_{0})=\gamma$ is
equivalent to the existence of $n\in\mathbb{N}$, and a primitive of order $n$
of $\mathsf{f}$, that is $\mathsf{F}^{(n)}=\mathsf{f}$, which corresponds,
near $x_{0}$, to a continuous function $F$ that satisfies
\begin{equation}
\lim_{x\rightarrow x_{0}}\frac{n!F(x)}{\left(  x-x_{0}\right)  ^{n}}=\gamma\,.
\label{2.3}%
\end{equation}
One can also define point values by using the operator
\begin{equation}
\partial_{x_{0}}\left(  \mathsf{f}\right)  =\left(  \left(  x-x_{0}\right)
\mathsf{f}\left(  x\right)  \right)  ^{\prime}\,, \label{2.3p}%
\end{equation}
since $\mathsf{f}_{1}(x_{0})=\gamma$ if and only if $\mathsf{f}(x_{0}%
)=\gamma,$ where $\mathsf{f}=\partial_{x_{0}}\left(  \mathsf{f}_{1}\right)  .$
Therefore \cite{CF} $\mathsf{f}$ has a distributional value equal to $\gamma$
at $x=x_{0}$ if and only if there exists $n\in\mathbb{N}$ and a function
$f_{n},$ continuous at $x=x_{0},$ with $f_{n}\left(  x_{0}\right)  =\gamma,$
such that $\mathsf{f}=\partial_{x_{0}}^{n}\left(  \mathsf{f}_{n}\right)  ,$ near $x_{0}$,
where $f_{n}\leftrightarrow\mathsf{f}_{n}\mathsf{.}$

Suppose that $\mathsf{f}\in\mathcal{S}^{\prime}(\mathbb{R})$ has the
\L ojasiewicz point value $\mathsf{f}(x_{0})=\gamma$. Initially, (\ref{pe1})
is only supposed to hold for $\phi\in\mathcal{D}(\mathbb{R})$; however, it is
shown in \cite{est3,VP09} that (\ref{pe1}) will remain true for all $\phi
\in\mathcal{S}(\mathbb{R})$. Actually using the notion of the Ces\`{a}ro
behavior of a distribution at infinity \cite{est2} explained below,
(\ref{pe1}) will hold \cite{est3,vindasthesis,VEmeasures,VEEdin} if $f\left(
x\right)  =O(\left\vert x\right\vert ^{\beta})$ $\left(  \text{C}\right)  ,$
as $\left\vert x\right\vert \rightarrow\infty,$ $\phi\left(  x\right)
=O(\left\vert x\right\vert ^{\alpha}),$ strongly $\left\vert x\right\vert
\rightarrow\infty,$ and $\alpha<-1,$ $\alpha+\beta<-1.$ An asymptotic estimate
is \emph{strong} if it remains valid after differentiation of any order, namely, if (\ref{strongasympt}) is satisfied.

The notion of distributional point value introduced by \L ojasiewicz has been
shown to be of fundamental importance in analysis \cite{CF, est1, ML, Peetre,
VEStudia, VEFourier, wal, Walter95, Walter01}. It seems to have originated in
the idea of generalized differentials studied by Denjoy in \cite{den}. There
are other notions of distributional point values as that of Campos Ferreira
\cite{CF,CF2}, who also introduced the very useful concept of bounded
distributions (see also \cite{Zie}). A distribution $\mathsf{f}$ is said to be
\emph{distributionally bounded} at $x_{0}$ if $\mathsf{f}(x_{0}+\varepsilon
x)=O(1)$ as $\varepsilon\to0$ in $\mathcal{D}^{\prime}(\mathbb{R})$, i.e., for
each test function $\left\langle \mathsf{f}(x_{0}+\varepsilon x),\phi
(x)\right\rangle =O(1).$ Distributional boundedness admits a characterization
\cite{vindas2} similar to that of \L ojasiewicz point values, but this time
one replaces (\ref{2.3}) by $F(x)=O(\left|  x-x_{0}\right|  ^{n})$.

The distributional limit $\lim_{x\rightarrow x_{0}%
}\mathsf{f}(x)$ exists and equals $L$ if 
\begin{equation}
\lim_{\varepsilon\rightarrow0}\left\langle \mathsf{f}(x_{0}+\varepsilon
x),\phi(x)\right\rangle =L\int_{-\infty}^{\infty}%
\phi(x)\,\mathrm{d}x\,,%
\end{equation}
for all test functions with support contained in $\mathbb{R}\setminus\{0\}.$ If the point value $\mathsf{f}(x_{0})$ exists
distributionally then the distributional limit $\lim_{x\rightarrow x_{0}%
}\mathsf{f}(x)$ exists and equals $\mathsf{f}(x_{0}).$ On the other hand, if
$\lim_{x\rightarrow x_{0}}\mathsf{f}(x)=L$ distributionally then there exist
constants $a_{0},\ldots,a_{n}$ such that $\mathsf{f}(x)=\mathsf{f}_{0}%
(x)+\sum_{j=0}^{n}a_{j}\delta^{\left(  j\right)  }(x-x_{0}),$ where the
distributional point value $\mathsf{f}_{0}(x_{0})$ exists and equals $L.$  Notice that the distributional limit  $\lim_{x\rightarrow x_{0}}\mathsf{f}(x)$
can actually be defined for distributions $\mathsf{f}\in\mathcal{D}^{\prime}(\mathbb{R}%
\setminus\{x_{0}\}).$ 

We may also consider lateral limits. We say that the distributional lateral
value $\mathsf{f}(x_{0}^{+})$ exists if $\mathsf{f}(x_{0}^{+})=\lim
_{\varepsilon\rightarrow0^{+}}\mathsf{f}(x_{0}+\varepsilon x)$ in
$\mathcal{D}^{\prime}(0,\infty),$ that is,
\begin{equation}
\lim_{\varepsilon\rightarrow0^{+}}\left\langle \mathsf{f}(x_{0}+\varepsilon
x),\phi(x)\right\rangle =\mathsf{f}(x_{0}^{+})\int_{0}^{\infty}\phi
(x)\,\mathrm{d}x\;\;,\;\;\phi\in\mathcal{D}(0,\infty)\,. \label{3.9.48}%
\end{equation}
Similar definitions apply to $\mathsf{f}(x_{0}^{-}).$ Notice that the
distributional limit $\lim_{x\rightarrow x_{0}}\mathsf{f}(x)$ exists if and
only if the distributional lateral limits $\mathsf{f}(x_{0}^{-})$ and
$\mathsf{f}(x_{0}^{+})$ exist and coincide.

\subsection{The Ces\`{a}ro behavior of distributions at
infinity}\label{Sbs Cesaro}

The Ces\`{a}ro behavior \cite{est2,est-kan} of a distribution at infinity is
studied by using the order symbols $O\left(  x^{\alpha}\right)  $ and
$o\left(  x^{\alpha}\right)  $ in the Ces\`{a}ro sense. If $\mathsf{f}%
\in\mathcal{D}^{\prime}(\mathbb{R})$ and $\alpha\in\mathbb{R\setminus}\left\{
-1,-2,-3,...\right\}  $, we say that $\mathsf{f}(x)=O\left(  x^{\alpha
}\right)  $ as $x\rightarrow\infty$ in the Ces\`{a}ro sense and write%
\begin{equation}
\mathsf{f}(x)=O\left(  x^{\alpha}\right)  \ (\mathrm{C})\,,\ \text{as}%
\ x\rightarrow\infty\,,
\end{equation}
if there exists $N\in\mathbb{N}$ such that every primitive $\mathsf{F}$ of
order $N$, i.e., $\mathsf{F}^{(N)}=\mathsf{f}$, corresponds for large
arguments to a locally integrable function, $\mathsf{F}\leftrightarrow F$,
that satisfies the ordinary order relation
\begin{equation}
F(x)=p(x)+O\left(  x^{\alpha+N}\right)  \,,\ \text{as}\ x\rightarrow\infty\,,
\end{equation}
for a suitable polynomial $p$ of degree at most $N-1$. Note that if
$\alpha>-1$, then the polynomial $p$ is irrelevant. If we want to specify the
value $N$, we write $(\mathrm{C},N)$ instead of just $(\mathrm{C})$. A similar
definition applies to the little $o$ symbol. The definitions when
$x\rightarrow-\infty$ are clear.

The elements of $\mathcal{S}^{\prime}(\mathbb{R})$ can be characterized by
their Ces\`{a}ro behavior at $\pm\infty$, in fact, $\mathsf{f}\in
\mathcal{S}^{\prime}(\mathbb{R})$ if and only if there exists $\alpha
\in\mathbb{R}$ such that $\mathsf{f}(x)=O\left(  x^{\alpha}\right)
\ (\mathrm{C})$, as $x\rightarrow\infty$, and $\mathsf{f}(x)=O\left(
\left\vert x\right\vert ^{\alpha}\right)  \ (\mathrm{C})$, as $x\rightarrow
-\infty$. On the other hand, this is true for all $\alpha\in\mathbb{R}$ if and
only if $\mathsf{f}\in\mathcal{K}^{\prime}(\mathbb{R}).$

Using these ideas, one can define the limit of a distribution at $\infty$ in
the Ces\`{a}ro sense. We say that $\mathsf{f}\in\mathcal{D}^{\prime
}(\mathbb{R})$ has a limit $L$ at infinity in the Ces\`{a}ro sense and write
\begin{equation}
\lim_{x\rightarrow\infty}\mathsf{f}(x)=L\ \ \ (\mathrm{C})\,,
\end{equation}
if $\mathsf{f}(x)=L+o(1)\ (\mathrm{C})$, as $x\rightarrow\infty$.

The Ces\`{a}ro behavior of a distribution $\mathsf{f}$ at infinity is related
to the parametric behavior of $\mathsf{f}(\lambda x)$ as $\lambda
\rightarrow\infty$. In fact, one can show \cite{est-kan,vindas1,vindas2} that
if $\alpha>-1$, then $\mathsf{f}(x)=O\left(  x^{\alpha}\right)  \ (\mathrm{C}%
)$ as $x\rightarrow\infty$ and $\mathsf{f}(x)=O\left(  \left\vert x\right\vert
^{\alpha}\right)  \ (\mathrm{C})$ as $x\rightarrow-\infty$ if and only if
\begin{equation}
\mathsf{f}(\lambda x)=O\left(  \lambda^{\alpha}\right)  \ \text{as}%
\ \lambda\rightarrow\infty\,,
\end{equation}
where the last relation holds weakly in $\mathcal{D}^{\prime}(\mathbb{R})$,
i.e., for all $\phi\in\mathcal{D}(\mathbb{R})$ fixed, $\left\langle
\mathsf{f}(\lambda x),\phi(x)\right\rangle =O\left(  \lambda^{\alpha}\right)
,$ $\lambda\rightarrow\infty$. A distribution $\mathsf{f}$ belongs to the
space $\mathcal{K}^{\prime}(\mathbb{R})$ if and only if it satisfies the
moment asymptotic expansion \cite{r43,est-kan},%
\begin{equation}
\mathsf{f}(\lambda x)\sim\sum\limits_{n=0}^{\infty}\frac{\left(  -1\right)
^{n}\mu_{n}\delta^{\left(  n\right)  }\left(  x\right)  }{n!\lambda^{n+1}%
}\,,\ \ \text{as \ }\lambda\rightarrow\infty\,,
\end{equation}
where the $\mu_{n}=\left\langle \mathsf{f}\left(  x\right)  ,x^{n}%
\right\rangle $ are the moments of $\mathsf{f.}$

\subsection{Evaluations}\label{Sbs Evaluations}

Let $\mathsf{f}\in\mathcal{D}^{\prime}\left(  \mathbb{R}\right)  $ with
support bounded on the left. If $\phi\in\mathcal{E}\left(  \mathbb{R}\right)
$ then the evaluation $\left\langle \mathsf{f}\left(  x\right)  ,\phi\left(
x\right)  \right\rangle $ will not be defined, in general. We say that the
evaluation exists in the Ces\`{a}ro sense and equals $L,$ written as%
\begin{equation}
\left\langle \mathsf{f}\left(  x\right)  ,\phi\left(  x\right)  \right\rangle
=L\ \ \ \left(  \mathrm{C}\right)  \,, \label{T.1}%
\end{equation}
if $\mathsf{g}\left(  x\right)  =L+o\left(  1\right)  $ $\left(
\text{C}\right)  $\ as $x\rightarrow\infty,$ where $\mathsf{g}$ is the
primitive of $\mathsf{f}\phi$ with support bounded on the left. A similar
definition applies if $\operatorname*{supp}\mathsf{f}$ is bounded on the
right. Observe that if $\mathsf{f}$ corresponds to a locally integrable
function $f$ with $\operatorname*{supp}\mathsf{f}\subset\lbrack a,\infty)$
then (\ref{T.1}) means that%
\begin{equation}
\int_{a}^{\infty}f\left(  x\right)  \phi\left(  x\right)  \,\mathrm{d}%
x=L\ \ \ \left(  \mathrm{C}\right)  \,. \label{T.2}%
\end{equation}
Naturally, this will hold for \emph{any} integration method we use. If
$\mathsf{f}\left(  x\right)  =\sum_{n=0}^{\infty}a_{n}\delta\left(
x-n\right)  $ then (\ref{T.1}) tells us that%
\begin{equation}
\sum_{n=0}^{\infty}a_{n}\phi\left(  n\right)  =L\ \ \ \left(  \mathrm{C}%
\right)  \,. \label{T.3}%
\end{equation}

In the general case when the support of $\mathsf{f}$ extends to both $-\infty$
and $+\infty,$ there are various different but related notions of evaluations
in the Ces\`{a}ro sense (or in any other summability sense, in fact). If
$\mathsf{f}$ admits a representation of the form $\mathsf{f}=\mathsf{f}%
_{1}+\mathsf{f}_{2},$ with $\operatorname*{supp}\mathsf{f}_{1}$ bounded on the
left and $\operatorname*{supp}\mathsf{f}_{2}$ bounded on the right, such that
$\left\langle \mathsf{f}_{j}\left(  x\right)  ,\phi\left(  x\right)
\right\rangle =L_{j}$ $\left(  \text{C}\right)  $ exist, then we say that the
$\left(  \text{C}\right)  $ evaluation $\left\langle \mathsf{f}\left(
x\right)  ,\phi\left(  x\right)  \right\rangle $ $\left(  \text{C}\right)  $
exists and equals $L=L_{1}+L_{2}.$ This is clearly independent of the
decomposition. The notation (\ref{T.1}) is used in this situation as well.

It happens many times that $\left\langle \mathsf{f}\left(  x\right)
,\phi\left(  x\right)  \right\rangle $ $\left(  \text{C}\right)  $ does not
exist, but the symmetric limit, $\lim_{x\rightarrow\infty}\left\{
\mathsf{g}\left(  x\right)  -\mathsf{g}\left(  -x\right)  \right\}  =L,$ where
$\mathsf{g}$ is any primitive of $\mathsf{f}\phi$, exists in the $\left(
\text{C}\right)  $ sense. Then we say that the evaluation $\left\langle
\mathsf{f}\left(  x\right)  ,\phi\left(  x\right)  \right\rangle $ exists in
the principal value Ces\`{a}ro sense \cite{est-kan,VEChina}, and write%
\begin{equation}
\mathrm{p.v.}\left\langle \mathsf{f}\left(  x\right)  ,\phi\left(  x\right)
\right\rangle =L\ \ \ \left(  \mathrm{C}\right)  \,. \label{T.4}%
\end{equation}
Observe that $\mathrm{p.v.}\sum_{n=-\infty}^{\infty}a_{n}\phi\left(  n\right)
=L$ $\left(  \text{C}\right)  $ if and only if $\sum_{-N}^{N}a_{n}\phi\left(
n\right)  $ $\rightarrow L$ $\left(  \text{C}\right)  $ as $N\rightarrow
\infty$ while $\mathrm{p.v.}\int_{-\infty}^{\infty}f\left(  x\right)
\phi\left(  x\right)  \,\mathrm{d}x=L$ $\left(  \text{C}\right)  $\ if and
only if $\int_{-A}^{A}f\left(  x\right)  \phi\left(  x\right)  \,\mathrm{d}x$
$\rightarrow L$ $\left(  \text{C}\right)  $ as $A\rightarrow\infty.$

A very useful intermediate notion is the following
\cite{VEStudia,VEFourier,VEChina}. If there exists $k$ such that
\begin{equation}
\lim_{x\rightarrow\infty}\left\{  \mathsf{g}\left(  ax\right)  -\mathsf{g}%
\left(  -x\right)  \right\}  =L\ \ \ \left(  \mathrm{C},k\right)
\,,\ \ \ \forall a>0\,, \label{T.5}%
\end{equation}
we say that the distributional evaluation exists in the e.v. Ces\`{a}ro sense
and write%
\begin{equation}
\mathrm{e.v.}\left\langle \mathsf{f}\left(  x\right)  ,\phi\left(  x\right)
\right\rangle =L\ \ \ \left(  \mathrm{C},k\right)  \,, \label{T.6}%
\end{equation}
or just $\mathrm{e.v.}\left\langle \mathsf{f}\left(  x\right)  ,\phi\left(
x\right)  \right\rangle =L$ $\ \left(  \mathrm{C}\right)  $ if there is no
need to call the attention to the value of $k.$

\subsection{\L ojasiewicz distributions}\label{Sbs Lojasiewics dist}

There is a class of distributions that correspond to ordinary functions, the
class of \L ojasiewicz distributions. In general \L ojasiewicz distributions
are not regular distributions, that is, they correspond to ordinary functions
that are not locally Lebesgue integrable functions.

The simplest class of distributions that correspond to functions are those
that come from continuous functions. If $f\leftrightarrow\mathsf{f}$ and $f$
is continuous then it is an ordinary function: We can always say what
$f\left(  x_{0}\right)  $ is for any $x_{0}.$ The function $f$ is not just
defined almost everywhere but it is actually defined \emph{everywhere.}%
\smallskip

\begin{definition}
\label{Def Loj dist}A distribution $\mathsf{f}$ is a \textit{\L ojasiewicz
distribution} if the distributional point value $\mathsf{f}\left(
x_{0}\right)  $ exists for every $x_{0}\in\mathbb{R}.$\smallskip
\end{definition}

\begin{definition}
\label{Def Loj func}A function $f$ defined in $\mathbb{R}$ is called a
\textit{\L ojasiewicz function} if there exists a \L ojasiewicz distribution
$\mathsf{f}$ such that
\begin{equation}
\mathsf{f}\left(  x\right)  =f\left(  x\right)  \ \ \ \ \ \forall
x\in\mathbb{R}\,.
\end{equation}

\end{definition}

The correspondence $f\leftrightarrow\mathsf{f}$ is clearly and uniquely defined in the case
of \L ojasiewicz functions and distributions \cite{loj}. The \L ojasiewicz functions can
be considered as a distributional generalization of continuous functions. They
are defined at all points, and furthermore the value at each given point is
not arbitrary but the (distributional) limit of the function as one approaches
the given point. The \L ojasiewicz functions and distributions were introduced
in \cite{loj}.

If $\mathsf{f}$ is a \L ojasiewicz distribution, and $\mathsf{F}$ is a
primitive, $\mathsf{F}^{\prime}=\mathsf{f},$ then $\mathsf{F}$ is also a
\L ojasiewicz distribution. If $\mathsf{f}$\textsf{ }is a \L ojasiewicz
distribution and $\psi$ is a smooth function, then $\psi\mathsf{f}$ is a
\L ojasiewicz distribution and%
\begin{equation}
(\psi\mathsf{f)}\left(  x\right)  =\psi\left(  x\right)  \mathsf{f}\left(
x\right)  \,.
\end{equation}

If $f$ is a \L ojasiewicz function, $f\leftrightarrow\mathsf{f},$ then we can
\emph{define} its definite integral \cite{AMS,loj} as
\begin{equation}
\int_{a}^{b}f\left(  x\right)  \,\mathrm{d}x=\mathsf{F}\left(  b\right)
-\mathsf{F}\left(  a\right)  \,, \label{INT}%
\end{equation}
where $\mathsf{F}^{\prime}=\mathsf{f}.$ The evaluation of $\mathsf{f}$ on a
test function $\phi,$ $\left\langle \mathsf{f},\phi\right\rangle ,$ can
actually be given as an integral, namely,
\begin{align}
\left\langle \mathsf{f},\phi\right\rangle  &  =\int_{-\infty}^{\infty}f\left(
x\right)  \phi\left(  x\right)  \,\mathrm{d}x\label{INTeval}\\
&  =\int_{a}^{b}f\left(  x\right)  \phi\left(  x\right)  \,\mathrm{d}%
x\,,\ \ \ \ \ \ \ \ \phi\in\mathcal{D}\left(  \mathbb{R}\right)  \,,\nonumber
\end{align}
where $\operatorname*{supp}\phi\subset\left[  a,b\right]  .$ We will give a
rather constructive procedure below (Sections \ref{Sc Integral} and
\ref{Section Comparison}) to calculate (\ref{INT}).

If $f_{0}$ is a \L ojasiewicz function, $f_{0}\leftrightarrow\mathsf{f}_{0},$
defined for $x<a,$ and $f_{1}$ is a \L ojasiewicz function, $f_{1}%
\leftrightarrow\mathsf{f}_{1},$ defined for $x>a,$ and if the distributional
lateral limits $\mathsf{f}_{0}\left(  a-0\right)  $ and $\mathsf{f}_{1}\left(
a+0\right)  $ exist and coincide, then there is a \L ojasiewicz function $f$
whose restriction to $\left(  -\infty,a\right)  $ is $f_{0}$ and whose
restriction to $\left(  a,\infty\right)  $ is $f_{1}.$

A typical example of a \L ojasiewicz function is
\begin{equation}
s_{\alpha,\beta}\left(  x\right)  =\left\{
\begin{array}
[c]{c}%
\left\vert x\right\vert ^{\alpha}\sin\left\vert x\right\vert ^{-\beta
},\ \ \ x\neq0\,,\\
0\ \ ,\ \ x=0\,,
\end{array}
\right.
\end{equation}
for $\alpha\in\mathbb{C}$ and $\beta>0.$ If $H$ is the Heaviside function,
then the functions $H\left(  \pm x\right)  s_{\alpha,\beta}\left(  x\right)  $
and their linear combinations are also \L ojasiewicz functions. It is not hard
to see that this implies that derivatives of arbitrary order of $\mathsf{s}%
_{\alpha,\beta},$ where $s_{\alpha,\beta}\leftrightarrow\mathsf{s}%
_{\alpha,\beta},$ are also \L ojasiewicz distributions. These are rapidly
oscillating functions. However, not all fast oscillating functions are
\L ojasiewicz functions. Curiously, the regular distribution $\sin\left(
\ln\left\vert x\right\vert \right)  $ is not a \L ojasiewicz distribution
since the distributional value at $x=0$ does not exist in the \L ojasiewicz
sense, even though it exists and equals $0$ in the Campos Ferreira sense
\cite{CF}.

\subsection{Distributionally regulated functions}\label{Sbs Regulated}

Another case when a distribution corresponds to a function\ is the case of
regulated distributions, introduced and studied in \cite{VEStudia}. They are
generalizations of the ordinary regulated functions \cite{Dido}, which are
functions whose lateral limits exist at all points, although they may be
different. They are related to the recently introduced \textquotedblleft
thick\textquotedblright\ points \cite{EF2007}.\smallskip

\begin{definition}
\label{Def Reg}A distribution $\mathsf{f}$ is called a \textit{regulated
distribution} if the distributional lateral limits
\begin{equation}
\mathsf{f}\left(  x_{0}^{+}\right)  \text{ \ \ and \ \ }\mathsf{f}\left(
x_{0}^{-}\right)  \,,
\end{equation}
exist $\forall x_{0}\in\mathbb{R},$ and there are no delta functions at any
point.\smallskip
\end{definition}
The statement that ``there are no delta functions'' at any point explicitly means that for each $\phi\in\mathcal{D}(\mathbb{R})$ and any $x_{0}\in\mathbb{R}$
\begin{equation}\label{jbeq}
\lim_{\varepsilon\to0^{+}}\left\langle \mathsf{f}(x_{0}+\varepsilon x),\phi(x)\right\rangle=\mathsf{f}(x_{0}^{-})\int_{-\infty}^{0}\phi(x)\mathrm{d}x+\mathsf{f}(x_{0}^{+})\int_{0}^{\infty}\phi(x)\mathrm{d}x\ .
\end{equation}
The relation (\ref{jbeq}) is known as (pointwise) distributional jump behavior and has interesting applications in the theory of Fourier series \cite{EVDiffMeans,VEJumpLog,VEChina}.

If $\mathsf{f}\left(  x_{0}^+\right)  =\mathsf{f}\left(  x_{0}^-\right)  $
then $\mathsf{f}\left(  x_{0}\right)  $ exists, since these distributions do
not have delta functions, and therefore we can define the function%
\begin{equation}
f\left(  x_{0}\right)  =\mathsf{f}\left(  x_{0}\right)  \,,
\end{equation}
for these $x_{0}.$ Then $f$ is called a distributionally regulated function. The
function $f$ is defined in the set $\mathbb{R}\setminus\mathfrak{S},$ where
$\mathfrak{S}$ is the set of points $x_{0}$ where $\mathsf{f}\left(
x_{0}^+\right)  \neq\mathsf{f}\left(  x_{0}^-\right)  .$ The set
$\mathfrak{S}$ has measure zero since in fact it is countable at the most
\cite{VEStudia}. One can actually define%
\begin{equation}
f\left(  x_{0}\right)  =\frac{\mathsf{f}\left(  x_{0}^+\right)  +\mathsf{f}%
\left(  x_{0}^-\right)  }{2}\,,
\end{equation}
and this is defined everywhere.

The basic properties of the distributionally regulated functions and the
corresponding regulated distributions are the following. If $\mathsf{f}$ is a
regulated distribution, and $\mathsf{F}$ is a primitive, $\mathsf{F}^{\prime
}=\mathsf{f},$ then $\mathsf{F}$ is a \L ojasiewicz distribution. If
$\mathsf{f}$ is a regulated distribution and $\psi$ is a smooth function, then
$\psi\mathsf{f}$ is a regulated distribution too. If $f$ is a regulated
function, $f\leftrightarrow\mathsf{f},$ then we can \emph{define} its definite
integral as%
\begin{equation}
\int_{a}^{b}f\left(  x\right)  \,\mathrm{d}x=\mathsf{F}\left(  b\right)
-\mathsf{F}\left(  a\right)  \,, \label{Int}%
\end{equation}
where $\mathsf{F}^{\prime}=\mathsf{f}.$ Then%
\begin{equation}
\left\langle \mathsf{f},\phi\right\rangle =\int_{-\infty}^{\infty}f\left(
x\right)  \phi\left(  x\right)  \,\mathrm{d}x\ ,\ \ \ \ \ \ \ \ \phi
\in\mathcal{D}\left(  \mathbb{R}\right)  \,.
\end{equation}
As in the case of \L ojasiewicz functions, the integral that we will define in
Section \ref{Sc Integral} coincides with (\ref{Int}) for distributionally
regulated functions (Theorem \ref{Theorem COM 5}).

\subsection{Romanovski's lemma}\label{Sbs Roma}

We shall use the following useful result \cite{Roma}, the Romanovski's lemma,
in some of our proofs. See \cite{Gordon} for many interesting applications of
this result, and \cite{EVraex} for generalizations to several variables.

\begin{theorem}
\label{TheoremRomanovskiLemma}\textrm{(Romanovski's lemma)} Let $\mathfrak{F}$
be a family of open intervals in $\left(  a,b\right)  $ with the following
four properties:

I. If $\left(  \alpha,\beta\right)  \in\mathfrak{F}$ and $\left(  \beta
,\gamma\right)  \in\mathfrak{F},$ then $\left(  \alpha,\gamma\right)
\in\mathfrak{F}.$

II. If $\left(  \alpha,\beta\right)  \in\mathfrak{F}$ and $\left(
\gamma,\delta\right)  \subset\left(  \alpha,\beta\right)  $ then $\left(
\gamma,\delta\right)  \in\mathfrak{F}.$

III. If $\left(  \alpha,\beta\right)  \in\mathfrak{F}$ for all $\left[
\alpha,\beta\right]  \subset\left(  c,d\right)  $ then $\left(  c,d\right)
\in\mathfrak{F}.$

IV. If all the intervals contiguous to a perfect closed set $K\subset\left[
a,b\right]  $ belong to $\mathfrak{F}$ then there exists an interval
$I\in\mathfrak{F}$ with $I\cap K\neq\emptyset.$

Then $\left(  a,b\right)  \in\mathfrak{F}.$
\end{theorem}

Observe that if we take $K=\left[  a,b\right]  $ in IV we obtain that
$\mathfrak{F}\neq\left\{  \emptyset\right\}  ,$ but it may be easier to show
this separately.

\subsection{Measures}\label{Sbs Measures}

We shall use the following nomenclature. A (Radon) measure would mean a
\emph{positive} functional on the space of compactly supported continuous
functions, which would be denoted by integral notation such as $\mathrm{d}%
\mu,$ or by distributional notation, $\mathsf{f}=\mathsf{f}_{\mu},$ so that%
\begin{equation}
\left\langle \mathsf{f},\phi\right\rangle =\int_{\mathbb{R}}\phi\left(
x\right)  \,\mathrm{d}\mu(x)\,, \label{M.1}%
\end{equation}
and $\left\langle \mathsf{f},\phi\right\rangle \geq0$ if $\phi\geq0.$ A signed
measure is a real bounded functional on the space of compactly supported
continuous functions, denoted as, say $\mathrm{d}\nu,$ or as $\mathsf{g}%
=\mathsf{g}_{\nu}.$ Observe that any signed measure can be written as $\nu
=\nu_{+}-\nu_{-},$ where $\nu_{\pm}$ are measures concentrated on disjoint
sets. We shall also use the Lebesgue decomposition, according to which any
signed measure $\nu$ can be written as $\nu=\nu_{\mathrm{abs}}+\nu
_{\mathrm{sig}},$ where $\nu_{\mathrm{abs}}$ is absolutely continuous with
respect to the Lebesgue measure, so that it corresponds to a regular
distribution, while $\nu_{\mathrm{sig}}$ is a signed measure concentrated on a
set of Lebesgue measure zero. We shall also need to consider the measures
$(\nu_{\mathrm{sig}})_{\pm}=(\nu_{\pm})_{\mathrm{sig}},$ the positive and
negative singular parts of $\nu.$

\section{The $\phi-$transform}\label{Sc The phi transform}

A very important tool in our definition of a general distributional integral
is the $\phi-$trans\-form. The $\phi-$transform \cite{DroZav03, VPwavelet,VEStudia,
VEmeasures} in one variable is defined as follows. Let $\phi
\in\mathcal{D}\left(  \mathbb{R}\right)  $ be a fixed \emph{normalized} test
function, that is, one that satisfies%
\begin{equation}
\int_{\mathbb{-\infty}}^{\infty}\phi\left(  x\right)  \,\mathrm{d}x=1\,.
\label{fi.1}%
\end{equation}
If $\mathsf{f}\in\mathcal{D}^{\prime}\left(  \mathbb{R}\right)  $ we introduce
the function of two variables $F=F_{\phi}\left\{  \mathsf{f}\right\}  $ by the
formula%
\begin{equation}
F\left(  x,t\right)  =\left\langle \mathsf{f}\left(  x+ty\right)  ,\phi\left(
y\right)  \right\rangle \,, \label{fi.2}%
\end{equation}
where $\left(  x,t\right)  \in\mathbb{H},$ the half plane $\mathbb{R}%
\times\left(  0,\infty\right)  .$ Naturally the evaluation in (\ref{fi.2}) is
with respect to the variable $y.$ We call $F$ the $\phi-$transform of
$\mathsf{f}.$ Whenever we consider $\phi-$trans\-forms we assume that $\phi$
satisfies (\ref{fi.1}).

The $\phi-$transform converges to the distribution as $t\rightarrow0^{+}$\cite{VEmeasures,VEEdin}: If
$\phi\in\mathcal{D}\left(  \mathbb{R}\right)  $ and $\mathsf{f}\in
\mathcal{D}^{\prime}\left(  \mathbb{R}\right)  ,$ then%
\begin{equation}
\lim_{t\rightarrow0^{+}}F\left(  x,t\right)  =\mathsf{f}\left(  x\right)  \,,
\label{fi.3}%
\end{equation}
distributionally in the space $\mathcal{D}^{\prime}\left(  \mathbb{R}\right)
,$ that is, if $\rho\in\mathcal{D}\left(  \mathbb{R}\right)  $ then%
\begin{equation}
\lim_{t\rightarrow0^{+}}\left\langle F\left(  x,t\right)  ,\rho\left(
x\right)  \right\rangle =\left\langle \mathsf{f}\left(  x\right)  ,\rho\left(
x\right)  \right\rangle \,. \label{fi.4}%
\end{equation}

The definition of the $\phi-$transform tells us that if the distributional
point value \cite{loj} $\mathsf{f}\left(  x_{0}\right)  $ exists and equals
$\gamma$ then $F\left(  x_{0},t\right)  \rightarrow\gamma$ as $t\rightarrow
0^{+},$ but actually $F\left(  x,t\right)  \rightarrow\gamma$ as $\left(
x,t\right)  \rightarrow\left(  x_{0},0\right)  $ in an \emph{angular} or
\emph{non-tangential} fashion, that is if $\left\vert x-x_{0}\right\vert \leq
Mt$ for some $M>0$ (just replace $\phi(y)$ by the compact set $\left\{
\phi\left(  y+ \tau\right)  :\: \left|  \tau\right|  \leq M \right\}  $).

The angular behavior of the $\phi-$transform at a point $\left(
x_{0},0\right)  $ gives us important information \cite{DroZav03,
VPwavelet, VEmeasures} about the nature of the distribution at $x=x_{0},$ even if the
angular limit does not exist.

If $x_{0}\in\mathbb{R}$ we shall denote by $C_{x_{0},\theta}$ the cone in
$\mathbb{H}$ starting at $x_{0}$ of angle $\theta,$
\begin{equation}
C_{x_{0},\theta}=\left\{  (x,t)\in\mathbb{H}:\left\vert x-x_{0}\right\vert
\leq(\tan\theta)t\right\}  \,. \label{M.2}%
\end{equation}
If $\mathsf{f}\in\mathcal{D}^{\prime}\left(  \mathbb{R}\right)  $ and
$x_{0}\in\mathbb{R}$ then we consider the upper and lower angular values of
its $\phi-$transform,%
\begin{equation}
\mathsf{f}_{\phi,\theta}^{+}\left(  x_{0}\right)  =\limsup
_{_{\substack{\left(  x,t\right)  \rightarrow\left(  x_{0},0\right)  \\\left(
x,t\right)  \in C_{x_{0},\theta}}}}F\left(  x,t\right)  \,, \label{M.3}%
\end{equation}%
\begin{equation}
\mathsf{f}_{\phi,\theta}^{-}\left(  x_{0}\right)  =\liminf
_{_{\substack{\left(  x,t\right)  \rightarrow\left(  x_{0},0\right)  \\\left(
x,t\right)  \in C_{x_{0},\theta}}}}F\left(  x,t\right)  \,. \label{M.4}%
\end{equation}
The quantities $\mathsf{f}_{\phi,\theta}^{\pm}\left(  x_{0}\right)  $ are well
defined at all points $x_{0},$ but, of course, they could be infinite. For $\theta=0$, we obtain the upper and lower \emph{radial} limits of the $\phi-$transform.

The following simple result would be useful.\smallskip

\begin{lemma}
\label{Lemma M.1}Let $\mathsf{f}\in\mathcal{D}^{\prime}\left(  \mathbb{R}%
\right)  $ and $x_{0}\in\mathbb{R}.$ If%
\begin{equation}
\mathsf{f}_{\phi,0}^{+}\left(  x_{0}\right)  =\mathsf{f}_{\phi,0
}^{-}\left(  x_{0}\right)  =\gamma\,, \label{M.4p1}%
\end{equation}
for all normalized positive test functions $\phi\in\mathcal{D}(\mathbb{R})$, then the distributional
point value $\mathsf{f}\left(  x_{0}\right)  $ exists and equals $\gamma.$
\end{lemma}

\begin{proof}
Indeed, (\ref{M.4p1}) yields that $\lim_{\varepsilon\rightarrow0}\left\langle
\mathsf{f}\left(  x_{0}+\varepsilon x\right)  ,\phi\left(  x\right)
\right\rangle $ exists and equals $\gamma$ for any positive normalized test
function. If we multiply by a constant, we obtain that the limit exists and
equals $\gamma\int_{-\infty}^{\infty}\phi\left(  x\right)  \,\mathrm{d}x$ for
any positive test function. The result now follows because any test function
is the difference of two positive test functions. Indeed, given an arbitrary test function $\phi\in\mathcal{D}(\mathbb{R})$, let $M=\max_{x\in\mathbb{R}}\left|\phi(x)\right|$. Find a positive $\varphi\in\mathcal{D}(\mathbb{R})$ so that $\varphi(x)=1$ for $x\in\operatorname*{supp}\:\phi$, then $\phi_{1}=M\varphi$ and $\phi_{2}=\phi+\phi_{1}$ are positive test functions with $\phi=\phi_{2}-\phi_1$.\smallskip
\end{proof}

We shall need several characterizations of positive measures in terms of the
extreme values $\mathsf{f}_{\phi,\theta}^{\pm}\left(  x\right)  $ of a
distribution $\mathsf{f}.$ The following result was proved in
\cite{VEmeasures}.\smallskip

\begin{theorem}
\label{Theorem M.2}Let $\mathsf{f}\in\mathcal{D}^{\prime}\left(
\mathbb{R}\right)  .$ Let $U$ be an open set. Then $\mathsf{f}$ is a measure
in $U$ if and only if its $\phi-$transform $F=F_{\phi}\left\{  \mathsf{f}%
\right\}  $ with respect to a given normalized, positive test function
$\phi\in\mathcal{D}\left(  \mathbb{R}\right)  $\ satisfies%
\begin{equation}
\mathsf{f}_{\phi,\theta}^{-}\left(  x\right)  \geq0\ \ \ \ \ \forall x\in U\,,
\label{M.5}%
\end{equation}
for all angles $\theta.$ Moreover, if the support of $\phi$ is contained in
$\left[  -R,R\right]  $ and if (\ref{M.5}) holds for a single value of
$\theta>\arctan R,$ then $\mathsf{f}$ is a measure in $U.$\smallskip
\end{theorem}

We should also point out that if there exists a constant $M>0$ such that
$\mathsf{f}_{\phi,\theta}^{-}\left(  x\right)  \geq-M,$ $\forall x\in U,$
where $\theta>\arctan R,$ then $\mathsf{f}$ is a signed measure in $U,$ whose
singular part is \emph{positive} \cite{VEmeasures}. It is easy to see that these results are not true if we use \emph{radial} limits instead of angular ones. An example is provided by taking
$\mathsf{f}\left(  x\right)  =-\delta^{\prime}\left(  x\right)  $ and $\phi
\in\mathcal{D}\left(  \mathbb{R}\right)  $ with $\phi^{\prime}\left(
0\right)  >0.$ Actually this example shows that if (\ref{M.5}) holds for a
value of $\theta<\arctan R,$ then $\mathsf{f}$ might not be a measure. 

Using the Romanovski's
lemma we were able to prove the ensuing stronger result in \cite{EVraex}%
.\smallskip

\begin{theorem}
\label{TheoremMeasures}Let $\mathsf{f}\in\mathcal{D}^{\prime}\left(
\mathbb{R}\right)  .$ Let $U$ be an open set. Suppose its $\phi-$transform
$F=F_{\phi}\left\{  \mathsf{f}\right\}  $ with respect to a given normalized,
positive test function $\phi\in\mathcal{D}\left(  \mathbb{R}\right)  $\ with
$\operatorname*{supp}\phi\subset\left[  -R,R\right]  $ satisfies
\begin{equation}
\mathsf{f}_{\phi,0}^{+}\left(  x\right)  \geq0\ \ \ \ \ \text{almost
everywhere in }U\,, \label{Mea.1}%
\end{equation}
while for each $x\in U$ there is a constant $M_{x}>0$
such that
\begin{equation}
\mathsf{f}_{\phi,\theta}^{-}\left(  x\right)  \geq-M_{x}\,, \label{Mea.2}%
\end{equation}
where $\theta>\arctan R.$ Then $\mathsf{f}$ is a measure in $U.$\smallskip
\end{theorem}

Furthermore, one needs the inequality (\ref{Mea.2}) to be true at \emph{all
}points of $U,$ as the example $\mathsf{f}\left(  x\right)  =-\delta\left(
x-a\right)  ,$ where $a\in U,$ shows. However, in our construction of the
general distributional integral we shall need to consider the case when
$\mathsf{f}_{\phi,\theta}^{-}\left(  x\right)  =-\infty$ for $x\in E$ where
$E$ is a small set in the sense that $\left\vert E\right\vert \leq\aleph_{0}.$
We have a corresponding result in this case if we ask that any primitive of
$\mathsf{f}$ be a \L ojasiewicz distribution.\smallskip

\begin{theorem}
\label{Theorem M3}Let $\mathsf{f}\in\mathcal{D}^{\prime}\left(  \mathbb{R}%
\right)  .$ Suppose that $\mathsf{f=F}^{\prime},$ where $\mathsf{F}$ is a
\L ojasiewicz distribution. Let $U$ be an open set. Suppose the $\phi
-$transform $F=F_{\phi}\left\{  \mathsf{f}\right\}  $ with respect to a given
normalized, positive test function $\phi\in\mathcal{D}\left(  \mathbb{R}%
\right)  $\ with $\operatorname*{supp}\phi\subset\left[  -R,R\right]  $
satisfies
\begin{equation}
\mathsf{f}_{\phi,0}^{+}\left(  x\right)  \geq0\ \ \ \ \ \text{almost
everywhere in }U\,, \label{Mea.1p}%
\end{equation}
while there exist a countable set $E$ such that for each
$x\in U\setminus E$ there is a constant $M_{x}>0$ such that
\begin{equation}
\mathsf{f}_{\phi,\theta}^{-}\left(  x\right)  \geq-M_{x}\,,\ \ \ \ \ x\in
U\setminus E\,, \label{Mea.2p}%
\end{equation}
where $\theta>\arctan R.$ Then $\mathsf{f}$ is a measure in $U.$
\end{theorem}

\begin{proof}
Suppose that $U$ is an open interval. Let $\mathfrak{U}$ be the family of open
subintervals $V$ of $U$ such that the restriction $\displaystyle\left.
\mathsf{f}\right\vert _{V}$ is a measure. We shall use the Theorem
\ref{TheoremRomanovskiLemma} to prove that $U\in\mathfrak{U}.$ Let us first
show that $\mathfrak{U}\neq\left\{  \emptyset\right\}  .$ Suppose that
$E\subseteq\left\{  x_{n}:1\leq n<\infty\right\}  .$ Let $t_{0}\geq1$ be fixed
and put
\begin{equation}
g_{n}\left(  x\right)  =\min\left\{  F\left(  y,t\right)  :\left\vert
y-x\right\vert \leq(\tan\theta)t,~n^{-1}\leq t\leq t_{0}\right\}  \,.
\label{Mea.3}%
\end{equation}
The functions $g_{n}$ are continuous and because of (\ref{Mea.2p}), for each
$x\in U\setminus E$ there exists a constant $M_{x}^{^{\prime}}>0$ such that
$g_{n}\left(  x\right)  \geq-M_{x}^{^{\prime}},$ for all $n$. Hence if%
\begin{equation}
W_{k}=\left\{  x\in U:g_{n}\left(  x\right)  \geq-k\ \ \forall n\in
\mathbb{N}\right\}  \cup\left\{  x_{1},\ldots,x_{k}\right\}  \,, \label{Mea.4}%
\end{equation}
then $U=\bigcup_{k=1}^{\infty}W_{k}.$ If we now employ the Baire theorem we
obtain the existence of $k\in\mathbb{N},$ such that $W_{k}$ has non-empty
interior, and thus the interior of the set 
\begin{equation}\left\{  x\in U:g_{n}\left(  x\right)
\geq-k\ \ \forall n\in\mathbb{N}\right\} 
\end{equation}
is also non-empty. Hence there is
a non-empty open interval $V\subset U$ and a constant $M>0$ such that
$F\left(  x,t\right)  \geq-M$ for all $\left(  y,t\right)  \in C_{x,\theta}$
with $x\in V$ and $0<t\leq t_{0},$ and hence $\mathsf{f}_{\phi,\theta}%
^{-}\left(  x\right)  \geq-M$ for $x\in V.$ The Theorem \ref{TheoremMeasures}
then yields that $\displaystyle\left.  \mathsf{f}\right\vert _{V}$ is a
measure. Therefore $V\in\mathfrak{U},$ and so $\mathfrak{U}\neq\left\{
\emptyset\right\}  .$

Condition I of the Theorem \ref{TheoremRomanovskiLemma} follows from the fact
that if $\displaystyle\left.  \mathsf{f}\right\vert _{\left(  \alpha
,\beta\right)  }$ and $\displaystyle\left.  \mathsf{f}\right\vert _{\left(
\beta,\gamma\right)  }$ are measures, then $\displaystyle\left.
\mathsf{F}\right\vert _{\left(  \alpha,\beta\right)  }$ and
$\displaystyle\left.  \mathsf{F}\right\vert _{\left(  \beta,\gamma\right)  }$
are distributions corresponding to increasing continuous functions, and since
$\mathsf{F}$ is a \L ojasiewicz distribution it follows that $F,$
$F\leftrightarrow\mathsf{F,}$ must also be continuous at $x=\beta,$ so that
$F$ is a continuous increasing function in $\left(  \alpha,\gamma\right)  $
and consequently $\displaystyle\left.  \mathsf{f}\right\vert _{\left(
\alpha,\gamma\right)  }$ is a measure.

It is clear that II and III are satisfied.

In order to prove IV$,$ let $K\subset\overline{U}$ be a perfect closed set
such that all the intervals contiguous to $K$ belong to $\mathfrak{U}.$ Then
using the Baire theorem again, there exists an open interval $V\subset U$ and
a constant $M>0$ such that $\mathsf{f}_{\phi,\theta}^{-}\left(  x\right)
\geq-M$ for all $x\in K\cap V\neq\emptyset.$ But $\mathsf{f}$ is a measure in $V\setminus
K,$ and thus $\mathsf{f}_{\phi,\theta}^{-}\left(  x\right)  \geq0$ for $x\in
V\setminus K.$ The Theorem \ref{TheoremMeasures} allows us to conclude that
$\displaystyle\left.  \mathsf{f}\right\vert _{V}$ is a measure, and thus
$V\in\mathfrak{U};$ this proves IV.\smallskip
\end{proof}

Observe that if the hypotheses of the Theorem \ref{Theorem M3} are satisfied
then $\mathsf{f}$ is a measure in $U,$ and thus $\mathsf{f}_{\phi,\theta
}^{-}\left(  x\right)  \geq0$ at all points of $U$ and for all angles, not just radially almost
everywhere, and similarly the set $E$ where $\mathsf{f}_{\phi,\theta}%
^{-}\left(  x\right)  =-\infty$ is actually empty.

We shall also employ characterizations merely in terms of \emph{radial} limits of the
$\phi-$trans\-form. The following is such a result for the lower radial limits
of a harmonic function.\smallskip

\begin{theorem}
\label{Theorem M3p}Let $H\left(  x,y\right)  $ be a harmonic function defined
in the upper half plane $\mathbb{H}.$ Suppose that $\lim_{\left(  x,y\right)
\rightarrow\infty}H\left(  x,y\right)  =0$. Also suppose that the
distributional limit of $H\left(  x,y\right)  $ as $y\rightarrow0^{+}$ exists
and equals $\mathsf{f}\in\mathcal{E}^{\prime}\left(  \mathbb{R}\right)  ;$
suppose that $\mathsf{f=F}^{\prime},$ where $\mathsf{F}$ is a \L ojasiewicz
distribution. If%
\begin{equation}
\limsup_{y\rightarrow0^{+}}H\left(  x,y\right)  \geq0\ \ \ \ \ \text{almost
everywhere in }\mathbb{R}\,, \label{Mea.3p1}%
\end{equation}
and there exists a countable set $E$ and constants $M_{x}<\infty$ for
$x\in\mathbb{R}\setminus E$ such that%
\begin{equation}
\liminf_{y\rightarrow0^{+}}H\left(  x,y\right)  \geq-M_{x}\,,\ \ \ \ \ x\in
\mathbb{R}\setminus E\,, \label{Mea.3p2}%
\end{equation}
then $\mathsf{f}$ is a measure and $H\left(  x,y\right)  \geq0$ for all
$\left(  x,y\right)  \in\mathbb{H}.$
\end{theorem}

\begin{proof}
We shall employ Romanovski's lemma, Theorem \ref{TheoremRomanovskiLemma}
to prove that $\mathsf{f}$ is a measure in $\mathbb{R}.$ Let $\left(
a,b\right)  $ be an open interval with $\operatorname*{supp}\mathsf{f}%
\subset\left(  a,b\right)  .$ Let $\mathfrak{U}$ be the family of open
subintervals of $\left(  a,b\right)  $ where the restriction of $\mathsf{f}$
is a measure; clearly $\mathfrak{U}$ contains non empty intervals. Observe
that if $\left(  c,d\right)  \in\mathfrak{U},$ then $F$ is an increasing
continuous function in $\left[  c,d\right]  ,$ where $F\leftrightarrow
\mathsf{F;}$ condition I follows from this observation. Conditions II and III
are easy. For condition IV, suppose that $K$ is a perfect compact subset of
$\left(  a,b\right)  $ such that $\left(  a,b\right)  \setminus K=\bigcup
_{n=1}^{\infty}\left(  a_{n},b_{n}\right)  ,$ with $\left(  a_{n}%
,b_{n}\right)  \in\mathfrak{U}.$ Let $m=\min_{x\in\mathbb{R}}H\left(
x,1\right)  .$ By the Baire theorem, there exists a constant $M,$ with $M>0$ and $M>-m,$ and an open
interval $I,$ such that $I\cap K\neq\emptyset$ and $H\left(
x,y\right)  \geq-M$ for $x\in\overline{I}\cap K$ and for $0<y\leq1.$ If
$\left(  a_{n},b_{n}\right)  \subset I,$ then the harmonic function $H$ is
bounded below by $-M$ in the boundary of the rectangle $\left(  a_{n}%
,b_{n}\right)  \times\left(  0,1\right)  \subset\mathbb{H},$ except perhaps at
the corners $a_{n}$ and $b_{n},$ but since $\mathsf{f}$ is the derivative of a
\L ojasiewicz distribution we obtain the bound $H\left(  x,y\right)  =o\left(
\left(  \left(  x-x_{0}\right)  ^{2}+y^{2}\right)  ^{-1/2}\right)  $ as
$\left(  x,y\right)  \rightarrow x_{0},$ for any $x_{0}\in\mathbb{R},$\ and
this allows to conclude that $H$ is bounded below by $-M$ in the rectangle
$\left[  a_{n},b_{n}\right]  \times(0,1].$ Actually if $H$ were not bounded
below in the rectangle then at one of the corners, $x_{0}=a_{n}$ or
$x_{0}=b_{n},$ $H$ would grow as fast as or faster than $\left(  \left(
x-x_{0}\right)  ^{2}+y^{2}\right)  ^{-1},$ as follows from the results of \cite[Section 4]{est3} when applied to the harmonic function $\tilde{H}(\xi)=H(\sqrt{\xi-x_{0}})$. Therefore $H\left(  x,y\right)
\geq-M$ for all $x\in\overline{I}$ and all $0<y\leq1,$ and the fact that
$I\in\mathfrak{U}$ is obtained.\smallskip
\end{proof}

It is convenient to define some classes of test functions.\smallskip

\begin{definition}
\label{Def M1}The class $\mathcal{T}_{0}$ consists of all positive normalized
functions $\phi\in\mathcal{E}\left(  \mathbb{R}\right)  $ that satisfy the
following condition:%
\begin{equation}
\exists\alpha<-1~~\text{such that}~~~~~\phi\left(  x\right)  =O\left(
\left\vert x\right\vert ^{\alpha}\right)  \ \ \ \ \text{strongly as
}\left\vert x\right\vert \rightarrow\infty\,. \label{Mea.4p}%
\end{equation}
The class $\mathcal{T}_{1}$\ is the subclass of $\mathcal{T}_{0}$\ consisting
of those functions that also satisfy%
\begin{equation}
x\phi^{\prime}\left(  x\right)  \leq0\ \ \ \ \text{for all}\ \ \ \ x\in
\mathbb{R}\,. \label{Mea.4p2}%
\end{equation}
\smallskip
\end{definition}

Observe that the $\phi-$transform is well defined when $f\in\mathcal{E}'(\mathbb{R})$ and $\phi\in\mathcal{T}_{0}$. Since the Poisson kernel $\varphi\left(  x\right)  =\pi^{-1}\left(  1+x^{2}\right)  ^{-1}$
belongs to $\mathcal{T}_{1}$ and the $\phi-$transform $H=F_{\varphi}\left\{
\mathsf{f}\right\}  $ with respect to this function $\varphi$ is the harmonic
function $H\left(  x,y\right)  $ defined for $(x,y)\in\mathbb{H},$ that
vanishes at infinity, and that satisfies $H\left(  x,0^{+}\right)
=\mathsf{f}\left(  x\right)  $ distributionally, we then have the following
result, corollary of the Theorem \ref{Theorem M3p}.\smallskip

\begin{theorem}
\label{Theorem M4}Let $\mathsf{f}\in\mathcal{E}^{\prime}\left(  \mathbb{R}%
\right)  .$ Suppose that $\mathsf{f=F}^{\prime},$ where $\mathsf{F}$ is a
\L ojasiewicz distribution. Suppose that the $\phi-$transform $F=F_{\phi
}\left\{  \mathsf{f}\right\}  $ with respect any $\phi\in\mathcal{T}_{1}$
satisfies%
\begin{equation}
\mathsf{f}_{\phi,0}^{+}\left(  x\right)  \geq0\ \ \ \ \ \text{almost
everywhere in }\mathbb{R}\,, \label{Mea.5e}%
\end{equation}%
\begin{equation}
\mathsf{f}_{\phi,0}^{-}\left(  x\right)  \geq-M_{x}>-\infty\,,\ \ \ \ \ x\in
\mathbb{R}\setminus E\,, \label{Mea.5}%
\end{equation}
where $E$ is a countable set. Then $\mathsf{f}$ is a measure in $\mathbb{R}.$
\end{theorem}

\section{The definite integral}\label{Sc Integral}

Let $f$ be a function defined in $\left[  a,b\right]  $ with values in
$\overline{\mathbb{R}}=\mathbb{R\cup}\left\{  -\infty,\infty\right\}  .$ We
now proceed to define its integral. We start with the concepts of major and
minor pairs.\smallskip

\begin{definition}
\label{Def IN1}A pair $\left(  \mathsf{u},\mathsf{U}\right)  $ is called a
\emph{major distributional pair} for the function $f$ if:

1) $\mathsf{u}\in\mathcal{E}^{\prime}\left[  a,b\right]  ,$ $\mathsf{U}%
\in\mathcal{D}^{\prime}\left(  \mathbb{R}\right)  ,$ and
\begin{equation}
\mathsf{U}^{\prime}=\mathsf{u}\,. \label{IN.1}%
\end{equation}

2) $\mathsf{U}$ is a \L ojasiewicz distribution, with $\mathsf{U}\left(
a\right)  =0.$

3) There exists a set $E,$ with $\left\vert E\right\vert \leq\aleph_{0},$ and
a set of null Lebesgue measure $Z,$ $m\left(  Z\right)  =0,$
such that for all $x\in\left[  a,b\right]  \setminus Z$ and all $\phi
\in\mathcal{T}_{0}$ we have
\begin{equation}
\left(  \mathsf{u}\right)  _{\phi,0}^{-}\left(  x\right)  \geq f\left(
x\right)  \,, \label{IN.2}%
\end{equation}
while for $x\in\left[  a,b\right]  \setminus E$ and all $\phi\in
\mathcal{T}_{1}$\
\begin{equation}
\left(  \mathsf{u}\right)  _{\phi,0}^{-}\left(  x\right)  >-\infty\,.
\label{IN.3}%
\end{equation}
\smallskip
\end{definition}

The definition of a minor distributional pair is similar.\smallskip

\begin{definition}
\label{Def IN2}A pair $\left(  \mathsf{v},\mathsf{V}\right)  $ is called a
\emph{minor distributional pair} for the function $f$ if:

1) $\mathsf{v}\in\mathcal{E}^{\prime}\left[  a,b\right]  ,$ $\mathsf{V}%
\in\mathcal{D}^{\prime}\left(  \mathbb{R}\right)  ,$ and
\begin{equation}
\mathsf{V}^{\prime}=\mathsf{v}\,. \label{IN.4}%
\end{equation}

2) $\mathsf{V}$ is a \L ojasiewicz distribution, with $\mathsf{V}\left(
a\right)  =0.$

3) There exists a set $E,$ with $\left\vert E\right\vert \leq\aleph_{0},$ and
a set of null Lebesgue measure $Z,$ $m\left(  Z\right)  =0,$
such that for all $x\in\left[  a,b\right]  \setminus Z$ and all $\phi
\in\mathcal{T}_{0}$ we have
\begin{equation}
\left(  \mathsf{v}\right)  _{\phi,0}^{+}\left(  x\right)  \leq f\left(
x\right)  \,, \label{IN.5}%
\end{equation}
while for $x\in\left[  a,b\right]  \setminus E$ and all $\phi\in
\mathcal{T}_{1}$
\begin{equation}
\left(  \mathsf{v}\right)  _{\phi,0}^{+}\left(  x\right)  <\infty\,.
\label{IN.6}%
\end{equation}
\smallskip
\end{definition}
Naturally, we may always assume in the Definitions \ref{Def IN1} and \ref{Def IN2} that the countable set satisfies $E\subset Z.$

Employing the results of the Theorem \ref{Theorem M4}, we immediately obtain
the following useful result.\smallskip

\begin{lemma}
\label{Lemma IN1}If $\left(  \mathsf{u},\mathsf{U}\right)  $ is a major
distributional pair and $\left(  \mathsf{v},\mathsf{V}\right)  $ is a minor
distributional pair for $f,$ then $\mathsf{u}-\mathsf{v}$ is a positive
measure and $U-V$ is a continuous increasing function, where $U\leftrightarrow
\mathsf{U}$ and $V\leftrightarrow\mathsf{V}$.\smallskip
\end{lemma}

If $\left(  \mathsf{u},\mathsf{U}\right)  $ is a major distributional pair and
$\left(  \mathsf{v},\mathsf{V}\right)  $ is a minor distributional pair for
$f,$ then $\mathsf{U}$ and $\mathsf{V}$ are constant in the interval
$[b,\infty),$ and $\mathsf{V}\left(  b\right)  \leq\mathsf{U}\left(  b\right)
.$\smallskip

\begin{definition}
\label{Def IN3}A function $f:\left[  a,b\right]  \rightarrow\overline
{\mathbb{R}}$ is called distributionally integrable if it has both major and
minor distributional pairs and if%
\begin{equation}
\sup_{\left(  \mathsf{v},\mathsf{V}\right)  \text{ minor pair}}\mathsf{V}%
\left(  b\right)  =\inf_{\left(  \mathsf{u},\mathsf{U}\right)  \text{ major
pair}}\mathsf{U}\left(  b\right)  \,. \label{IN.7}%
\end{equation}
When this is the case this common value is the integral of $f$ over $\left[
a,b\right]  $ and is denoted as
\begin{equation}
\left(  \mathfrak{dist}\right)  \int_{a}^{b}f\left(  x\right)  \,\mathrm{d}%
x\,, \label{IN.8}%
\end{equation}
or just as $\int_{a}^{b}f\left(  x\right)  \,\mathrm{d}x$ if there is no risk
of confusion.\smallskip
\end{definition}

We shall show in Section \ref{Section Comparison} that any Lebesgue integrable function, and more generally, any
Denjoy-Perron-Henstock-Kurzweil integrable function is distributionally integrable, and
the integrals are the same. Therefore the symbol $\int_{a}^{b}f\left(
x\right)  \,\mathrm{d}x$ will have only one possible meaning if the function
$f$ is Denjoy-Perron-Henstock-Kurzweil integrable or Lebesgue integrable. In some cases
we shall use the notation $\left(  \mathfrak{dist}\right)  \int_{a}%
^{b}f\left(  x\right)  \,\mathrm{d}x,$ however, to emphasize that we are
dealing with the integral defined in this article. Occasionally, we shall also
use the notation $\left(  \mathfrak{DPHK}\right)  \int_{a}^{b}f\left(
x\right)  \,\mathrm{d}x$ for a Denjoy-Perron-Henstock-Kurzweil integral and $\left(
\mathfrak{Leb}\right)  \int_{a}^{b}f\left(  x\right)  \,\mathrm{d}x$ for a
Lebesgue integral.

Observe that the function $f$ is distributionally integrable over $\left[
a,b\right]  $ if and only if for each $\varepsilon>0$ there are minor and
major pairs, $\left(  \mathsf{v},\mathsf{V}\right)  $ and $\left(
\mathsf{u},\mathsf{U}\right)  $, such that
\begin{equation}
\mathsf{U}\left(  b\right)  -\mathsf{V}\left(  b\right)  <\varepsilon\,.
\label{IN.9}%
\end{equation}

We shall first show that the distributional integral has the standard
properties of an integral.\smallskip

\begin{proposition}
\label{Prop. IN 1}If $f$ is distributionally integrable over $\left[
a,b\right]  $ then it is distributionally integrable over any subinterval
$\left[  c,d\right]  \subset\left[  a,b\right]  .$
\end{proposition}

\begin{proof}
Let $\varepsilon>0,$ and choose minor and major pairs for $f$ over $\left[
a,b\right]  $, $\left(  \mathsf{v},\mathsf{V}\right)  $ and $\left(
\mathsf{u},\mathsf{U}\right)  $, such that $\mathsf{U}\left(  b\right)
-\mathsf{V}\left(  b\right)  <\varepsilon.$ Let $U\leftrightarrow\mathsf{U}$
and $V\leftrightarrow\mathsf{V}.$ Let now $\widetilde{\mathsf{U}}$ and
$\widetilde{\mathsf{V}}$ be the \L ojasiewicz distributions corresponding to
the \L ojasiewicz functions $\widetilde{U}$ and $\widetilde{V}$ given by%
\begin{equation}
\widetilde{U}\left(  x\right)  =\left\{
\begin{array}
[c]{cc}%
0 & x<c\,,\\
U\left(  x\right)  -U\left(  c\right)  \,, & c\leq x\leq d\,,\\
U\left(  d\right)  -U\left(  c\right)  \,, & x>d\,,
\end{array}
\right.  \label{IN.10}%
\end{equation}
and
\begin{equation}
\widetilde{V}\left(  x\right)  =\left\{
\begin{array}
[c]{cc}%
0 & x<c\,,\\
V\left(  x\right)  -V\left(  c\right)  \,, & c\leq x\leq d\,,\\
V\left(  d\right)  -V\left(  c\right)  \,, & x>d\,.
\end{array}
\right.  \label{IN.11}%
\end{equation}
Then $\left(  \widetilde{\mathsf{V}}^{\prime},\widetilde{\mathsf{V}}\right)  $
and $\left(  \widetilde{\mathsf{U}}^{\prime},\widetilde{\mathsf{U}}\right)  $
are minor and major pairs for $f$ over $\left[  c,d\right]  ,$ and
$\widetilde{\mathsf{U}}\left(  d\right)  -\widetilde{\mathsf{V}}\left(
d\right)  <\varepsilon.$\smallskip
\end{proof}

We now consider the integrals of functions that are equal almost everywhere.
As it is the case with other integrals, the integral can actually be defined
as a functional on the space of equivalence classes of functions equal
$\left(  \text{a.e.}\right)  ,$ and each class has elements that are finite
everywhere.\smallskip

\begin{proposition}
\label{Prop. IN 5}If $f$ is distributionally integrable over $\left[
a,b\right]  $ then it is finite almost everywhere.
\end{proposition}

\begin{proof}
Let $A$ be the set of points where $\left\vert f\left(  x\right)  \right\vert
=\infty.$ Let $\left(  \mathsf{v},\mathsf{V}\right)  $ and $\left(
\mathsf{u},\mathsf{U}\right)  $ be minor and major pairs for $f$ over $\left[
a,b\right]  ,$ and let $E$ be the denumerable set outside of where $\left(
\mathsf{u}\right)  _{\phi,0}^{-}\left(  x\right)  >-\infty$ and $\left(
\mathsf{v}\right)  _{\phi,0}^{+}\left(  x\right)  <\infty\,,$ for all $\phi
\in\mathcal{T}_{1}.$ Consider the increasing continuous function $\rho\left(
x\right)  =\mathsf{U}\left(  x\right)  -\mathsf{V}\left(  x\right)  .$ Using
(\ref{IN.3}) and (\ref{IN.6}) we obtain that if $x\in A\setminus E$ then
$\rho^{\prime}\left(  x\right)  =\infty,$ but the set of points where the
derivative of an increasing continuous function is infinite has measure
$0.$\smallskip
\end{proof}

The ensuing result allows us to consider distributional integration of
functions that are defined almost everywhere.\smallskip

\begin{proposition}
\label{Prop. IN 6}If $f$ is distributionally integrable over $\left[
a,b\right]  $ and $g\left(  x\right)  =f\left(  x\right)  $ $\left(
\text{a.e.}\right)  $ then $g$ is also distributionally integrable over
$\left[  a,b\right]  $ and
\begin{equation}
\int_{a}^{b}g\left(  x\right)  \,\mathrm{d}x=\int_{a}^{b}f\left(  x\right)
\,\mathrm{d}x\,. \label{IN.15}%
\end{equation}

\end{proposition}

\begin{proof}
Indeed, any major or minor pair for $f$ is also a major or a minor pair for
$g,$ and conversely.\smallskip
\end{proof}

The integral has the expected linear properties.\smallskip

\begin{proposition}
\label{Prop. IN 2}If $f_{1}$ and $f_{2}$ are distributionally integrable over
$\left[  a,b\right]  $ then so is $f_{1}+f_{2}$ and
\begin{equation}
\int_{a}^{b}\left(  f_{1}\left(  x\right)  +f_{2}\left(  x\right)  \right)
\,\mathrm{d}x=\int_{a}^{b}f_{1}\left(  x\right)  \,\mathrm{d}x+\int_{a}%
^{b}f_{2}\left(  x\right)  \,\mathrm{d}x\,. \label{IN.12}%
\end{equation}

\end{proposition}

\begin{proof}
Using Propositions \ref{Prop. IN 5} and \ref{Prop. IN 6} it follows that we
may assume that both $f_{1}$ and $f_{2}$ are finite everywhere, so that its
sum is also defined everywhere. Then we just observe that the sum of major
pairs for $f_{1}$ and $f_{2}$ is a major pair for $f_{1}+f_{2},$ and similarly
for the sum of minor pairs.\smallskip
\end{proof}

\begin{proposition}
\label{Prop. IN 3}If $f$ is distributionally integrable over $\left[
a,b\right]  $ then so is $kf$ for any constant $k$ and
\begin{equation}
\int_{a}^{b}kf\left(  x\right)  \,\mathrm{d}x=k\int_{a}^{b}f\left(  x\right)
\,\mathrm{d}x\,. \label{IN.13}%
\end{equation}

\end{proposition}

\begin{proof}
The result follows from the following observations. If $k>0$ then multiplying
a major pair for $f$ with $k$ gives a major pair for $kf,$ and similarly for
minor pairs. If $k<0$ then multiplication with $k$ transforms major pairs for
$f$ into minor pairs for $kf$ and minor pairs for $f$ into major pairs for
$kf.$\smallskip
\end{proof}

It follows from the previous results that the set of distributionally
integrable functions over $\left[  a,b\right]  $ is a linear space and that
the integral is a linear functional.

We also have the following easy result.\smallskip

\begin{proposition}
\label{Prop. IN 4}Suppose $a<c<b.$ A function $f$ defined in $\left[
a,b\right]  $ is distributionally integrable there if and only if it is
distributionally integrable over $\left[  a,c\right]  $ and $\left[
c,b\right]  ,$ and when this is the case,%
\begin{equation}
\int_{a}^{b}f\left(  x\right)  \,\mathrm{d}x=\int_{a}^{c}f\left(  x\right)
\,\mathrm{d}x+\int_{c}^{b}f\left(  x\right)  \,\mathrm{d}x\,. \label{IN.14}%
\end{equation}
\smallskip
\end{proposition}

If $A\subset\left[  a,b\right]  $ then we say that $f$ is distributionally
integrable over $A$ if $\chi_{A}f,$ where $\chi_{A}$ is the characteristic
function of $A,$ is distributionally integrable, and use the notation
\begin{equation}
\left(  \mathfrak{dist}\right)  \int_{A}f\left(  x\right)  \,\mathrm{d}x\,.
\label{IN.16}%
\end{equation}
As with any non-absolute integral, $f$ will not be integrable over all
measurable subsets of $\left[  a,b\right]  ,$ but if $A$ has measure $0$ the
distributional integral exists and equals $0.$ Also, according to Proposition
\ref{Prop. IN 1}, if $f$ is distributionally integrable over $\left[
a,b\right]  $ then it is integrable over any of its subintervals.\smallskip

\section{The indefinite integral}\label{Section: Indefinite integral}

We shall now study the indefinite integral function
\begin{equation}
F\left(  x\right)  =\left(  \mathfrak{dist}\right)  \int_{a}^{x}f\left(
t\right)  \,\mathrm{d}t\,, \label{PR.1}%
\end{equation}
of a function $f$ that is distributionally integrable over $\left[
a,b\right]  .$ We are interested in the case when $a\leq x\leq b,$ but
sometimes it would be convenient to extend the domain of $F$ by putting
$F\left(  x\right)  =0$ for $x<a$ and $F\left(  x\right)  =F\left(  b\right)
$ for $x>b.$

The indefinite integral of a Lebesgue integrable function is absolutely
continuous, while that of a Denjoy-Perron-Henstock-Kurzweil integrable function is
continuous. We shall show that (\ref{PR.1}) defines a \L ojasiewicz function,
with associated \L ojasiewicz distribution $\mathsf{F,}$ $F\leftrightarrow
\mathsf{F.}$ We shall also show that the derivative $\mathsf{f=F}^{\prime}$ is
a distribution that has \L ojasiewicz distributional point values almost
everywhere and actually $\mathsf{f}\left(  x\right)  =f\left(  x\right)  $
$\left(  \text{a.e.}\right)  .$

We start with some useful results.\smallskip

\begin{lemma}
\label{Lemma PR 1}Let $\left(  \mathsf{v},\mathsf{V}\right)  $ and $\left(
\mathsf{u},\mathsf{U}\right)  $ be minor and major pairs for a
distributionally integrable function $f$ over $\left[  a,b\right]  .$ Let
$U\leftrightarrow\mathsf{U}$ and $V\leftrightarrow\mathsf{V}.$ Then $U-F$ and
$F-V$ are both continuous increasing functions that vanish at $x=a.$
\end{lemma}

\begin{proof}
Observe that if $a\leq c<d\leq b$ then (\ref{IN.10}) gives a major pair
$\left(  \widetilde{\mathsf{U}}^{\prime},\widetilde{\mathsf{U}}\right)  $ for
$f$ over $\left[  c,d\right]  $ with $\widetilde{U}\left(  t\right)  =U\left(
t\right)  -U\left(  c\right)  $ for $c\leq x\leq d.$ Thus%
\[
F\left(  d\right)  -F\left(  c\right)  =\int_{c}^{d}f\left(  x\right)
\,\mathrm{d}x\leq\widetilde{U}\left(  d\right)  =U\left(  d\right)  -U\left(
c\right)  \,,
\]
and so%
\begin{equation}
U\left(  c\right)  -F\left(  c\right)  \leq U\left(  d\right)  -F\left(
d\right)  \,. \label{PR.2}%
\end{equation}
Similarly one shows that $F-V$ is increasing.

Observe now that $U-V=(U-F)+(F-V)$ is a continuous increasing function (Lemma
\ref{Lemma IN1}) written as the sum of two increasing functions: we conclude
that both $U-F$ and $F-V$ are continuous.\smallskip
\end{proof}

Using the lemma we see that $F=(F-V)+V$ is the sum of a continuous function
and a \L ojasiewicz function and thus it is a \L ojasiewicz
function.\smallskip

\begin{theorem}
\label{Theorem PR 1}Let $f$ be a distributionally integrable function over
$\left[  a,b\right]  ,$ with indefinite integral $F.$ Then $F$ is a
\L ojasiewicz function.\smallskip
\end{theorem}

Observe that one may consider $f$ as an equivalence class of functions defined
almost everywhere, and thus the value $f\left(  x\right)  $ for a particular
$x$ may or may not have a useful meaning. However, $F$ is a \L ojasiewicz
function, and this implies that the value $F\left(  x\right)  $ has a clear
interpretation for \emph{all} numbers $x.$

Since $F$ is a \L ojasiewicz function, it has an associated \L ojasiewicz
distribution $\mathsf{F.}$ The distributional derivative $\mathsf{f=F}%
^{\prime}$ is a well defined distribution with $\operatorname*{supp}%
\mathsf{f}\subset\left[  a,b\right]  .$ The relationship between $\mathsf{f}$
and $f$ is as follows.\smallskip

\begin{theorem}
\label{Theorem PR 2}Let $f$ be a distributionally integrable function over
$\left[  a,b\right]  ,$ with indefinite integral $F,$ let $F\leftrightarrow
\mathsf{F,}$ and let $\mathsf{f=F}^{\prime}.$ Then $\mathsf{f}$ has point
values almost everywhere and%
\begin{equation}
\mathsf{f}\left(  x\right)  =f\left(  x\right)  \ \ \ \ \ \ \left(
\text{a.e.}\right)  \,. \label{PR.3}%
\end{equation}

\end{theorem}

\begin{proof}
Let $\varepsilon,\eta>0.$ Let $\left(  \mathsf{u},\mathsf{U}\right)  $ be a
major pair for $f$ over $\left[  a,b\right]  $ with
\begin{equation}
U\left(  b\right)  -F\left(  b\right)  <\varepsilon\eta\,, \label{PR.4}%
\end{equation}
where $U\leftrightarrow\mathsf{U}.$ Let $\rho=U-F,$ an increasing continuous function.

Consider the set $A=\left\{  x\in\left[  a,b\right]  :\rho^{\prime}\left(
x\right)  \geq\varepsilon\right\}  .$ Since
\[
\varepsilon m\left(  A\right)  \leq\int_{a}^{b}\rho^{\prime}\left(  x\right)
\,\mathrm{d}x\leq\rho\left(  b\right)  <\varepsilon\eta\,,
\]
it follows that $m\left(  A\right)  <\eta,$ where $m\left(  A\right)  $ is the
Lebesgue measure.

Notice now that if $x\in\left[  a,b\right]  \setminus(A\cup Z),$ where $Z$ is
the null set outside of where $\left(  \mathsf{u}\right)  _{\phi,0}%
^{-}\left(  x\right)  \geq f\left(  x\right)  >-\infty,$ for all $\phi
\in\mathcal{T}_{0}$, then%
\[
\left(  \mathsf{f}\right)  _{\phi,0}^{-}\left(  x\right)  =\left(
\mathsf{u}\right)  _{\phi,0}^{-}\left(  x\right)  -\rho^{\prime}\left(
x\right)  >f\left(  x\right)  -\varepsilon\,.
\]
Hence%
\begin{equation}
m\left(  \left\{  x\in\left[  a,b\right]  :\left(  \mathsf{f}\right)
_{\phi,0}^{-}\left(  x\right)  \leq f\left(  x\right)  -\varepsilon\text{
}\forall\phi\in\mathcal{T}_{0}\right\}  \right)  <\eta\,.
\label{PR.5}%
\end{equation}
But $\eta$ is arbitrary, and thus the set where $\left(  \mathsf{f}\right)
_{\phi,0}^{-}\left(  x\right)  \leq f\left(  x\right)  -\varepsilon$ has
measure $0,$ and since $\varepsilon$ is also arbitrary we obtain that $\left(
\mathsf{f}\right)  _{\phi,0}^{-}\left(  x\right)  \geq f\left(  x\right)
$ $\left(  \text{a.e.}\right)  .$

Using a similar analysis involving minor pairs one likewise obtains that
$\left(  \mathsf{f}\right)  _{\phi,0}^{+}\left(  x\right)  \leq f\left(
x\right)  $ $\left(  \text{a.e.}\right)  .$ If we now use the Lemma
\ref{Lemma M.1} then (\ref{PR.3}) follows.\smallskip
\end{proof}

The following consequence of the preceding theorem is worth mentioning.\smallskip

\begin{corollary}
\label{Cor PR 1}
If $f$ is distributionally integrable over $\left[  a,b\right]  $ then it is measurable.
\end{corollary}

\begin{proof}
Let $\phi\in\mathcal{D}\left(  \mathbb{R}\right)  $ be a normalized test
function. Then the sequence of continuous functions
\begin{equation}
f_{n}\left(  x\right)  =\left\langle \mathsf{f}\left(  x+y/n),\phi\left(
y\right)  \right)  \right\rangle \,,
\end{equation}
converges to $f$ almost everywhere, namely where (\ref{PR.3}) holds, and the
measurability of $f$ is thus obtained.\smallskip
\end{proof}

If we now use Theorem \ref{Theorem PR 2}, combined with Lemma \ref{Lemma PR 1}%
, we obtain more information on the nature of major and minor pairs.\smallskip

\begin{proposition}
\label{Prop PR 1}Let $\left(  \mathsf{v},\mathsf{V}\right)  $ and $\left(
\mathsf{u},\mathsf{U}\right)  $ be minor and major pairs for a
distributionally integrable function $f$ over $\left[  a,b\right]  .$ Then the
distributional point values $\mathsf{v}\left(  t\right)  $ and $\mathsf{u}%
\left(  t\right)  $ exist almost everywhere in $\left[  a,b\right]  .$ If
$\widetilde{v}$ is a function given by the point values of $\mathsf{v},$
namely, $\widetilde{v}\left(  t\right)  =\mathsf{v}\left(  t\right)  $ when
the value exists, extended in any way to a function over $\left[  a,b\right]
,$ then $\widetilde{v}$ is distributionally integrable over $\left[
a,b\right]  .$ Similarly the function $\widetilde{u}\left(  t\right)
=\mathsf{u}\left(  t\right)  ,$ when the value exists, is distributionally
integrable over $[a,b].$ Furthermore,%
\begin{equation}
\mathsf{V}\left(  d\right)  -\mathsf{V}\left(  c\right)  \leq\int_{c}%
^{d}\widetilde{v}\left(  x\right)  \,\mathrm{d}x\leq\int_{c}^{d}f\left(
x\right)  \,\mathrm{d}x\,, \label{PR.6}%
\end{equation}
and%
\begin{equation}
\int_{c}^{d}f\left(  x\right)  \,\mathrm{d}x\leq\int_{c}^{d}\widetilde
{u}\left(  x\right)  \,\mathrm{d}x\leq\mathsf{U}\left(  d\right)
-\mathsf{U}\left(  c\right)  \,. \label{PR.6p}%
\end{equation}

\end{proposition}

\begin{proof}
Let $U\leftrightarrow\mathsf{U},$ $V\leftrightarrow\mathsf{V},$ and
$F\leftrightarrow\mathsf{F}.$ Since $F-V$ is an increasing continuous
function, it follows that $\left(  \mathsf{F}-\mathsf{V}\right)  ^{\prime
}=\mathsf{f}-\mathsf{v}$ is a positive measure, and thus it has distributional
values almost everywhere, and since $\mathsf{f}$ has a.e. distributional
values (equal to $f),$ it follows that likewise $\mathsf{v}$ has
distributional values a.e.. The function $\tilde{v}$ is distributionally
integrable because $\tilde{v}(t)=f(t)-h(t)$ $\left(  \text{a.e.}\right)  $,
where $h$ is the Lebesgue integrable function which corresponds to the
absolutely continuous part of $\mathsf{f-v}$ (see Theorem \ref{Theorem COM 1} below). The inequality (\ref{PR.6}) is
obtained from the fact that
\begin{equation}
0\leq\int_{c}^{d}(f\left(  x\right)  -\widetilde{v}\left(  x\right)
)\,\mathrm{d}x\leq\left(  \mathsf{F}\left(  d\right)  -\mathsf{V}\left(
d\right)  \right)  -\left(  \mathsf{F}\left(  c\right)  -\mathsf{V}\left(
c\right)  \right)  \,.
\end{equation}
The results for the major pair are obtained in a similar fashion.\smallskip
\end{proof}

This proposition suggests an alternative approach to the distributional
integral. Call a pair $\left(  \mathsf{u},\mathsf{U}\right)  $ a major pair
v.2 (version 2) if it satisfies all the conditions of the Definition
\ref{Def IN1} plus the extra requirement that $\mathsf{u}\left(  x\right)  $
exist almost everywhere in $\left[  a,b\right]  .$ Define, analogously, minor
pairs v.2 and an integral in terms of major and minor pairs v.2. Then this
integral would be identical to the distributional integral we have been
considering, because any major or minor pair in the original sense is actually
a pair in the v.2 sense. However, use of the definition v.2 allows one to
obtain some proofs, as that of Theorems \ref{Theorem PR 1} and
\ref{Theorem PR 2}, in a rather simple way.

Propositon \ref{Prop PR 1} also has the following consequence on the major and minor distributional pairs.\smallskip

\begin{corollary}
\label{Cor PR 2} Let $\left(  \mathsf{v},\mathsf{V}\right)  $ and $\left(
\mathsf{u},\mathsf{U}\right)  $ be minor and major pairs for a
distributionally integrable function $f$ over $\left[  a,b\right]  .$ Then, there exists a set of null Lebesgue measure $Z$
such that for all $x\in\left[  a,b\right]  \setminus Z$, all $\phi
\in\mathcal{T}_{0}$, and all angles we have
\begin{equation}
\left(  \mathsf{u}\right)  _{\phi,\theta}^{-}\left(  x\right)  \geq f\left(
x\right)  \,, \label{ang.1}%
\end{equation}
and
\begin{equation}
\left(  \mathsf{v}\right)  _{\phi,\theta}^{+}\left(  x\right)  \leq f\left(
x\right)  \,. \label{ang.2}%
\end{equation}
\end{corollary}
\begin{proof} Let $Z$ be the complement in $[a,b]$ of the set on which the distributional point values of $\mathsf{u}$, $\mathsf{v},$ and $\mathsf{f}$ exist. Then $Z$ has null Lebesgue measure and (\ref{ang.1}) and (\ref{ang.2}) are both valid on $[a,b]\setminus Z$.
\end{proof}
Corollary \ref{Cor PR 2} implicitly suggests a third variant yet for the definition of the distributional integral.
Let us say that $\left(  \mathsf{u},\mathsf{U}\right)  $ is a major pair
v.3 (version 3) if it satisfies the conditions of Definition
\ref{Def IN1} and additionally we replace the radial condition (\ref{IN.3}) by the stronger requirement (\ref{ang.1}), assumed to hold for all $x\in[a,b]\setminus Z$, $m(Z)=0$, all $\phi\in\mathcal{T}_{0}$, and all angles. Likewise, one defines minor pairs v.3. If we define an integral in terms of major and minor pairs v.3, then we obtain nothing new, because in view of Corollary \ref{Cor PR 2} this integral coincides with the distributional integral defined in Section \ref{Sc Integral}.

\section{Comparison with other integrals}\label{Section Comparison}

We shall now consider the relationship of the distributional integral to the
Lebesgue integral, to the Denjoy-Perron-Henstock-Kurzweil, and to the \L ojasiewicz
method (see (\ref{INT})). We also give a constructive solution to Denjoy's
problem on the reconstruction of functions from their higher order
differential quotients \cite{den}.

Let us start with the Lebesgue integration.\smallskip

\begin{theorem}
\label{Theorem COM 1}Any Lebesgue integrable function over $\left[
a,b\right]  $ is also distributionally integrable over $\left[  a,b\right]  $
and the integrals coincide.
\end{theorem}

\begin{proof}
Let $\varepsilon>0.$ If $f$ is a Lebesgue integrable function over $\left[
a,b\right]  $, we can apply the Vitali-Carath\'{e}odory Theorem \cite[III
(7.6)]{Sacks}\ to find a lower semi-continuous function $u$ with $u\left(
x\right)  \geq f\left(  x\right)  $ for all $x,$ and with%
\begin{equation}
\left(  \mathfrak{Leb}\right)  \int_{a}^{b}\left(  u\left(  x\right)
-f\left(  x\right)  \right)  \,\mathrm{d}x<\frac{\varepsilon}{2}\,.
\label{Perr.1}%
\end{equation}
If $U\left(  x\right)  =\int_{a}^{x}u\left(  x\right)  \,\mathrm{d}x$, then
the pair $\left(  \mathsf{U}^{\prime}\mathsf{,U}\right)  ,$ where
$U\leftrightarrow\mathsf{U,}$ is a distributional major pair for $f;$ with
\begin{equation}
U\left(  b\right)  <\left(  \mathfrak{Leb}\right)  \int_{a}^{b}f\left(
x\right)  \,\mathrm{d}x+\frac{\varepsilon}{2}\,. \label{Perr.2}%
\end{equation}
Similarly, employing minor functions and upper semi-continuous functions, we
can find a minor distributional major pair for $f,$ $\left(  \mathsf{V}%
^{\prime}\mathsf{,V}\right)  $ with%
\begin{equation}
V\left(  b\right)  >\left(  \mathfrak{Leb}\right)  \int_{a}^{b}f\left(
x\right)  \,\mathrm{d}x-\frac{\varepsilon}{2}\,. \label{Perr.3}%
\end{equation}
The distributional integrability of $f$ and the fact that
\begin{equation}
\left(  \mathfrak{dist}\right)  \int_{a}^{b}f\left(  x\right)  \,\mathrm{d}%
x=\left(  \mathfrak{Leb}\right)  \int_{a}^{b}f\left(  x\right)  \,\mathrm{d}%
x\,, \label{Perr.4}%
\end{equation}
then follow.\smallskip
\end{proof}

The Perron method of integration uses major and minor \emph{functions}
\cite{Gordon, Natanson2, Sacks}. We shall show that these functions give major
and minor distributional pairs in a natural way.\smallskip

\begin{theorem}
\label{Theorem COM 1p}Any Denjoy-Perron-Henstock-Kurzweil integrable function over
$\left[  a,b\right]  $ is also distributionally integrable over $\left[
a,b\right]  $ and the integrals coincide.
\end{theorem}

\begin{proof}
Let $U$ be a continuous major function for a Denjoy-Perron-Henstock-Kurzweil
integrable\ function $f$ over $\left[  a,b\right]  .$ Then the pair $\left(
\mathsf{U}^{\prime}\mathsf{,U}\right)  ,$ where $U\leftrightarrow\mathsf{U,}$
is a distributional major pair for $f.$ Indeed, the derivative $U^{\prime
}\left(  x\right)  $ exists $\left(  \text{a.e.}\right)  $ in $\left[
a,b\right]  ,$ and at those points the distributional value $\mathsf{U}%
^{\prime}\left(  x\right)  $ exists, and thus $(\mathsf{U}^{\prime}%
)_{\phi,0}^{-}\left(  x\right) =\mathsf{U}^{\prime}\left(  x\right)
=U^{\prime}\left(  x\right)  \geq f\left(  x\right)  $ for all $\phi
\in\mathcal{T}_{0}$.

Furthermore, for any $x\in\left[  a,b\right]  ,$%
\begin{equation}
\liminf_{y\rightarrow x}\frac{U\left(  y\right)  -U\left(  x\right)  }%
{y-x}>-\infty\,. \label{Perron.5}%
\end{equation}
But if $\left(  y-x\right)^{-1}(U\left(  y\right)  -U\left(  x\right))  \geq M  $ for
$\left\vert x-y\right\vert <c,$ then we can write $\mathsf{U}=\mathsf{U}%
_{1}+\mathsf{U}_{2},$ where $\mathsf{U}_{1}\left(  y\right)  =\chi_{\left(
x-c,x+c\right)  }\left(  y\right)  \mathsf{U}_{1}\left(  y\right)  .$ Let
$\phi\in\mathcal{T}_{1}.$ Since $\mathsf{U}_{2}\left(  y\right)  =0$ in a
neighborhood of $y=x,$ it follows that $\left\langle \mathsf{U}_{2}^{^{\prime
}}\left(  x+\varepsilon y\right)  ,\phi\left(  y\right)  \right\rangle
\rightarrow0.$ Also,%
\begin{align*}
\liminf_{\varepsilon\rightarrow0^{+}}\left\langle \mathsf{U}_{1}^{^{\prime}%
}\left(  x+\varepsilon y\right)  ,\phi\left(  y\right)  \right\rangle  &
=-\liminf_{\varepsilon\rightarrow0^{+}}\frac{1}{\varepsilon}\left\langle
\mathsf{U}_{1}\left(  x+\varepsilon y\right)  -\mathsf{U}_{1}\left(  x\right)
,\phi^{\prime}\left(  y\right)  \right\rangle \\
&  \geq-M\int_{-\infty}^{\infty}y\phi^{\prime}\left(  y\right)  \,\mathrm{d}%
y\\
&  =M\,.
\end{align*}
Hence $(\mathsf{U}^{\prime})_{\phi,0}^{-}\left(  x\right)  \geq M>-\infty.$

Similarly, if $V$ is a continuous minor function for $f,$ then $\left(
\mathsf{V}^{\prime}\mathsf{,V}\right)  ,$ where $V\leftrightarrow\mathsf{V,}$
is a distributional minor pair for $f.$ The fact that $f$ is distributionally
integrable over $\left[  a,b\right]  $ and that the distributional and the
Denjoy-Perron-Henstock-Kurzweil integrals coincide is now clear.\smallskip
\end{proof}

On the other hand, one does not need to go beyond the Lebesgue integral when
considering positive functions.\smallskip

\begin{theorem}
\label{Theorem COM 2}Let $f$ be distributionally integrable over $\left[
a,b\right]  .$ If $f\left(  x\right)  \geq0$ $\forall x\in\left[  a,b\right]
$ then $f$ is Lebesgue integrable over $\left[  a,b\right]  .$

Suppose that $f_{1}$ and $f_{2}$ are distributionally integrable over $\left[
a,b\right]  $ and $f_{1}\left(  x\right)  \geq f_{2}\left(  x\right)  $
$\forall x\in\left[  a,b\right]  .$ Then $f_{1}$ is Lebesgue integrable over
$\left[  a,b\right]  $ if and only if $f_{2}$ is Lebesgue integrable over
$\left[  a,b\right]  .$ Similarly, $f_{1}$ is Denjoy-Perron-Henstock-Kurzweil
integrable over $\left[  a,b\right]  $ if and only if $f_{2}$ is
Denjoy-Perron-Henstock-Kurzweil integrable over $\left[  a,b\right]  .$
\end{theorem}

\begin{proof}
Let $\left(  \mathsf{u},\mathsf{U}\right)  $ be a major pair for $f.$ Then
because $f\left(  x\right)  \geq0$ $\forall x$ it follows that $\left(
\mathsf{0,0}\right)  $ is a minor pair for $f.$ Therefore, $\mathsf{u}$ is a
positive measure, the point values $\widetilde{u}\left(  x\right)
=\mathsf{u}\left(  x\right)  $ exist almost everywhere, and satisfy
$\widetilde{u}\left(  x\right)  \geq f\left(  x\right)  $ almost everywhere.
Since (\ref{PR.6p}) yields that $\int_{a}^{b}\widetilde{u}\left(  x\right)
\,\mathrm{d}x<\infty$ we obtain by comparison that $f$ is Lebesgue integrable
over $\left[  a,b\right]  .$

The second part follows by writing $f_{1}=f_{2}+\left(  f_{1}-f_{2}\right)  $
and observing that $\left(  f_{1}-f_{2}\right)  $ is Lebesgue integrable over
$\left[  a,b\right]  .$\smallskip
\end{proof}

>From this theorem it follows that if $f$ is distributionally integrable over
$\left[  a,b\right]  $ and $f\left(  x\right)  \geq0$ $\forall x\in\left[
a,b\right]  $ then actually the distribution $\mathsf{f=F}^{\prime}$ is a
positive measure (indeed, a regular distribution).

We also have the following results, natural for non absolute
integrals.\smallskip

\begin{theorem}
\label{Theorem COM 3}If $f$ is distributionally integrable over any measurable
subset of $\left[  a,b\right]  $ then $f$ is Lebesgue integrable over $\left[
a,b\right]  .$

If $f$ is measurable and $\left\vert f\right\vert $ is distributionally
integrable over $\left[  a,b\right]  $ then $f$ is Lebesgue integrable over
$\left[  a,b\right]  .$
\end{theorem}

\begin{proof}
Indeed, in the first case both $f_{+}=f\chi_{A^{+}}=\left(  f+\left\vert
f\right\vert \right)  /2$ and $f_{-}=f\chi_{A^{-}}=\left(  \left\vert
f\right\vert -f\right)  /2,$ where $A^{\pm}$ is the set where $\pm f\left(
t\right)  >0,$ are distributionally integrable, and since they are positive
they must be Lebesgue integrable. Then $f=f_{+}-f_{-}$ would be Lebesgue integrable.

If now $\left\vert f\right\vert $ is distributionally integrable over $\left[
a,b\right]  $ then it is Lebesgue integrable over $\left[  a,b\right]  .$
Since $0\leq f_{\pm}\leq\left\vert f\right\vert ,$ it follows that both
$f_{\pm}$ are Lebesgue integrable, and so is $f=f_{+}-f_{-}.$\smallskip
\end{proof}

Our next task is to consider \L ojasiewicz functions.\smallskip

\begin{theorem}
\label{Theorem COM 4}Let $G$ be a \L ojasiewicz function, with associated
distribution $\mathsf{G.}$ Let $\mathsf{g=G}^{\prime},$ and suppose the
distributional point values
\begin{equation}
g\left(  x\right)  =\mathsf{g}\left(  x\right)  \,, \label{COM.1}%
\end{equation}
exist for all $x\in\left[  a,b\right]  \setminus E,$ where $\left\vert
E\right\vert \leq\aleph_{0}.$ Then $g$ is distributionally integrable over
$\left[  a,b\right]  $ and
\begin{equation}
\left(  \mathfrak{dist}\right)  \int_{c}^{d}g\left(  x\right)  \,\mathrm{d}%
x=G\left(  d\right)  -G\left(  c\right)  \,, \label{COM.2}%
\end{equation}
for $\left[  c,d\right]  \subset\left[  a,b\right]  .$
\end{theorem}

\begin{proof}
Let $\mathsf{H}$ be the \L ojasiewicz distribution equal to $0$ in $\left(
-\infty,a\right)  ,$ equal to $\mathsf{G-G}\left(  a\right)  $ in $\left(
a,b\right)  ,$ and equal to the constant $\mathsf{G}\left(  b\right)
\mathsf{-G}\left(  a\right)  $ in $\left(  b,\infty\right)  .$ Then the pair
$\left(  \mathsf{H}^{\prime}\mathsf{,H}\right)  $ is both a major and a minor
pair for $g.$\smallskip
\end{proof}

This theorem applies to \L ojasiewicz functions, and, more generally, to
distributionally regulated functions.\smallskip

\begin{theorem}
\label{Theorem COM 5}Any \L ojasiewicz function is distributionally
integrable. Any distributionally regulated function\ is distributionally
integrable.\smallskip
\end{theorem}

Observe that the integral of a \L ojasiewicz function, obtained from the
Definition \ref{Def IN3}, is equal to (\ref{INT}), the definition given in
\cite{loj}. Similarly the integral of a distributionally regulated function
reduces to the definition (\ref{Int}).

We can generalize the Theorem \ref{Theorem COM 4} by considering a
\L ojasiewicz distribution whose derivative has values almost everywhere if we
assume distributional boundedness at the other points.\smallskip

\begin{theorem}
\label{Theorem COM 6}Let $G$ be a \L ojasiewicz function, with associated
distribution $\mathsf{G.}$ Let $\mathsf{g=G}^{\prime},$ and suppose the
distributional point values
\begin{equation}
g\left(  x\right)  =\mathsf{g}\left(  x\right)  \,, \label{COM.3}%
\end{equation}
exist almost everywhere in $\left[  a,b\right]  ,$ while $\mathsf{g}$ is
distributionally bounded at all $x\in\left[  a,b\right]  \setminus E,$ where
$\left\vert E\right\vert \leq\aleph_{0}.$ Then $g$ is distributionally
integrable over $\left[  a,b\right]  $ and
\begin{equation}
\left(  \mathfrak{dist}\right)  \int_{c}^{d}g\left(  x\right)  \,\mathrm{d}%
x=G\left(  d\right)  -G\left(  c\right)  \,,
\end{equation}
for $\left[  c,d\right]  \subseteq\left[  a,b\right]  .$
\end{theorem}

\begin{proof}
Let $g$ be an extension of the function $\mathsf{g}\left(  t\right)  ,$
defined in the set of full measure where the values exist, to $\left[
a,b\right]  .$ If, as before, $\mathsf{H}$ is the \L ojasiewicz distribution
equal to $0$ in $\left(  -\infty,a\right)  ,$ equal to $\mathsf{G-G}\left(
a\right)  $ in $\left(  a,b\right)  ,$ and equal to the constant
$\mathsf{G}\left(  b\right)  \mathsf{-G}\left(  a\right)  $ in $\left(
b,\infty\right)  ,$ then the pair $\left(  \mathsf{H}^{\prime}\mathsf{,H}%
\right)  $ is both a major and a minor pair for $g$ because of the
distributional boundedness of $\mathsf{g}$ on $[a,b]\setminus E.$\smallskip
\end{proof}

We now apply the ideas of this section to reconstruct functions from their
higher order \emph{Peano generalized derivatives}. Let $f$ be continuous on
$[a,b]$, we say that $f$ has a Peano $n^{\text{th}}$ derivative at $x\in(a,b)$
if there are $n$ numbers $f_{1}(x),\dots,f_{n}(x)$ such that
\begin{equation}
\label{PD}f(x+h)=f(x)+f_{1}(x)h+\dots+f_{n}(x)\frac{h^{n}}{n!}+o(h^{n})\;,
\ \ \ \mbox{as }h\to0\ .
\end{equation}
We call each $f_{j}(x)$ its Peano $j^{\text{th}}$ derivative at $x$. The same
notion makes sense at $x=a$ or $x=b$ if we only ask (\ref{PD}) to hold as
$h\to0^{+}$ or $h\to0^{-}$, respectively. Notice that the ordinary first order
derivative of $f$ must exist at $x$, and actually $f^{\prime}(x)=f_{1}(x)$. We
set $f_{0}(x)=f(x)$.

Suppose that (\ref{PD}) holds everywhere in $[a,b]$. Naturally, the everywhere
existence of the Peano $n^{\text{th}}$ derivative does not even imply that $f$
is $C^{1}$. On the other hand, if the distribution $\mathsf{f}$ corresponds to
$f$, where $f$ has been extended to $\mathbb{R}$ as $f(x)=\sum_{j=0}^{n}%
(f_{j}(a)/j!)(x-a)^{j}$ for $x\leq a$ and $f(x)=\sum_{j=0}^{n}(f_{j}%
(b)/j!)(x-b)^{j}$ for $b\leq x$, then we do have that the $\mathsf{f}^{(j)}$
are \L ojasiewicz distributions for all $0\leq j\leq n$ and, indeed,
$\mathsf{f}^{(j)}(x)=f_{j}(x)$ $\forall x\in[a,b]$. Thus, the functions
$f_{j}$ are distributionally integrable over $[a,b]$ and
\begin{equation}
\label{INTD}f_{j-1}(x)= f_{j-1}(a)+ \left(  \mathfrak{dist}\right)  \int
_{a}^{x}f_{j}(x)\:\mathrm{d}x, \ \ \ j=n,\dots,1\ ,
\end{equation}
The relations (\ref{INTD}) allow us to reconstruct $f$ from $f_{n}$ as an
$n$-times iterated integral. Furthermore, we obtain the ensuing stronger
result if we employ Theorem \ref{Theorem COM 6}.

\begin{theorem}
\label{Theorem COM 7} Let $f$ be continuous on $[a,b]$. Suppose that $f$ has
Peano $(n-1)^{\text{th}}$ derivative at every point of $[a,b]$. Furthermore,
assume that there is a denumerable set $E$ such that for all $x\in
[a,b]\setminus E$
\begin{equation}
f(x+h)=f(x)+f_{1}(x)h+\dots+f_{n-1}(x)\frac{h^{n-1}}{(n-1)!}+O(h^{n})\;,
\ \ \ \mbox{as }h\to0\,.
\end{equation}
If the Peano $n^{\mathit{th}}$ derivative $f_{n}(x)$ of $f$ exits almost
everywhere in $[a,b]$, then $f_{n}$ is distributionally integrable over
$[a,b]$ and
\begin{equation}
\label{RPD}f(x)=\sum_{j=0}^{n-1}\frac{f_{j}(a)}{j!}(x-a)^{j}+ \left(
\mathfrak{dist}\right)  \int_{a}^{x}\int_{a}^{t_{n-1}}\dots\int_{a}^{t_{2}%
}f_{n}(t_{1})\: \mathrm{d}t_{1}\cdots\mathrm{d}t_{n-1}\mathrm{d}t_{n}\: .
\end{equation}

In particular, if $f_{n}(x)=0$ a.e., then $f$ is a polynomial of degree at
most $n-1$.
\end{theorem}

\begin{remark}
The last integration step in (\ref{RPD}) may be made with the
Denjoy-Perron-Henstock-Kurzweil integral; however, in general, the previous integrals
do not have to exist in the sense of Denjoy-Perron-Henstock-Kurzweil.
\end{remark}

\section{Distributions and
integration}\label{Section Distributions and Integration}

If $f$ is a distributionally integrable function over $\left[  a,b\right]  ,$
with indefinite integral $F,$ which in turn has an associated distribution
$\mathsf{F,}$ $F\leftrightarrow\mathsf{F,}$ then we proved in Theorem
\ref{Theorem PR 2} that the distribution $\mathsf{f=F}^{\prime}$ has
distributional point values almost everywhere and actually $\mathsf{f}\left(
x\right)  =f\left(  x\right)  $ almost everywhere in $\left[  a.b\right]  .$
Our aim is to show that the association $f\leftrightarrow\mathsf{f,}$ is a
\emph{natural} one, in the same way that Lebesgue integrable functions are
associated to regular distributions, by showing that
\begin{equation}
\left\langle \mathsf{f,}\psi\right\rangle =\left(  \mathfrak{dist}\right)
\int_{a}^{b}f\left(  x\right)  \psi\left(  x\right)  \,\mathrm{d}x\,,
\label{DI.1}%
\end{equation}
for all test functions $\psi\in\mathcal{D}\left(  \mathbb{R}\right)  .$

Observe, first of all, that if a distribution $\mathsf{f}$ has point values
\emph{everywhere} then there is a well defined association between $f,$ the
function given by those values, and $\mathsf{f.}$ That (\ref{DI.1}) is
satisfied in this case was proved by \L ojasiewicz \cite{loj}. However, when
we extend this idea to values that exist almost everywhere we need to proceed
with care. For instance, the Dirac delta function $\delta\left(  x\right)  $
has distributional values almost everywhere equal to $0,$ but is not the zero
distribution, or the distributional derivative of the Cantor function is a
measure concentrated on the Cantor set, and thus it has values a.e. equal to
$0$ without being the null distribution. \emph{Unless }(\ref{DI.1})\emph{ is
satisfied one cannot associate a distribution to the function given by its
point values, even if those values exist almost everywhere.}

Actually the almost everywhere values of a distribution tell us very little
about the nature of the distribution.\smallskip

\begin{theorem}
\label{Theorem DI 1}Let $f$ be any finite measurable function defined in
$\left[  a,b\right]  .$ Then there are infinitely many distributions
$\mathsf{g}$ that have distributional values almost everywhere and satisfy%
\begin{equation}
\mathsf{g}\left(  x\right)  =f\left(  x\right)  ~~~~~\left(  \text{a.e.}%
\right)  \,. \label{DI.2}%
\end{equation}

\end{theorem}

\begin{proof}
Existence follows at once from Lusin's Theorem \cite[Section VII. 2]{Sacks},
that says that if $f$ is any finite measurable function defined in $\left[
a,b\right]  $ then there exists a continuous function $F$ such that
$F^{\prime}\left(  x\right)  =f\left(  x\right)  $ almost everywhere. We then
consider the distribution $\mathsf{g}_{0}\mathsf{=F}^{\prime},$ where
$F\leftrightarrow\mathsf{F.}$ If $\mathsf{g}_{1}$ is any distribution whose
support has measure $0$ then $\mathsf{g}=\mathsf{g}_{0}+\mathsf{g}_{1}$
satisfies (\ref{DI.2}).\smallskip
\end{proof}

We now proceed to the proof of the formula (\ref{DI.1}) when $f$ is
distributionally integrable over $\left[  a,b\right]  .$ First we shall prove
that $f\psi$ is distributionally integrable whenever $\psi$ is $C^{\infty}$ on
$\left[  a,b\right]  .$ Since any smooth function defined in $\left[
a,b\right]  $ can be extended to the whole real line, and since the integral
of $f\psi$ over $\left[  a,b\right]  $ does not depend on how we do the
extension, it is convenient to assume that $\psi$ is actually smooth in
$\mathbb{R}.$\smallskip

\begin{theorem}
\label{Theorem DI 2}Let $f$ be distributionally integrable over $\left[
a,b\right]  $ and let $\psi$ be any smooth function defined in $\mathbb{R}.$
Then $f\psi$ is distributionally integrable over $\left[  a,b\right]  $ and
\begin{equation}
\left(  \mathfrak{dist}\right)  \int_{a}^{b}f\left(  x\right)  \psi\left(
x\right)  \,\mathrm{d}x=F\left(  b\right)  \psi\left(  b\right)  -\left(
\mathfrak{dist}\right)  \int_{a}^{b}F\left(  x\right)  \psi^{\prime}\left(
x\right)  \,\mathrm{d}x\,, \label{DI.3}%
\end{equation}
where $F$ is the indefinite integral $F\left(  x\right)  =\left(
\mathfrak{dist}\right)  \int_{a}^{x}f\left(  t\right)  \,\mathrm{d}t.$
\end{theorem}

\begin{proof}
Observe, first, that $F$ is a \L ojasiewicz function, and hence so is
$F\psi^{\prime}.$ Thus $F\psi^{\prime}$ is distributionally integrable over
$\left[  a,b\right]  $ and consequently the right side of the equation
(\ref{DI.3}) is well defined.

We start with the case when $\psi\left(  x\right)  >0$ $\forall x\in\left[
a,b\right]  .$ Let $\left(  \mathsf{u},\mathsf{U}\right)  $ be major pair for
$f$ in $\left[  a,b\right]  .$ Let $\mathsf{H}_{+}$ be the \L ojasiewicz
distribution that satisfies $\mathsf{H}_{+}^{\prime}=\psi^{\prime}\mathsf{U}$
in $\left(  a,b\right)  ,$ equal to $0$ in $\left(  -\infty,a\right)  ,$ and
constant in $\left(  b,\infty\right)  .$ Then the pair $\left(  \psi
\mathsf{u},\psi\mathsf{U}-\mathsf{H}_{+}\right)  $ is a distributional major
pair for $\psi f.$ Indeed, conditions 1 and 2 of the Definition \ref{Def IN1}
are clear, while for 3 we observe for a fixed $x$
\begin{equation}
\mathsf{U}\left(  x+ty\right)  =\mathsf{U}\left(  x\right)  +o\left(
1\right)  \,,\ \ \ \text{as }t\rightarrow0^{+}\,,\ \ \ \mbox{in }\mathcal{D}%
^{\prime}(\mathbb{R})\:,\label{DI.4}%
\end{equation}
and thus%
\begin{equation}
\mathsf{u}\left(  x+ty\right)  =o\left(  1/t\right)  \,,\ \ \ \text{as
}t\rightarrow0^{+}\,,\ \ \ \mbox{in }\mathcal{D}^{\prime}(\mathbb{R}%
)\:.\label{DI.5}%
\end{equation}
Hence%
\begin{align}
\psi\left(  x+ty\right)  \mathsf{u}\left(  x+ty\right)   &  =\left(
\psi\left(  x\right)  +O\left(  t\right)  \right)  \mathsf{u}\left(
x+ty\right)  \nonumber\\
&  =\psi\left(  x\right)  \mathsf{u}\left(  x+ty\right)  +o\left(  1\right)
\,,\ \ \ \mbox{in }\mathcal{D}^{\prime}(\mathbb{R})\:.\label{DI.6}%
\end{align}
Since $\psi\left(  x\right)  >0$ we obtain that if $\mathsf{u}_{\phi,0}%
^{-}\left(  x\right)  >-\infty,$ $\phi\in\mathcal{T}_{1},$ then%
\begin{equation}
\left(  \psi\mathsf{u}\right)  _{\phi,0}^{-}\left(  x\right)  >-\infty
\,,\label{DI.7}%
\end{equation}
while if $\mathsf{u}_{\phi,0}^{-}\left(  x\right)  \geq f\left(
x\right)  ,$ $\phi\in\mathcal{T}_{0},$ then%
\begin{equation}
\left(  \psi\mathsf{u}\right)  _{\phi,0}^{-}\left(  x\right)  \geq
\psi\left(  x\right)  f\left(  x\right)  \,.\label{DI.8}%
\end{equation}
Similarly, if $\left(  \mathsf{v},\mathsf{V}\right)  $ is a minor pair for $f$
in $\left[  a,b\right]  ,$ then $\left(  \psi\mathsf{v},\psi\mathsf{V}%
-\mathsf{H}_{-}\right)  $ is a minor pair for $\psi f,$ where $\mathsf{H}_{-}$
is the \L ojasiewicz distribution that satisfies $\mathsf{H}_{-}^{\prime}%
=\psi^{\prime}\mathsf{V}$ in $\left(  a,b\right)  ,$ equal to $0$ in $\left(
-\infty,a\right)  ,$ and constant in $\left(  b,\infty\right)  .$

Let $\varepsilon>0,$ and choose the major and minor pairs $\left(
\mathsf{u},\mathsf{U}\right)  $ and $\left(  \mathsf{v},\mathsf{V}\right)  $
in such a way that%
\begin{equation}
\mathsf{U}\left(  b\right)  -\mathsf{V}\left(  b\right)  <\varepsilon\left(
\psi\left(  b\right)  +\int_{a}^{b}\left\vert \psi^{\prime}\left(  x\right)
\right\vert \,\mathrm{d}x\right)  ^{-1} \,. \label{DI.9}%
\end{equation}
Then the major and minor pairs $\left(  \psi\mathsf{u},\psi\mathsf{U}%
-\mathsf{H}_{+}\right)  $ and $\left(  \psi\mathsf{v},\psi\mathsf{V}%
-\mathsf{H}_{-}\right)  $ for $\psi f$ satisfy%
\begin{equation}
\left(  \psi\left(  b\right)  \mathsf{U}\left(  b\right)  -\mathsf{H}%
_{+}\left(  b\right)  \right)  -\left(  \psi\left(  b\right)  \mathsf{V}%
\left(  b\right)  -\mathsf{H}_{-}\left(  b\right)  \right)  <\varepsilon\,,
\label{DI.10}%
\end{equation}
where we have used Lemma \ref{Lemma IN1}. The distributional integrability of
$\psi f$ is obtained.

If we take the infimum of $\psi\left(  b\right)  \mathsf{U}\left(  b\right)
-\mathsf{H}_{+}\left(  b\right)  ,$ or the supremum of $\psi\left(  b\right)
\mathsf{V}\left(  b\right)  -\mathsf{H}_{-}\left(  b\right)  ,$ we obtain
$\psi\left(  b\right)  F\left(  b\right)  -\int_{a}^{b}\psi^{\prime}\left(
x\right)  F\left(  x\right)  \,\mathrm{d}x,$ and this yields the integration
by parts formula (\ref{DI.3}).

For a general function $\psi\in C^{\infty}\left(  \mathbb{R}\right)  $ we can
find a constant $k>0$ such that $k+\psi\left(  x\right)  >0$ $\forall
x\in\left[  a,b\right]  .$ The distributional integrability of $\psi f=\left(
k+\psi\right)  f-kf$ follows, while formula (\ref{DI.3}) is obtained because
it holds for both $\left(  k+\psi\right)  f$ and $kf.$\smallskip
\end{proof}

We can now prove that the association $f\leftrightarrow\mathsf{f}$ is a
natural one.\smallskip

\begin{theorem}
\label{Theorem DI 3}Let $f$ be distributionally integrable function over
$\left[  a,b\right]  ,$ its indefinite integral be $F,$ with associated
distribution $\mathsf{F,}$ $F\leftrightarrow\mathsf{F,}$ and let
$\mathsf{f=F}^{\prime}\in\mathcal{E}^{\prime}(\mathbb{R}),$ so that
$\mathsf{f}\left(  x\right)  =f\left(  x\right)  $ almost everywhere in
$\left[  a,b\right]  .$ Then for any $\psi\in\mathcal{E}\left(  \mathbb{R}%
\right)  ,$%
\begin{equation}
\left\langle \mathsf{f,}\psi\right\rangle =\left(  \mathfrak{dist}\right)
\int_{a}^{b}f\left(  x\right)  \psi\left(  x\right)  \,\mathrm{d}x\,.
\label{DI.11}%
\end{equation}

\end{theorem}

\begin{proof}
Let $\chi$ be the characteristic function of $\left[  a,b\right]  .$ Then
$\chi\mathsf{F}$ is distributionally regulated, with a jump of magnitude
$-F\left(  b\right)  $ at $x=b.$ Thus%
\begin{equation}
\left(  \chi\left(  x\right)  \mathsf{F}\left(  x\right)  \right)  ^{\prime
}=\mathsf{f}\left(  x\right)  -F\left(  b\right)  \delta\left(  x-b\right)
\,, \label{DI.12}%
\end{equation}
and this yields%
\begin{align*}
\left\langle \mathsf{f,}\psi\right\rangle  &  =\left\langle \left(
\chi\mathsf{F}\right)  ^{\prime}+F\left(  b\right)  \delta\left(  x-b\right)
\mathsf{,}\psi\right\rangle \\
&  =F\left(  b\right)  \psi\left(  b\right)  -\left\langle \chi\mathsf{F,}%
\psi^{\prime}\right\rangle \\
&  =F\left(  b\right)  \psi\left(  b\right)  -\left(  \mathfrak{dist}\right)
\int_{a}^{b}F\left(  x\right)  \psi^{\prime}\left(  x\right)  \,\mathrm{d}x\\
&  =\left(  \mathfrak{dist}\right)  \int_{a}^{b}f\left(  x\right)  \psi\left(
x\right)  \,\mathrm{d}x\,,
\end{align*}
as required.\smallskip
\end{proof}

\begin{remark}
The Theorem \ref{Theorem DI 2} actually shows that the integration by parts
formula%
\begin{equation}
\left(  \mathfrak{dist}\right)  \int_{a}^{b}f\left(  x\right)  \psi\left(
x\right)  \,\mathrm{d}x=\left.  G\left(  x\right)  \psi\left(  x\right)
%TCIMACRO{\QATOPD{.}{.}{{}}{{}}}%
%BeginExpansion
\genfrac{.}{.}{0pt}{}{{}}{{}}%
%EndExpansion
\right\vert _{x=a}^{x=b}-\left(  \mathfrak{dist}\right)  \int_{a}^{b}G\left(
x\right)  \psi^{\prime}\left(  x\right)  \,\mathrm{d}x\,, \label{DI.10p}%
\end{equation}
holds for any \L ojasiewicz function $G,$ $G\leftrightarrow\mathsf{G,}$ with
$\mathsf{f=G}^{\prime},$ and $\psi$ smooth.\smallskip
\end{remark}

We say that a function $f$ defined in $\mathbb{R}$ is locally distributionally
integrable if $f$ is integrable over any compact \emph{interval} of
$\mathbb{R}.$ For such a function we define the improper integral%
\begin{equation}
\left(  \mathfrak{dist}\right)  \int_{-\infty}^{\infty}f\left(  x\right)
\,\mathrm{d}x=\lim_{\substack{a\rightarrow-\infty\\b\rightarrow\infty}}\left(
\mathfrak{dist}\right)  \int_{a}^{b}f\left(  x\right)  \,\mathrm{d}x\,,
\label{DI.12p}%
\end{equation}
if the limit exists.

The previous theorem treats the case of $\mathcal{E}^{\prime}\left(
\mathbb{R}\right)  ;$ for the space $\mathcal{D}^{\prime}\left(
\mathbb{R}\right)  $ we have a corresponding result.\smallskip

\begin{theorem}
\label{Theorem DI 4}Let $f$ be locally distributionally integrable over
$\mathbb{R}.$ Then the formula%
\begin{equation}
\psi\rightsquigarrow\left(  \mathfrak{dist}\right)  \int_{-\infty}^{\infty
}f\left(  x\right)  \psi\left(  x\right)  \,\mathrm{d}x\,, \label{DI.13}%
\end{equation}
for $\psi\in\mathcal{D}\left(  \mathbb{R}\right)  ,$ defines a distribution
$\mathsf{f}\in\mathcal{D}^{\prime}\left(  \mathbb{R}\right)  .$ Actually for
any fixed $a\in\mathbb{R},$ $\mathsf{f=F}^{\prime},$ where $F\leftrightarrow
\mathsf{F,}$ and $F\left(  x\right)  =\left(  \mathfrak{dist}\right)  \int
_{a}^{x}f\left(  t\right)  \,\mathrm{d}t.$

Furthermore, $\mathsf{f}\left(  x\right)  =f\left(  x\right)  $ almost
everywhere in $\mathbb{R}$.
\end{theorem}

\begin{proof}
It follows at once from the Theorem \ref{Theorem DI 3}. Observe that for any
$\psi$ the integral in (\ref{DI.13}) is not really an improper integral, but
actually an integral over a compact interval.\smallskip
\end{proof}

We can now extend the notion of association between a function and a
distribution. From our results, we can associate to any locally
distributionally integrable function over $\mathbb{R}$ a unique distribution.
We shall call those distributions \emph{locally integrable distributions.} The
association%
\begin{equation}
f\leftrightarrow\mathsf{f}\,, \label{DI.14}%
\end{equation}
between locally distributionally integrable functions and locally integrable
distributions is characterized by the equation%
\begin{equation}
\left\langle \mathsf{f},\psi\right\rangle =\left(  \mathfrak{dist}\right)
\int_{-\infty}^{\infty}f\left(  x\right)  \psi\left(  x\right)  \,\mathrm{d}%
x\,,\ \ \ \ \ \psi\in\mathcal{D}\left(  \mathbb{R}\right)  \,. \label{DI.15}%
\end{equation}
This association generalizes the association between locally Lebesgue
integrable functions and regular distributions as well as the association
between \L ojasiewicz functions and \L ojasiewicz distributions.\smallskip

We also have similar integral representation results for other spaces of
distributions. In fact, since all evaluations in $\mathcal{S}^{\prime}\left(
\mathbb{R}\right)  $ and $\mathcal{K}^{\prime}\left(  \mathbb{R}\right)  $ are
Ces\`{a}ro evaluations \cite{VEEdin}\ we immediately obtain the
following.\smallskip

\begin{theorem}
\label{Theorem DI 5}Let $\mathsf{f}$ be a locally integrable distribution,
$f\leftrightarrow\mathsf{f.}$ If $\mathsf{f}\in\mathcal{S}^{\prime}\left(
\mathbb{R}\right)  $ then there exists $k\in\mathbb{N}$ such that for all
$\phi\in\mathcal{S}\left(  \mathbb{R}\right)  ,$%
\begin{equation}
\left\langle \mathsf{f},\phi\right\rangle =\left(  \mathfrak{dist}\right)
\int_{-\infty}^{\infty}f\left(  x\right)  \phi\left(  x\right)  \,\mathrm{d}%
x\ \ \ \ \left(  \mathrm{C},k\right)  \,. \label{DI.15p}%
\end{equation}
If $\mathsf{f}\in\mathcal{K}^{\prime}\left(  \mathbb{R}\right)  $ and $\phi
\in\mathcal{K}\left(  \mathbb{R}\right)  ,$ then (\ref{DI.15p}) holds for some
$k\in\mathbb{N}$ that depends on $\phi.$\smallskip
\end{theorem}

Using the results of \cite{est2, EF}, (see \cite[Chp. 6]{est-kan}) we also
obtain the ensuing useful characterization.\smallskip

\begin{theorem}
\label{Theorem DI 6}Let $\mathsf{f}$ be a locally integrable distribution,
$f\leftrightarrow\mathsf{f.}$ Then $\mathsf{f}\in\mathcal{K}^{\prime}\left(
\mathbb{R}\right)  $ if and only if the integrals%
\begin{equation}
\left(  \mathfrak{dist}\right)  \int_{-\infty}^{\infty}f\left(  x\right)
x^{n}\mathrm{d}x\ \ \ \ \left(  \mathrm{C}\right)  \,, \label{DI.15p2}%
\end{equation}
exist in the Ces\`{a}ro sense for all $n\in\mathbb{N}.$\smallskip
\end{theorem}

At this point we point out an useful local bound for locally
integrable distributions.\smallskip

\begin{proposition}
\label{Prop DI 1}Let $\mathsf{f}$ be a locally integrable distribution. Then
for any $x\in\mathbb{R}$%
\begin{equation}
\mathsf{f}\left(  x+\varepsilon y\right)  =o\left(  1/\varepsilon\right)
\,,\ \ \ \ \ \varepsilon\rightarrow0^{+}, \label{DI.16}%
\end{equation}
in the space $\mathcal{D}^{\prime}\left(  \mathbb{R}\right)  ,$ that is, if
$\psi\in\mathcal{D}\left(  \mathbb{R}\right)  ,$ then%
\begin{equation}
\left\langle \mathsf{f}\left(  x+\varepsilon y\right)  ,\psi\left(  y\right)
\right\rangle =o\left(  1/\varepsilon\right)  \,,\ \ \ \ \ \varepsilon
\rightarrow0^{+}. \label{DI.17}%
\end{equation}

\end{proposition}

\begin{proof}
Indeed, if $\mathsf{F}$ is a primitive for $\mathsf{f,}$ $\mathsf{F}^{\prime
}=\mathsf{f,}$ then $\mathsf{F}$ is a \L ojasiewicz distribution, and thus the
point value $\mathsf{F}\left(  x\right)  $ exists. Hence in $\mathcal{D}%
^{\prime}\left(  \mathbb{R}\right)  ,$%
\begin{equation}
\mathsf{F}\left(  x+\varepsilon y\right)  =\mathsf{F}\left(  x\right)
+o\left(  1\right)  \ \ \ \text{as\ \ }\varepsilon\rightarrow0^{+}.
\label{DI.18}%
\end{equation}
Differentiation of (\ref{DI.18}) yields (\ref{DI.16}).\smallskip
\end{proof}

Suppose that $\mathsf{f}$ is a locally integrable distribution with compact
support contained in $\left[  a,b\right]  .$ Then while (\ref{DI.17}) is valid
at the endpoints $x=a$ and $x=b$ if $\psi\in\mathcal{D}\left(  \mathbb{R}%
\right)  ,$ it is enough to consider the distributional limits $\mathsf{f}%
\left(  a+\varepsilon y\right)  $ as $\varepsilon\rightarrow0^{+},$ only for
$y>0,$ and $\mathsf{f}\left(  b+\varepsilon y\right)  $ as $\varepsilon
\rightarrow0^{+},$ only for $y<0.$ This means that if $x=a,$ it suffices to ask
(\ref{DI.17}) to hold if $\psi\in\mathcal{D}\left(  \mathbb{R}\right)  $
satisfies $\operatorname*{supp}\psi\subset\left(  0,\infty\right)  ,$ or, when
needed, if $\operatorname*{supp}\psi\subset\lbrack0,\infty).$ Similarly at
$x=b$ one just needs to consider test functions with $\operatorname*{supp}%
\psi\subset\left(  -\infty,0\right)  $ or $\operatorname*{supp}\psi
\subset(-\infty,0].$

\section{Improper integrals}\label{Section Improper Integrals}

It is well known that if $f$ is Lebesgue integrable over $\left[  a,c\right]
$ for any $c<b,$ then the improper integral%
\begin{equation}
\int_{a}^{b}f\left(  x\right)  \,\mathrm{d}x=\lim_{c\rightarrow b^{-}}\int
_{a}^{c}f\left(  x\right)  \,\mathrm{d}x\,, \label{Imp.1}%
\end{equation}
may exist even when $f$ is not Lebesgue integrable over $\left[  a,b\right]
.$

On the other hand, according to Hake's theorem \cite{Hake} (see \cite{Bartle},
\cite{Gordon}, or \cite{Sacks}), if $f$ is Denjoy-Perron-Henstock-Kurzweil integrable
over $\left[  a,c\right]  $ for any $c<b,$ and the improper integral
(\ref{Imp.1}) exists, then $f$ must be Denjoy-Perron-Henstock-Kurzweil integrable over
$\left[  a,b\right]  .$ In other words, there is no such thing as improper
Denjoy-Perron-Henstock-Kurzweil integrals over a finite interval.

For the distributional integral we have the following result.\smallskip

\begin{theorem}
\label{Theorem Imp 1}Let $f$ be distributionally integrable over $\left[
a,c\right]  $ for any $c<b.$ Let $F\left(  x\right)  =\left(  \mathfrak{dist}%
\right)  \int_{a}^{x}f\left(  t\right)  \,\mathrm{d}t,$ $x<b,$ be its
indefinite integral, and let $\mathsf{F}$ be the corresponding \L oja\-sie\-wicz
distribution defined for $x<b,$ $F\leftrightarrow\mathsf{F.}$ Suppose that the
distributional limit%
\begin{equation}
\lim_{c\rightarrow b^{-}}\mathsf{F}\left(  c\right)  =L\,, \label{Imp.2}%
\end{equation}
exists. Then $f$ is distributionally integrable over $\left[  a,b\right]  $
and%
\begin{equation}
\left(  \mathfrak{dist}\right)  \int_{a}^{b}f\left(  x\right)  \,\mathrm{d}%
x=L\,. \label{Imp.3}%
\end{equation}

\end{theorem}

\begin{proof}
Let $\varepsilon>0.$ Let $\left\{  c_{n}\right\}  _{n=1}^{\infty}$ be a
strictly increasing sequence with $c_{1}=a$ and $c_{n}\nearrow b.$ For each
$n$ let $\left(  \mathsf{u}_{n},\mathsf{U}_{n}\right)  $ and $\left(
\mathsf{v}_{n},\mathsf{V}_{n}\right)  $ be major and minor pairs for $f$ over
$\left[  c_{n},c_{n+1}\right]  $ that satisfy $\mathsf{U}_{n}\left(
c_{n+1}\right)  -\mathsf{V}_{n}\left(  c_{n+1}\right)  <\varepsilon/2^{n}.$
The two series%
\begin{equation}
\widetilde{\mathsf{U}}=\sum_{n=1}^{\infty}\mathsf{U}_{n}%
\,,\ \ \ \ \ \widetilde{\mathsf{V}}=\sum_{n=1}^{\infty}\mathsf{V}_{n}\,,
\label{Imp.4}%
\end{equation}
converge distributionally in the interval $(-\infty,b),$ since in any interval
of the form $(-\infty,c)$ for $c<b$ the series become finite sums. The
distributions $\widetilde{\mathsf{U}}$ and $\widetilde{\mathsf{V}}$ are
\L ojasiewicz distributions for $x<b,$ and for each $c<b$ they yield major and
minor pairs for $f$ over $\left[  a,c\right]  ,$ $\left(  \mathsf{U}_{\left(
c\right)  }^{\prime},\mathsf{U}_{\left(  c\right)  }\right)  $ and $\left(
\mathsf{V}_{\left(  c\right)  }^{\prime},\mathsf{V}_{\left(  c\right)
}\right)  $ by taking $\mathsf{U}_{\left(  c\right)  }$ and $\mathsf{V}%
_{\left(  c\right)  }$ to be \L ojasiewicz distributions over $\mathbb{R}$
that equal $\widetilde{\mathsf{U}}$ and $\widetilde{\mathsf{V}},$
respectively, over $\left(  -\infty,c\right)  $ and constant over $\left(
c,\infty\right)  .$ Also, $\widetilde{\mathsf{U}}\left(  c\right)
-\widetilde{\mathsf{V}}\left(  c\right)  <\varepsilon$ for $c<b.$

Observe now that $\widetilde{\mathsf{U}}-\mathsf{F}$ and $\mathsf{F}%
-\widetilde{\mathsf{V}}$ are both \L ojasiewicz distributions over
$(-\infty,b),$ corresponding to continuous increasing functions. Since
$\mathsf{F}$ has a distributional limit from the left at $x=b,$ the same is
true of both $\widetilde{\mathsf{U}}$ and $\widetilde{\mathsf{V}},$ and thus
one can extend them as \L ojasiewicz distributions over $\mathbb{R}$ by asking
the extensions, say $\mathsf{U}$ and $\mathsf{V},$ to be constant over
$\left(  b,\infty\right)  .$ Then $\left(  \mathsf{U}^{\prime},\mathsf{U}%
\right)  $ and $\left(  \mathsf{V}^{\prime},\mathsf{V}\right)  $ are major and
minor pairs for $f$ over $\left[  a,b\right]  $ with $\mathsf{U}\left(
b\right)  -\mathsf{V}\left(  b\right)  \leq\varepsilon,$ and the
distributional integrability of $f$ over $\left[  a,b\right]  $ follows.
Furthermore, we also obtain the bounds%
\[
L-\varepsilon\leq\mathsf{V}\left(  b\right)  \leq\mathsf{U}\left(  b\right)
\leq L+\varepsilon\,,
\]
which immediately yield (\ref{Imp.3}).\smallskip
\end{proof}

One can rephrase the previous theorem by simply saying that the distributional integral
$\left(  \mathfrak{dist}\right)  \int_{a}^{b}f\left(  x\right)  \,\mathrm{d}x$
exists, and is finite, if and only if the distributional limit of $\left(
\mathfrak{dist}\right)  \int_{a}^{c}f\left(  x\right)  \,\mathrm{d}x$ as
$c\rightarrow b^{-}$ exists. We may reformulate Theorem \ref{Theorem Imp 1} if
we use local Ces\`{a}ro limits. Let $g$ be distributional integrable over
$[a,c]$ for any $c<b$. Define its sequence of $n$ primitives $\left\{
g_{a}^{(-n)}\right\}  _{n=0}^{\infty}$ on $[a,b)$ recursively as
\begin{equation}
g_{a}^{(0)}(x)=g(x)\;,\ \ \ g_{a}^{(-n-1)}(x)=\left(  \mathfrak{dist}\right)
\int_{a}^{x}g_{a}^{(-n)}(t)\;\mathrm{d}t\;,\ \ \ x\in\lbrack a,b)\,.
\end{equation}
We say that $g$ has a Ces\`{a}ro limit as $c\rightarrow b_{-}$, and write
\begin{equation}
\lim_{c\rightarrow b^{-}}g(c)=L\ \ \ (\mathrm{C})\;, \label{CLL}%
\end{equation}
if there exist $d\in\lbrack a,b)$, $n\in\mathbb{N}$, and a polynomial $p$ of
degree at most $n-1$ such that $g_{a}^{(-n)}$ is continuous on $(d,b)$ and
\begin{equation}
\lim_{c\rightarrow b^{-}}\frac{g_{a}^{(-n)}(c)-p(c)}{(c-b)^{n}}=\frac{L}%
{n!}\ .
\end{equation}
Let $g\leftrightarrow\mathsf{g}\in\mathcal{E}^{\prime}[a,b]$. Because of
\L ojasiewicz characterization of distributional limits
\cite{loj,vindas2,VP09}, we have that (\ref{CLL}) is equivalent to the
distributional lateral limit $\lim_{c\rightarrow b^{-}}\mathsf{g}(c)=L$. This
yields immediately the following version of Theorem \ref{Theorem Imp 1} in
which we replace (\ref{Imp.2}) by a Ces\`{a}ro limit.

\begin{theorem}
\label{Theorem Imp 1 v.2} Let $f$ be distributionally integrable over $\left[
a,c\right]  $ for any $c<b.$ Then $f$ is distributionally integrable over
$\left[  a,b\right]  $ if only if the following Ces\`{a}ro limit exists, and
it is finite,
\begin{equation}
\label{Imp.2.1}\lim_{c\to b^{-}}\left(  \mathfrak{dist}\right)  \int_{a}%
^{c}f\left(  x\right)  \,\mathrm{d}x=L \ \ \ (\mathrm{C})\, ,
\end{equation}
In this case (\ref{Imp.3}) holds.
\end{theorem}

Theorem \ref{Theorem Imp 1 v.2} therefore tells us that improper Ces\`{a}ro
distributional integrals are always definite integrals, and, conversely, any
definite integral may be computed by the Ces\`{a}ro limit (\ref{Imp.2.1}).
Thus, we have the following analogy with the Denjoy-Perron-Henstock-Kurzweil integral:
there are \emph{no improper} Ces\`{a}ro distributional integrals over finite intervals.

If the integrability of $f$ in $\left[  a,c\right]  $ for all $c<b$ is known,
then we may determine the integrability of $f$ over $\left[  a,b\right]  $
from the behavior of $\mathsf{f,}$ where $f\leftrightarrow\mathsf{f,}$ near
$x=b.$ One result in this direction is the following.\smallskip

\begin{theorem}
\label{Theorem Imp 2}Let $f$ be distributionally integrable over $\left[
a,c\right]  $ for any $c<b.$ Let $f\leftrightarrow\mathsf{f,}$ where
$\mathsf{f}$ is a distribution in $\mathcal{D}^{\prime}\left(  -\infty
,b\right)  .$ Suppose that
\begin{equation}
\mathsf{f}\left(  b+\varepsilon x\right)  =O\left(  \varepsilon^{\alpha
}\right)  \,,\ \ \ \ \ \varepsilon\rightarrow0^{+}, \label{Imp.5}%
\end{equation}
for some $\alpha>-1$ in the space $\mathcal{D}^{\prime}\left(  (-\infty
,0)\right)  ,$ that is, for $\psi\in\mathcal{D}\left(  \mathbb{R}\right)  ,$%
\begin{equation}
\left\langle \mathsf{f}\left(  b+\varepsilon y\right)  ,\psi\left(  y\right)
\right\rangle =O\left(  \varepsilon^{\alpha}\right)  \,,\ \ \ \ \ \varepsilon
\rightarrow0^{+},\ \ \ \text{whenever \ }\operatorname*{supp}\psi
\subset(-\infty,0)\,. \label{Imp.6}%
\end{equation}
Then $f$ is distributionally integrable over $\left[  a,b\right]  .$
\end{theorem}

\begin{proof}
Let $F\left(  x\right)  =\left(  \mathfrak{dist}\right)  \int_{a}^{x}f\left(
t\right)  \,\mathrm{d}t,$ $x<b,$ be the indefinite integral of $f,$ and let
$\mathsf{F}$ be the corresponding \L ojasiewicz distribution defined for
$x<b,$ $F\leftrightarrow\mathsf{F.}$ We need to show that $L,$ the
distributional limit of $\mathsf{F}\left(  c\right)  $ as $c\rightarrow b^{-}$
exists, namely, that
\begin{equation}
\lim_{\varepsilon\rightarrow0^{+}}\left\langle \mathsf{F}\left(  b+\varepsilon
y\right)  ,\psi\left(  y\right)  \right\rangle =L\int_{-\infty}^{\infty}%
\psi\left(  x\right)  \,\mathrm{d}x\,, \label{Imp.7}%
\end{equation}
whenever $\operatorname*{supp}\psi\subset(-\infty,0).$

Observe first that if $\int_{-\infty}^{\infty}\psi\left(  x\right)
\,\mathrm{d}x=0$ then $\psi=\varphi^{\prime},$ where $\varphi\in
\mathcal{D}\left(  \mathbb{R}\right)  ,$ $\operatorname*{supp}\varphi
\subset(-\infty,0).$ Thus%
\[
\left\langle \mathsf{F}\left(  b+\varepsilon y\right)  ,\psi\left(  y\right)
\right\rangle =\varepsilon\left\langle \mathsf{f}\left(  b+\varepsilon
y\right)  ,\varphi\left(  y\right)  \right\rangle =O\left(  \varepsilon
^{\alpha+1}\right)  \,,
\]
as $\varepsilon\rightarrow0^{+},$ so that (\ref{Imp.7}) holds with any $L$ if
the integral of $\psi$ vanishes.

Let $\psi_{0}$ be a fixed test function of $\mathcal{D}\left(  \mathbb{R}%
\right)  ,$ with $\operatorname*{supp}\psi_{0}\subset(-\infty,0)$ that
satisfies%
\begin{equation}
\int_{-\infty}^{\infty}\psi_{0}\left(  x\right)  \,\mathrm{d}x=1\,.
\label{Imp.8}%
\end{equation}
If $\psi\in\mathcal{D}\left(  \mathbb{R}\right)  ,$ $\operatorname*{supp}%
\psi\subset(-\infty,0),$ we can write $\psi=c\psi_{0}+\psi_{1},$ where
$c=\int_{-\infty}^{\infty}\psi\left(  x\right)  \,\mathrm{d}x,$ and where
$\int_{-\infty}^{\infty}\psi_{1}\left(  x\right)  \,\mathrm{d}x=0.$ Therefore,%
\begin{equation}
\left\langle \mathsf{F}\left(  b+\varepsilon y\right)  ,\psi\left(  y\right)
\right\rangle =\rho\left(  \varepsilon\right)  \int_{-\infty}^{\infty}%
\psi\left(  x\right)  \,\mathrm{d}x+O\left(  \varepsilon^{\alpha+1}\right)
\,, \label{Imp.9}%
\end{equation}
as $\varepsilon\rightarrow0^{+},$ where%
\begin{equation}
\rho\left(  \varepsilon\right)  =\left\langle \mathsf{F}\left(  b+\varepsilon
y\right)  ,\psi_{0}\left(  y\right)  \right\rangle \,. \label{Imp.10}%
\end{equation}
If $a>0$ then
\begin{equation}
\rho\left(  a\varepsilon\right)  =\rho\left(  \varepsilon\right)  +O\left(
\varepsilon^{\alpha+1}\right)  \,, \label{Imp.11}%
\end{equation}
since%
\begin{align*}
\rho\left(  a\varepsilon\right)   &  =\left\langle \mathsf{F}\left(
b+a\varepsilon y\right)  ,\psi_{0}\left(  y\right)  \right\rangle \\
&  =\frac{1}{a}\left\langle \mathsf{F}\left(  b+\varepsilon y\right)
,\psi_{0}\left(  y/a\right)  \right\rangle \\
&  =\frac{\rho\left(  \varepsilon\right)  }{a}\int_{-\infty}^{\infty}\psi
_{0}\left(  x/a\right)  \,\mathrm{d}x+O\left(  \varepsilon^{\alpha+1}\right)
\\
&  =\rho\left(  \varepsilon\right)  +O\left(  \varepsilon^{\alpha+1}\right)
\,.
\end{align*}
The asymptotic identity (\ref{Imp.11}) is actually valid uniformly in $a$ if
$a\in\left[  A,B\right]  $ and $0<A<B<\infty$ because weak convergence yields
strong convergence in spaces of distributions and thus (\ref{Imp.6}) holds
uniformly on $\psi$ if $\psi$ belongs to a compact subset of $\mathcal{D}%
\left(  \mathbb{R}\right)  .$ Hence $L,$ the limit of $\rho\left(
\varepsilon\right)  $ as $\varepsilon\rightarrow0^{+}$ exists (and actually
$\rho\left(  \varepsilon\right)  =L+O\left(  \varepsilon^{\alpha+1}\right)
$). The required formula, (\ref{Imp.7}) then follows from (\ref{Imp.9}%
).\smallskip
\end{proof}

It is interesting to observe that the condition $\mathsf{f}\left(
b+\varepsilon x\right)  =o\left(  1/\varepsilon\right)  \,,$ as $\varepsilon
\rightarrow0^{+}$ is not enough to give integrability of $f$ over $\left[
a,b\right]  .$ Take $f\left(  x\right)  =(\left(  b-x\right)  \ln\left(
b-x\right)  )^{-1},$ $(x<b),$ for instance. The preceding proof does not work
because (\ref{Imp.11}) becomes $\rho\left(  a\varepsilon\right)  =\rho\left(
\varepsilon\right)  +o\left(  1\right)  ,$ so that $\rho$ is asymptotically
homogeneous of degree $0$ \cite{est-kan,vindasthesis,VP09}, and some
asymptotically homogeneous functions may tend to infinity as $\varepsilon
\rightarrow0^{+}.$

Using the Theorem \ref{Theorem Imp 2} it is possible to give a clear meaning
to some irregular operations involving integrable distributions. If
$\mathsf{f}\in\mathcal{D}^{\prime}\left(  \mathbb{R}\right)  $ and $\chi$ is
the characteristic function of an interval $\left[  c,d\right]  $ then in
general there is no \emph{canonical} way of defining a distribution
$\chi\mathsf{f}$ of the space $\mathcal{D}^{\prime}\left(  \mathbb{R}\right)
$ \cite{est2003}; however, if $\mathsf{f}$ is a locally integrable
distribution then the Proposition \ref{Prop. IN 1} says that $\chi\mathsf{f}$
is defined in a natural way. The following result gives another such natural
definition, namely, that of $\left(  b-x\right)  ^{\beta}\mathsf{f}\left(
x\right)  $ if $\mathsf{f}\in\mathcal{E}^{\prime}\left[  a,b\right]  $ is
integrable and $\beta>0.$\smallskip

\begin{proposition}
Let $f$ be distributionally integrable over $\left[  a,b\right]  .$ If
$\beta>0$ then the function
\[
f_{\beta}\left(  x\right)  =\left(  b-x\right)  ^{\beta}f\left(  x\right)
\,,\label{Imp.12}%
\]
is distributionally integrable over $\left[  a,b\right]  .$ Similarly,
$\left(  x-a\right)  ^{\beta}f\left(  x\right)  $ is also distributionally
integrable over $\left[  a,b\right]  .$
\end{proposition}

\begin{proof}
This follows at once from the Theorem \ref{Theorem Imp 2}, since using the
Proposition \ref{Prop DI 1} we see that (\ref{Imp.5}) holds for $f_{\beta}$
with $\alpha=\beta-1.$
\end{proof}

\section{Convergence theorems}\label{Section Convergence Thms}

We shall now show that the usual convergence theorems, namely, the bounded
convergence theorem, the monotone convergence theorem, and Fatou's lemma are
valid for the distributional integral.

It is convenient to first introduce the notation for integrals that have an
infinite value. If $f$ is measurable in $\left[  a,b\right]  ,$ $f\left(
x\right)  \geq0$ almost everywhere, and $f$ is not integrable we put
\begin{equation}
\left(  \mathfrak{dist}\right)  \int_{a}^{b}f\left(  x\right)  \,\mathrm{d}%
x=\infty\,. \label{Conv.1}%
\end{equation}
More generally, we use (\ref{Conv.1}) if $f=f_{1}+f_{2},$ where $f_{1}$ is
distributionally integrable and $f_{2}$ is positive a.e. but not integrable.
The notation $\left(  \mathfrak{dist}\right)  \int_{a}^{b}g\left(  x\right)
\,\mathrm{d}x=-\infty$ is interpreted in a corresponding fashion.

Given a measurable function $f$ defined in $\left[  a,b\right]  $ then there
are three possibilities, namely, $f$ may be distributionally integrable, in
which case $\left(  \mathfrak{dist}\right)  \int_{a}^{b}f\left(  x\right)
\,\mathrm{d}x$ is a real number, maybe $\left(  \mathfrak{dist}\right)
\int_{a}^{b}f\left(  x\right)  \,\mathrm{d}x=\pm\infty,$ or maybe the
distributional integral is undefined. This is also the case for other
integrals, such as the Lebesgue integral or the Denjoy-Perron-Henstock-Kurzweil
integral. If $f$ is distributionally integrable but not Denjoy-Perron-Henstock-Kurzweil
integrable then the symbol $\left(  \mathfrak{DPHK}\right)  \int_{a}%
^{b}f\left(  x\right)  \,\mathrm{d}x$ is undefined, but, even more, if $f$ is
decomposed as $f=f_{1}+f_{2}$ and the symbols $\left(  \mathfrak{DPHK}%
\right)  \int_{a}^{b}f_{j}\left(  x\right)  \,\mathrm{d}x$ are defined for
$j=1$ or $2,$ then one of them is $+\infty$ and the other is $-\infty.$
Similarly, if the Lebesgue integral of $f$ is undefined, but $f$ is
Denjoy-Perron-Henstock-Kurzweil integrable then whenever $f=f_{1}+f_{2}$ and the
symbols $\left(  \mathfrak{Leb}\right)  \int_{a}^{b}f_{j}\left(  x\right)
\,\mathrm{d}x$ are defined for $j=1$ or $2,$ we must have that one of them is
$+\infty$ and the other is $-\infty.$ In a sense, going to a more general
integral means that a method to solve some indefinite forms $+\infty-\infty$
has been included in the definition of the more general and refined integral.

We now consider the following comparison results.\smallskip

\begin{proposition}
\label{Prop Conv 1}Let $f$ and $g$ be measurable on $\left[  a,b\right]  $ and
suppose that $f\left(  x\right)  \geq g\left(  x\right)  $ almost everywhere.
If $g$ is distributionally integrable, then $f$ is also distributionally
integrable or $\int_{a}^{b}f\left(  x\right)  \,\mathrm{d}x=\infty.$
Similarly, if $f$ is is distributionally integrable, then $g$ is
distributionally integrable, too, or $\int_{a}^{b}g\left(  x\right)
\,\mathrm{d}x=-\infty.$\smallskip
\end{proposition}

Observe that the proposition implies that if $g$ is distributionally
integrable and $f\geq g,$ then $\left(  \mathfrak{dist}\right)  \int_{a}%
^{b}f\left(  x\right)  \,\mathrm{d}x$ will always be defined, as a number in
$\mathbb{R}\cup\left\{  \infty\right\}  .$ If $f$ is just a measurable
function, without such an inequality, however, then $\left(  \mathfrak{dist}%
\right)  \int_{a}^{b}f\left(  x\right)  \,\mathrm{d}x$ would in general be
meaningless.\smallskip

\begin{proposition}
\label{Prop Conv 2}Let $f,$ $g$ and $h$ be measurable on $\left[  a,b\right]
$ and suppose that
\begin{equation}
f\left(  x\right)  \geq g\left(  x\right)  \geq h\left(  x\right)  \,,
\label{Conv.2}%
\end{equation}
almost everywhere. Suppose $f$ and $h$ are distributionally integrable. Then
$g$ is also distributionally integrable.

If, in addition, one of the functions is Lebesgue integrable, then so are the
other two. Similarly if one of the functions is Denjoy-Perron-Henstock-Kurzweil
integrable then the other two are as well.
\end{proposition}

\begin{proof}
If $f$ and $h$ are distributionally integrable, then so is $f-h,$ which being
positive, must be Lebesgue integrable. By comparison, $g-h$ is also Lebesgue
integrable. It follows that $g=h+\left(  g-h\right)  $ is distributionally
integrable. The second part is obtained directly from Theorem
\ref{Theorem COM 2}.\smallskip
\end{proof}

We now give the Bounded Convergence Theorem.\smallskip

\begin{theorem}
\label{Theorem Conv 1}Let $f$ and $h$ be distributionally integrable on
$\left[  a,b\right]  $ and suppose that $\left\{  g_{n}\right\}
_{n=1}^{\infty}$ is a sequence of distributionally integrable functions that
satisfies%
\begin{equation}
f\left(  x\right)  \geq g_{n}\left(  x\right)  \geq h\left(  x\right)  \,,
\label{Conv.3}%
\end{equation}
almost everywhere. If $g_{n}\rightarrow g$ almost everywhere then $g$ is
distributionally integrable and%
\begin{equation}
\lim_{n\rightarrow\infty}\left(  \mathfrak{dist}\right)  \int_{a}^{b}%
g_{n}\left(  x\right)  \,\mathrm{d}x=\left(  \mathfrak{dist}\right)  \int
_{a}^{b}g\left(  x\right)  \,\mathrm{d}x\,. \label{Conv.4}%
\end{equation}

\end{theorem}

\begin{proof}
Observe that $g$ also satisfies $f\left(  x\right)  \geq g\left(  x\right)
\geq h\left(  x\right)  $ almost everywhere, and thus the comparison result,
Proposition \ref{Prop Conv 2}, gives that $g$ is distributionally integrable.
Notice now that $\left\vert g\left(  x\right)  -g_{n}\left(  x\right)
\right\vert \leq f\left(  x\right)  -h\left(  x\right)  $ almost everywhere,
$f-h$ is Lebesgue integrable, and $\left\vert g-g_{n}\right\vert \rightarrow0$
almost everywhere. We conclude from the Lebesgue bounded convergence theorem
that $\int_{a}^{b}\left\vert g\left(  x\right)  -g_{n}\left(  x\right)
\right\vert \,\mathrm{d}x\rightarrow0,$ that is, that $\left\{  g-g_{n}%
\right\}  _{n=1}^{\infty}$ converges to $0$ in $L^{1}\left[  a,b\right]  ,$
and, in particular, that (\ref{Conv.4}) holds.\smallskip
\end{proof}

We also have a Monotone Convergence Theorem.\smallskip

\begin{theorem}
\label{Theorem Conv 2}Let $h$ be distributionally integrable on $\left[
a,b\right]  $ and let $\left\{  g_{n}\right\}  _{n=1}^{\infty}$ be a\ monotone
sequence of measurable functions that satisfies%
\begin{equation}
g_{n+1}\left(  x\right)  \geq g_{n}\left(  x\right)  \geq h\left(  x\right)
\,, \label{Conv.5}%
\end{equation}
almost everywhere. Let $g\left(  x\right)  =\lim_{n\rightarrow\infty}%
g_{n}\left(  x\right)  .$ Then $g$ is distributionally integrable if and only
if%
\begin{equation}
\lim_{n\rightarrow\infty}\left(  \mathfrak{dist}\right)  \int_{a}^{b}%
g_{n}\left(  x\right)  \,\mathrm{d}x<\infty\,, \label{Conv.6}%
\end{equation}
and if that is the case then $\left(  \mathfrak{dist}\right)  \int_{a}%
^{b}g_{n}\left(  x\right)  \,\mathrm{d}x\rightarrow\left(  \mathfrak{dist}%
\right)  \int_{a}^{b}g\left(  x\right)  \,\mathrm{d}x$ as $n\rightarrow
\infty.$
\end{theorem}

\begin{proof}
Suppose first that $g$ is distributionally integrable. Then $g\left(
x\right)  \geq g_{n}\left(  x\right)  \geq h\left(  x\right)  $ almost
everywhere, and from the bounded convergence theorem, Theorem
\ref{Theorem Conv 1}, we conclude that the increasing numerical sequence
$\left\{  \int_{a}^{b}g_{n}\left(  x\right)  \,\mathrm{d}x\right\}
_{n=1}^{\infty}$ converges to $\int_{a}^{b}g\left(  x\right)  \,\mathrm{d}x;$
(\ref{Conv.6}) follows.

Conversely, if (\ref{Conv.6}) holds, then $\left\{  g_{n}-h\right\}
_{n=1}^{\infty}$ is a Cauchy sequence in the space $L^{1}\left[  a,b\right]
,$ because
\[
\lim_{\substack{n,m\rightarrow\infty\\n\geq m}}\int_{a}^{b}\left\vert
(g_{n}\left(  x\right)  -h\left(  x\right)  )-(g_{m}\left(  x\right)
-h\left(  x\right)  )\right\vert \,\mathrm{d}x
\]%
\begin{align*}
\hspace{1.5in}  &  =\lim_{\substack{n,m\rightarrow\infty\\n\geq m}}\int
_{a}^{b}(g_{n}\left(  x\right)  -g_{m}\left(  x\right)  )\,\mathrm{d}x\\
\hspace{1.5in}  &  =\lim_{\substack{n,m\rightarrow\infty\\n\geq m}}\left(
\int_{a}^{b}g_{n}\left(  x\right)  \,\mathrm{d}x-\int_{a}^{b}g_{m}\left(
x\right)  \,\mathrm{d}x\right) \\
\hspace{1.5in}  &  =0\,.
\end{align*}
Since $\left\{  g_{n}-h\right\}  _{n=1}^{\infty}$ converges a.e. to $g-h,$ it
must also converge to $g-h$ in $L^{1}\left[  a,b\right]  ;$ we also obtain the
convergence of $\int_{a}^{b}(g_{n}\left(  x\right)  -h\left(  x\right)
)\,\mathrm{d}x$ to $\int_{a}^{b}(g\left(  x\right)  -h\left(  x\right)
)\,\mathrm{d}x.$ Thus $\int_{a}^{b}(g\left(  x\right)  -h\left(  x\right)
)\,\mathrm{d}x<\infty,$ and so $g=\left(  g-h\right)  +h$ is distributionally
integrable. The convergence of $\int_{a}^{b}g_{n}\left(  x\right)
\,\mathrm{d}x$ to $\int_{a}^{b}g\left(  x\right)  \,\mathrm{d}x$ is now
clear.\smallskip
\end{proof}

Observe that the Monotone Convergence Theorem also says that if $g$ is not
distributionally integrable, so that $\left(  \mathfrak{dist}\right)  \int
_{a}^{b}g\left(  x\right)  \,\mathrm{d}x=\infty,$ then $\left(
\mathfrak{dist}\right)  \int_{a}^{b}g_{n}\left(  x\right)  \,\mathrm{d}%
x\nearrow\infty.$\smallskip

Fatou's lemma takes the following form.\smallskip

\begin{theorem}
\label{Theorem Conv 3}Let $h$ be distributionally integrable on $\left[
a,b\right]  $ and let $\left\{  g_{n}\right\}  _{n=1}^{\infty}$ be a\ sequence
of measurable functions that satisfies%
\begin{equation}
g_{n}\left(  x\right)  \geq h\left(  x\right)  \,, \label{Conv.7}%
\end{equation}
almost everywhere. Suppose that%
\begin{equation}
\liminf_{n\rightarrow\infty}\left(  \mathfrak{dist}\right)  \int_{a}^{b}%
g_{n}\left(  x\right)  \,\mathrm{d}x<\infty\,. \label{Conv.8}%
\end{equation}
Then the function defined by $g_{\ast}\left(  x\right)  =\liminf
_{n\rightarrow\infty}g_{n}\left(  x\right)  $ is distributionally integrable
and%
\begin{equation}
\left(  \mathfrak{dist}\right)  \int_{a}^{b}g_{\ast}\left(  x\right)
\,\mathrm{d}x\leq\liminf_{n\rightarrow\infty}\left(  \mathfrak{dist}\right)
\int_{a}^{b}g_{n}\left(  x\right)  \,\mathrm{d}x\,. \label{Conv.9}%
\end{equation}

\end{theorem}

\begin{proof}
Let $h_{n}\left(  x\right)  =\inf\left\{  g_{j}\left(  x\right)  :n\leq
j<\infty\right\}  .$ Then $h\leq h_{n}\leq g_{j}$ for $n\leq j,$ and since
(\ref{Conv.8}) implies that for each $n$ there are indices $j$\ with $n\leq j$
such that $g_{j}$ is distributionally integrable, it follows that $h_{n}$ is
distributionally integrable for all $n.$ Notice also that the sequence
$\left\{  h_{n}\right\}  $ is increasing. Since $\int_{a}^{b}h_{n}\left(
x\right)  \,\mathrm{d}x\leq\int_{a}^{b}g_{n}\left(  x\right)  \,\mathrm{d}x$
we obtain that
\[
\lim_{n\rightarrow\infty}\int_{a}^{b}h_{n}\left(  x\right)  \,\mathrm{d}%
x\leq\liminf_{n\rightarrow\infty}\int_{a}^{b}g_{n}\left(  x\right)
\,\mathrm{d}x<\infty\,.
\]
If we now use the fact that $g_{\ast}\left(  x\right)  =\lim_{n\rightarrow
\infty}h_{n}\left(  x\right)  $ and Theorem \ref{Theorem Conv 2}, we obtain
that $g_{\ast}$ is distributionally integrable and%
\[
\int_{a}^{b}g_{\ast}\left(  x\right)  \,\mathrm{d}x=\lim_{n\rightarrow\infty
}\int_{a}^{b}h_{n}\left(  x\right)  \,\mathrm{d}x\leq\liminf_{n\rightarrow
\infty}\int_{a}^{b}g_{n}\left(  x\right)  \,\mathrm{d}x\,,
\]
as required.\smallskip
\end{proof}

It is interesting to observe that if a sequence of integrable distributions
$\left\{  \mathsf{f}_{n}\right\}  _{n=1}^{\infty}$ of the space $\mathcal{E}%
^{\prime}\left[  a,b\right]  $ converges distributionally to $\mathsf{f},$ and
$\mathsf{f}$ is integrable, then trivially%
\begin{equation}
\lim_{n\rightarrow\infty}\left(  \mathfrak{dist}\right)  \int_{a}^{b}%
f_{n}\left(  x\right)  \,\mathrm{d}x=\left(  \mathfrak{dist}\right)  \int
_{a}^{b}f\left(  x\right)  \,\mathrm{d}x\,, \label{Conv.10}%
\end{equation}
where $f\leftrightarrow\mathsf{f,}$ $f_{n}\leftrightarrow\mathsf{f}_{n},$
since $\left\langle \mathsf{f}_{n},\psi\right\rangle \rightarrow\left\langle
\mathsf{f},\psi\right\rangle $ for all test functions $\psi.$ However, in
general, if $\left\{  \mathsf{f}_{n}\right\}  _{n=1}^{\infty}$ converges
distributionally to $\mathsf{f,}$ then $\mathsf{f}$ does not have to be
integrable. Actually, $\left\{  f_{n}\right\}  _{n=1}^{\infty}$ could even
converge $\left(  \text{a.e.}\right)  $ to a function $f,$ but $f$ and
$\mathsf{f}$ cannot be associated if $\mathsf{f}$ is not integrable;
(\ref{Conv.10}) may or may not hold in such a case. For example, if
$I_{n}=\left[  1/n,2/n\right]  ,$ and $f_{n}=n\chi_{I_{n}},$ then $\left\{
\mathsf{f}_{n}\right\}  _{n=1}^{\infty}$ converges distributionally to
$\delta\left(  x\right)  ,$ but $\left\{  f_{n}\right\}  _{n=1}^{\infty}$
converges everywhere to $f=0.$ Naturally (\ref{Conv.10}) does not hold if
$a<0<b.$

\section{Change of variables}\label{Section Change of Variables}

We now consider changes of variables in the integral. Let us start with a
function
\begin{equation}
\rho:\left[  c,d\right]  \rightarrow\left[  a,b\right]  \,, \label{ChV.1}%
\end{equation}
that is of class $C^{\infty},$ even at the endpoints, and satisfies
$\left\vert \rho^{\prime}\left(  t\right)  \right\vert >0$ for all
$t\in\left[  c,d\right]  .$ Then $\rho$ induces an isomorphism
\begin{equation}
T_{\rho}:\mathcal{E}^{\prime}\left[  a,b\right]  \rightarrow\mathcal{E}%
^{\prime}\left[  c,d\right]  \,, \label{ChV.2}%
\end{equation}
given by $T_{\rho}\left\{  \mathsf{f}\right\}  \left(  t\right)
=\mathsf{f}\left(  \rho\left(  t\right)  \right)  $ for $\mathsf{f}%
\in\mathcal{E}^{\prime}\left[  a,b\right]  .$ Observe that $T_{\rho}%
^{-1}=T_{\rho^{-1}}.$ In this case we say that $\rho$ is a change of variables
of type I.

For changes of type I it is easy to see \cite{CF,loj} that the distributional
point value $\mathsf{f}\left(  x\right)  $ exists at $x=x_{0}$ if and only if
the distributional point value $\mathsf{f}\left(  \rho\left(  t\right)
\right)  $ exists at $t=t_{0},$ where $x_{0}=\rho\left(  t_{0}\right)  ,$ and
when both values exist they coincide. Also, if $f$ is a function defined in
$\left[  a,b\right]  $, $\rho^{\prime}>0$, and $\left(  \mathsf{u}%
,\mathsf{U}\right)  $ is a major distributional pair for $f$ in $\left[
a,b\right]  ,$ then $\left(  \mathsf{u}\left(  \rho\left(  t\right)  \right)
\rho^{\prime}\left(  t\right)  ,\mathsf{U}\left(  \rho\left(  t\right)
\right)  \right)  $ is a major pair for $f\left(  \rho\left(  t\right)
\right)  \rho^{\prime}\left(  t\right)  $ in $\left[  c,d\right]  ,$ and
similarly for minor pairs. Thus, we immediately obtain the following
result.\smallskip

\begin{proposition}
\label{Prop. ChV 1}Let $\rho:\left[  c,d\right]  \rightarrow\left[
a,b\right]  $ be a change of variables of type I. A function $f$ is
distributionally integrable over $\left[  a,b\right]  $ if and only if
$f\left(  \rho\left(  t\right)  \right)  \rho^{\prime}\left(  t\right)  $ is
distributionally integrable over $\left[  c,d\right]  $ and if $a=\rho\left(
c\right)  ,$ $b=\rho\left(  d\right)  ,$%
\begin{equation}
\left(  \mathfrak{dist}\right)  \int_{a}^{b}f\left(  x\right)  \,\mathrm{d}%
x=\left(  \mathfrak{dist}\right)  \int_{c}^{d}f\left(  \rho\left(  t\right)
\right)  \rho^{\prime}\left(  t\right)  \,\mathrm{d}t\,. \label{ChV.3}%
\end{equation}
\smallskip
\end{proposition}

The change of variables formula (\ref{ChV.3}) remains valid under more general
conditions on the function $\rho.$ It holds for $\rho\left(  t\right)
=t^{\alpha}$ in $\left[  0,d\right]  ,$ if $\alpha>0;$ this change of
variables is not of type I.\smallskip

\begin{lemma}
\label{Lemma ChV 1}Let $\mathsf{f}\in\mathcal{D}^{\prime}\left(  0,b\right)  $
and let $\alpha>0$. Then the distributional limit of $\mathsf{f}$ from the
right at $x=0$ exists and equals $\gamma$ if and only if the distributional
limit of $\mathsf{f}\left(  t^{\alpha}\right)  $ from the right at $t=0$
exists and equals $\gamma.$

A function $f$ is distributionally integrable over $\left[  0,d\right]  $ if
and only if $f\left(  t^{\alpha}\right)  t^{\alpha-1}$ is distributionally
integrable over $\left[  0,d^{1/\alpha}\right]  $ and
\begin{equation}
\left(  \mathfrak{dist}\right)  \int_{0}^{d}f\left(  x\right)  \,\mathrm{d}%
x=\left(  \mathfrak{dist}\right)  \int_{0}^{d^{1/\alpha}}\alpha f\left(
t^{\alpha}\right)  t^{\alpha-1}\,\mathrm{d}t\,. \label{ChV.4}%
\end{equation}

\end{lemma}

\begin{proof}
The distributional limit of $\mathsf{f}$ from the right at $x=0$ exists and
equals $\gamma$ if and only%
\begin{equation}
\lim_{\varepsilon\rightarrow0^{+}}\left\langle \mathsf{f}\left(  \varepsilon
x\right)  ,\psi\left(  x\right)  \right\rangle =\gamma\int_{0}^{\infty}%
\psi\left(  x\right)  \,\mathrm{d}x\,, \label{ChV.5}%
\end{equation}
for all $\psi\in\mathcal{D}\left(  0,\infty\right)  .$ But if (\ref{ChV.5})
holds then%
\begin{align*}
\lim_{\varepsilon\rightarrow0^{+}}\left\langle \mathsf{f}\left(  \left(
\varepsilon t\right)  ^{\alpha}\right)  ,\psi\left(  t\right)  \right\rangle
&  =\frac{1}{\alpha}\lim_{\varepsilon\rightarrow0^{+}}\left\langle
\mathsf{f}\left(  \varepsilon^{\alpha}x\right)  ,\psi\left(  x^{1/\alpha
}\right)  x^{1/\alpha-1}\right\rangle \\
&  =\frac{\gamma}{\alpha}\int_{0}^{\infty}\psi\left(  x^{1/\alpha}\right)
x^{1/\alpha-1}\,\mathrm{d}x\\
&  =\gamma\int_{0}^{\infty}\psi\left(  t\right)  \,\mathrm{d}t\,,
\end{align*}
and it follows that the distributional limit of $\mathsf{f}\left(  t^{\alpha
}\right)  $ from the right at $t=0$ exists and equals $\gamma.$

Suppose now that $f$ is distributionally integrable over $\left[  0,d\right]
.$ Using the Proposition \ref{Prop. ChV 1} we obtain that $f\left(  t^{\alpha
}\right)  t^{\alpha-1}$ is distributionally integrable over $\left[
c,d^{1/\alpha}\right]  $ for any $c>0$ and if $F\left(  c\right)  =\left(
\mathfrak{dist}\right)  \int_{c}^{d}f\left(  x\right)  \,\mathrm{d}x$ is the
indefinite integral of $f,$ then the indefinite integral of $\alpha f\left(
t^{\alpha}\right)  t^{\alpha-1}$ is $F\left(  c^{\alpha}\right)  =\left(
\mathfrak{dist}\right)  \int_{c}^{d^{1/\alpha}}\alpha f\left(  t^{\alpha
}\right)  t^{\alpha-1}\,\mathrm{d}t\,.$ Now $\mathsf{F}\left(  c\right)
\mathsf{,}$ where $\mathsf{F}\leftrightarrow F,$ has a distributional limit
from the right at $c=0,$ equal to $\left(  \mathfrak{dist}\right)  \int
_{0}^{d^{1/\alpha}}f\left(  x\right)  \,\mathrm{d}x,$ and it follows that the
distributional limit of $\mathsf{F}\left(  c^{\alpha}\right)  ,$ which
corresponds to the function $\left(  \mathfrak{dist}\right)  \int
_{c}^{d^{1/\alpha}}\alpha f\left(  t^{\alpha}\right)  t^{\alpha-1}%
\,\mathrm{d}t,$ as $c\rightarrow0^{+}$ also exists and equals $\left(
\mathfrak{dist}\right)  \int_{0}^{d}f\left(  x\right)  \,\mathrm{d}x.$ The
integrability of $\alpha f\left(  t^{\alpha}\right)  t^{\alpha-1}$ over
$\left[  0,d^{1/\alpha}\right]  $ and formula (\ref{ChV.4}) then follow from
the Theorem \ref{Theorem Imp 1}.\smallskip
\end{proof}

Introduce the changes of variables of type II as those continuous functions
$\rho$ from $\left[  c,d\right]  $ to $\left[  a,b\right]  $ that are of type
I in $\left[  c_{1},d_{1}\right]  $ whenever $c<c_{1}<d_{1}<d,$ and that at
the endpoints satisfy that there exist $\alpha>0$ and $\beta>0$ such that
$\left\vert \rho\left(  x\right)  -\rho\left(  c\right)  \right\vert ^{\alpha
}$ is of type I in $\left[  c,d_{1}\right]  $ and $\left\vert \rho\left(
x\right)  -\rho\left(  d\right)  \right\vert ^{\beta}$ is of type I in
$\left[  c_{1},d\right]  .$ Then using the Lemma \ref{Lemma ChV 1}\ we obtain
that the Proposition \ref{Prop. ChV 1} also holds for changes of variables of
type II.

Actually the Proposition \ref{Prop. ChV 1} remains valid for changes of
variables of type III, which are those continuous functions $\rho$ from
$\left[  c,d\right]  $ to $\mathbb{R}$ for which there are numbers $\left\{
c_{j}\right\}  _{j=0}^{n}$ with $c=c_{0}<c_{1}<\cdots<c_{n-1}<c_{n}=d$ such
that $\rho$ is of type II in each of the subintervals $\left[  c_{j}%
,c_{j+1}\right]  $ for $0\leq j<n.$\smallskip

\begin{proposition}
\label{Prop. ChV 2}Let $\rho:\left[  c,d\right]  $ be a change of variables of
type III. A function $f$ is distributionally integrable over $\rho\left(
\left[  c,d\right]  \right)  $ if and only if $f\left(  \rho\left(  t\right)
\right)  \rho^{\prime}\left(  t\right)  $ is distributionally integrable over
$\left[  c,d\right]  $ and
\begin{equation}
\left(  \mathfrak{dist}\right)  \int_{\rho\left(  c\right)  }^{\rho\left(
d\right)  }f\left(  x\right)  \,\mathrm{d}x=\left(  \mathfrak{dist}\right)
\int_{c}^{d}f\left(  \rho\left(  t\right)  \right)  \rho^{\prime}\left(
t\right)  \,\mathrm{d}t\,.
\end{equation}
\smallskip
\end{proposition}

Let us now consider changes of variables with an infinite range or
domain.\smallskip

\begin{lemma}
\label{Lemma ChV 2}Let $\mathsf{f}\in\mathcal{D}^{\prime}\left(  0,1\right)
.$ Then the distributional limit of $\mathsf{f}$ from the right at $x=0$
exists and equals $\gamma$ if and only if the Ces\`{a}ro limit of
$\mathsf{f}\left(  1/t\right)  $ as $t\rightarrow\infty$ exists and equals
$\gamma.$

A function $f$ is distributionally integrable over $\left[  0,1\right]  $ if
and only if $f\left(  t^{-1}\right)  t^{-2}$ is distributionally Ces\`{a}ro
integrable over $[1,\infty)$ and%
\begin{equation}
\left(  \mathfrak{dist}\right)  \int_{0}^{1}f\left(  x\right)  \,\mathrm{d}%
x=\left(  \mathfrak{dist}\right)  \int_{1}^{\infty}f\left(  t^{-1}\right)
t^{-2}\,\mathrm{d}t\ \ \ \ \left(  \mathrm{C}\right)  \,. \label{ChV.6}%
\end{equation}

\end{lemma}

\begin{proof}
The first part follows from the results of \cite[Chap. 6]{est-kan}. The second
part is obtained because using the Proposition \ref{Prop. ChV 1} $f$ is
distributionally integrable over $\left[  c,1\right]  $ if and only if
$f\left(  t^{-1}\right)  t^{-2}$ is distributionally integrable over $[1,1/c]$
and if $F\left(  c\right)  =\left(  \mathfrak{dist}\right)  \int_{c}%
^{1}f\left(  x\right)  \,\mathrm{d}x$ then $F\left(  1/c\right)  =\left(
\mathfrak{dist}\right)  \int_{1}^{1/c}f\left(  t^{-1}\right)  t^{-2}%
\,\mathrm{d}t.$ But the distribution corresponding to $F\left(  c\right)  $
has a distributional lateral limit from the right at $c=0$ if and only if the
distribution corresponding to $F\left(  1/c\right)  $ has a Ces\`{a}ro limit
at infinity, and the limits coincide; in the first case the integral $\left(
\mathfrak{dist}\right)  \int_{0}^{1}f\left(  x\right)  \,\mathrm{d}x$ exists,
while in the second the Ces\`{a}ro integral $\left(  \mathfrak{dist}\right)
\int_{1}^{\infty}f\left(  t^{-1}\right)  t^{-2}\,\mathrm{d}t$ $\left(
\text{C}\right)  $ exists.\smallskip
\end{proof}

Let us say that a function $\rho:[c,d)\rightarrow\lbrack a,\infty),$ with
$\lim_{x\rightarrow d^{-}}\rho\left(  x\right)  =\infty,$ is a change of
variables of type IV if whenever $c<x<d,$ $\rho$ is of type II in $\left[
c,x\right]  $ and $1/\rho$ is of type II in $\left[  x,d\right]  .$ Then our
previous results immediately yield the ensuing change of variables
formula.\smallskip

\begin{proposition}
\label{Prop. ChV 3}Let $\rho:[c,d)\rightarrow\lbrack a,\infty)$ be a change of
variables of type IV. A function is Ces\`{a}ro distributionally integrable
over $[a,\infty)$ if and only if $f\left(  \rho\left(  t\right)  \right)
\rho^{\prime}\left(  t\right)  $ is distributionally integrable over $\left[
c,d\right]  $ and
\begin{equation}
\left(  \mathfrak{dist}\right)  \int_{\rho\left(  c\right)  }^{\infty}f\left(
x\right)  \,\mathrm{d}x=\left(  \mathfrak{dist}\right)  \int_{c}^{d}f\left(
\rho\left(  t\right)  \right)  \rho^{\prime}\left(  t\right)  \,\mathrm{d}%
t\ \ \ \ \left(  \mathrm{C}\right)  \,. \label{ChV.7}%
\end{equation}

\end{proposition}

\section{Mean value theorems}\label{Section MVT}

In this section we shall show how the usual three mean value theorems of
integral calculus have versions for the distributional integral.\smallskip

\begin{proposition}
\label{Prop. MVT 1}Let $f$ be a \L ojasiewicz function on $\left[  a,b\right]
$ and let $\psi$ be smooth and positive in $\left[  a,b\right]  .$ Then there
exists $\xi\in\left(  a,b\right)  $ such that%
\begin{equation}
\left(  \mathfrak{dist}\right)  \int_{a}^{b}f\left(  x\right)  \psi\left(
x\right)  \,\mathrm{d}x=f\left(  \xi\right)  \int_{a}^{b}\psi\left(  x\right)
\,\mathrm{d}x\,. \label{MVT.1}%
\end{equation}

\end{proposition}

\begin{proof}
Observe that since $\psi$ is $C^{\infty}$ then $f\psi$ is also a \L ojasiewicz
function, and thus integrable. Since $\psi\geq0,$ it follows that
\begin{equation}
m\int_{a}^{b}\psi\left(  x\right)  \,\mathrm{d}x\leq\left(  \mathfrak{dist}%
\right)  \int_{a}^{b}f\left(  x\right)  \psi\left(  x\right)  \,\mathrm{d}%
x\leq M\int_{a}^{b}\psi\left(  x\right)  \,\mathrm{d}x\,, \label{MVT.2}%
\end{equation}
where $m=\inf\left\{  f\left(  x\right)  :x\in\left[  a,b\right]  \right\}  ,$
$M=\sup\left\{  f\left(  x\right)  :x\in\left[  a,b\right]  \right\}  .$
Naturally $m$ and $M$ do not have to be real numbers in this case,
$-\infty\leq m\leq M\leq\infty.$ Notice now \cite{loj}\ that any \L ojasiewicz
function satisfies the Darboux or intermediate value property, that is,
$\left[  f\left(  c\right)  ,f\left(  d\right)  \right]  \subset f\left(
\left[  c,d\right]  \right)  $ for any subinterval $\left[  c,d\right]  $ of
$\left[  a,b\right]  .$ Hence there exists $\xi\in\left(  a,b\right)  $ such
that $f\left(  \xi\right)  =(\int_{a}^{b}f\left(  x\right)  \psi\left(
x\right)  \,\mathrm{d}x)/(\int_{a}^{b}\psi\left(  x\right)  \,\mathrm{d}%
x).$\smallskip
\end{proof}

We also have the following second mean value theorem.\smallskip

\begin{proposition}
\label{Prop. MVT 2}Let $f$ be distributionally integrable over $\left[
a,b\right]  $ and let $\psi$ be smooth and monotonic. Then there exists
$\xi\in\left(  a,b\right)  $ such that%
\begin{equation}
\int_{a}^{b}f\left(  x\right)  \psi\left(  x\right)  \,\mathrm{d}x=\psi\left(
a\right)  \int_{a}^{\xi}f\left(  x\right)  \,\mathrm{d}x+\psi\left(  b\right)
\int_{\xi}^{b}f\left(  x\right)  \,\mathrm{d}x\,. \label{MVT.3}%
\end{equation}

\end{proposition}

\begin{proof}
Let $F\left(  x\right)  =\left(  \mathfrak{dist}\right)  \int_{a}^{x}f\left(
t\right)  \,\mathrm{d}t$ be the indefinite integral of $f.$ Then applying the
Proposition \ref{Prop. MVT 1} to $\int_{a}^{b}F\left(  x\right)  \psi^{\prime
}\left(  x\right)  \,\mathrm{d}x$ we obtain the existence of $\xi\in\left(
a,b\right)  $ such that%
\begin{align*}
\int_{a}^{b}f\left(  x\right)  \psi\left(  x\right)  \,\mathrm{d}x  &
=F\left(  b\right)  \psi\left(  b\right)  -\int_{a}^{b}F\left(  x\right)
\psi^{\prime}\left(  x\right)  \,\mathrm{d}x\\
&  =F\left(  b\right)  \psi\left(  b\right)  -F\left(  \xi\right)  \int
_{a}^{b}\psi^{\prime}\left(  x\right)  \,\mathrm{d}x\\
&  =\psi\left(  b\right)  \left(  F\left(  b\right)  -F\left(  \xi\right)
\right)  +\psi\left(  a\right)  F\left(  \xi\right)  \,,
\end{align*}
as required.\smallskip
\end{proof}

The Bonnet form of the mean value theorem is as follows.\smallskip

\begin{proposition}
\label{Prop. MVT 3}Let $f$ be distributionally integrable over $\left[
a,b\right]  $ and let $\psi$ be smooth, positive, and increasing. Then there
exists $\xi\in\left(  a,b\right)  $ such that%
\begin{equation}
\left(  \mathfrak{dist}\right)  \int_{a}^{b}f\left(  x\right)  \psi\left(
x\right)  \,\mathrm{d}x=\psi\left(  b\right)  \left(  \mathfrak{dist}\right)
\int_{\xi}^{b}f\left(  x\right)  \,\mathrm{d}x\,. \label{MVT.4}%
\end{equation}

\end{proposition}

\begin{proof}
Let $a^{\prime}<a,$ and extend $f$ and $\psi$ to $\left[  a^{\prime},b\right]
$ as follows. Put $f\left(  x\right)  =0$ for $a^{\prime}\leq x<a,$ and let
$\psi$ be an extension that is smooth, positive, increasing, and with
$\psi\left(  a^{\prime}\right)  =0.$ Then employing Proposition
\ref{Prop. MVT 2} for the integral $\left(  \mathfrak{dist}\right)
\int_{a^{\prime}}^{b}f\left(  x\right)  \psi\left(  x\right)  \,\mathrm{d}x$
we obtain (\ref{MVT.4}).\smallskip
\end{proof}

We also have the other form of the Bonnet mean value theorem, namely, if $f$
is distributionally integrable over $\left[  a,b\right]  $ and $\psi$ is
smooth, positive, and \emph{decreasing,} then there exists $\omega\in\left(
a,b\right)  $ such that%
\begin{equation}
\left(  \mathfrak{dist}\right)  \int_{a}^{b}f\left(  x\right)  \psi\left(
x\right)  \,\mathrm{d}x=\psi\left(  a\right)  \left(  \mathfrak{dist}\right)
\int_{a}^{\omega}f\left(  x\right)  \,\mathrm{d}x\,. \label{MVT.5}%
\end{equation}

We now give an example that shows how many standard arguments can still be
used with the distributional integral.\smallskip

\begin{example}
Let $a>0,$ and let $f$ be a locally distributionally integrable function
defined in $[a,\infty).$ Let us suppose that the indefinite integral of $f,$
$F\left(  x\right)  =\left(  \mathfrak{dist}\right)  \int_{a}^{x}f\left(
t\right)  \,\mathrm{d}t,$ is a bounded function in $[a,\infty).$ Let now
$\psi$ be a $C^{\infty}$ function defined in $[a,\infty)$ that decreases to
$0,$ $\lim_{x\rightarrow\infty}\psi\left(  x\right)  =0.$ \emph{Then the
improper distributional integral}
\begin{equation}
\left(  \mathfrak{dist}\right)  \int_{a}^{\infty}f\left(  x\right)
\psi\left(  x\right)  \,\mathrm{d}x=\lim_{b\rightarrow\infty}\int_{a}%
^{b}f\left(  x\right)  \psi\left(  x\right)  \,\mathrm{d}x\,, \label{MVT.6}%
\end{equation}
\emph{converges.}

Naturally, this is a well known result for locally Lebesgue integrable
functions, but our aim is to show that the usual ideas also work for the
distributional integral. Indeed, let%
\begin{equation}
G\left(  b\right)  =\left(  \mathfrak{dist}\right)  \int_{a}^{b}f\left(
x\right)  \psi\left(  x\right)  \,\mathrm{d}x\,. \label{MVT.7}%
\end{equation}
Then if $a\leq b<b^{\prime}$ there exists $\xi\in\left(  b,b^{\prime}\right)
$ such that%
\[
G\left(  b^{\prime}\right)  -G\left(  b\right)  =\psi\left(  b\right)  \left(
F\left(  \xi\right)  -F\left(  b\right)  \right)  +\psi\left(  b^{\prime
}\right)  \left(  F\left(  b^{\prime}\right)  -F\left(  \xi\right)  \right)
\,,
\]
and since $F$ is bounded and $\psi$ tends to $0,$ we conclude that $G\left(
b^{\prime}\right)  -G\left(  b\right)  \rightarrow0$ as $b,b^{\prime
}\rightarrow\infty,$ that is, $G$ satisfies the Cauchy criterion at $\infty.$
Thus $G$ has a limit at infinity, as we wanted to show.
\end{example}

\section{Examples}\label{Section Examples}

We shall now give several examples of functions that are or are not
distributionally integrable. We shall also consider several illustrations of
our ideas.

\begin{example}
Observe, first of all, that for positive functions distributional integration
is equivalent to Lebesgue integration, and thus nothing new arises in this
case. Let $\alpha\in\mathbb{R}.$ The positive function $\left\vert
x\right\vert ^{\alpha}$ is distributionally integrable over $\mathbb{R}$ if
and only if $\alpha>-1;$ the same is true for the function $x^{\alpha}H\left(
x\right)  .$ There is a well defined and well known distribution,
$x_{+}^{\alpha},$ whenever $\alpha\neq-1,-2,-3,\ldots;$ the distribution
$x_{+}^{\alpha}$ is a regularization of the function $x^{\alpha}H\left(
x\right)  ,$ but the association $x^{\alpha}H\left(  x\right)  \leftrightarrow
x_{+}^{\alpha}$ between an \emph{integrable} function and the corresponding
distribution holds only for $\alpha>-1.$\smallskip
\end{example}

Functions that are distributionally integrable but not Lebesgue integrable
need to be oscillatory.\smallskip

\begin{example}
Let us now consider the distribution $\mathsf{s}_{\alpha}\left(  x\right)  $
that corresponds to the function $\left\vert x\right\vert ^{\alpha}\sin\left(  1/x\right)
$ for $\alpha\in\mathbb{C}.$ If $\Re e\,\alpha>-1$ then the function
$\left\vert x\right\vert ^{\alpha}\sin\left(  1/x\right)  $ is locally
Lebesgue integrable and thus it yields a regular distribution given by
\begin{equation}
\left\langle \mathsf{s}_{\alpha}\left(  x\right)  ,\psi\left(  x\right)
\right\rangle =\int_{-\infty}^{\infty}\left\vert x\right\vert ^{\alpha}%
\sin\left(  1/x\right)  \,\psi\left(  x\right)  \,\mathrm{d}x\,,\ \ \psi
\in\mathcal{D}\left(  \mathbb{R}\right)  \,. \label{Exa.1}%
\end{equation}
It is easy to show that $\mathsf{s}_{\alpha}$ admits an analytic continuation
from the right side half-plane $\Re e\,\alpha>-1$ to the whole complex plane.
If $-1\geq\Re e\,\alpha>-2$ then the function $\left\vert x\right\vert
^{\alpha}\sin\left(  1/x\right)  $ is not Lebesgue integrable near $x=0$ but
it is Denjoy-Perron-Henstock-Kurzweil integrable. The function $\left\vert x\right\vert
^{\alpha}\sin\left(  1/x\right)  $ is locally distributionally integrable for
all $\alpha\in\mathbb{C},$ since actually it is a \L ojasiewicz function
because \cite{loj}\ the distributional value $\mathsf{s}_{\alpha}\left(
0\right)  $ exists and equals $0$ for all $\alpha.$ The association
$\left\vert x\right\vert ^{\alpha}\sin\left(  1/x\right)  \leftrightarrow
\mathsf{s}_{\alpha}\left(  x\right)  $ holds $\forall\alpha\in\mathbb{C}.$

Similarly, $\left\vert x\right\vert ^{\alpha}\cos\left(  1/x\right)  $ is
locally distributionally integrable for all $\alpha\in\mathbb{C}$ and it thus
defines a distribution $\mathsf{c}_{\alpha}\left(  x\right)  $ given by%
\begin{equation}
\left\langle \mathsf{c}_{\alpha}\left(  x\right)  ,\psi\left(  x\right)
\right\rangle =\left(  \mathfrak{dist}\right)  \int_{-\infty}^{\infty
}\left\vert x\right\vert ^{\alpha}\cos\left(  1/x\right)  \,\psi\left(
x\right)  \,\mathrm{d}x\,, \label{Exa.2}%
\end{equation}
for $\psi\in\mathcal{D}\left(  \mathbb{R}\right)  .$ The generalized function
$\mathsf{c}_{\alpha}$ is an entire function of $\alpha.$\smallskip
\end{example}

\begin{example}
Making a change of variables, we obtain that the functions $\left\vert
x\right\vert ^{\alpha}\sin\left\vert x\right\vert ^{-\beta}$ and $\left\vert
x\right\vert ^{\alpha}\cos\left\vert x\right\vert ^{-\beta}$ are locally
distributionally integrable for all $\alpha\in\mathbb{C}$ and for all
$\beta>0.$ We can multiply locally distributionally integrable functions by
the characteristic functions of intervals and still obtain locally
distributionally integrable functions. Thus $x^{\alpha}H\left(  x\right)
\sin\left\vert x\right\vert ^{-\beta},$ $x^{\alpha}H\left(  x\right)
\cos\left\vert x\right\vert ^{-\beta},$ as well as $\left\vert x\right\vert
^{\alpha}\operatorname*{sgn}x\sin\left\vert x\right\vert ^{-\beta}$ and
$\left\vert x\right\vert ^{\alpha}\operatorname*{sgn}x\cos\left\vert
x\right\vert ^{-\beta}$ are also locally distributionally integrable functions
for all $\alpha\in\mathbb{C}$ and for all $\beta>0.$\smallskip
\end{example}

\begin{example}
Let $g$ be a locally distributionally integrable function, $g\leftrightarrow
\mathsf{g,}$ where $\mathsf{g}\in\mathcal{K}^{\prime}\left(  \mathbb{R}%
\right)  .$ Then a change of variables show that for each $a\in\mathbb{R}$ the
function $f\left(  x\right)  =g\left(  \left(  x-a\right)  ^{-1}\right)  $ is
likewise locally distributionally integrable over $\mathbb{R},$ even at $x=a.$

Functions like $H\left(  x-a\right)  \left(  x-a\right)  ^{\alpha}J_{\nu
}\left(  \left(  x-a\right)  ^{-\beta}\right)  $ will be locally
distributionally integrable over $\mathbb{R}$ for all $\alpha\in\mathbb{C},$
$\beta>0,$ and $\nu\in\mathbb{R}.$

In particular, if $h$ is a locally distributionally integrable function,
periodic, and with zero mean, then $\mathsf{h}\in\mathcal{K}^{\prime}\left(
\mathbb{R}\right)  ,$ where $h\leftrightarrow\mathsf{h.}$ Thus functions like
$\left\vert x-a\right\vert ^{\alpha}h\left(  \left\vert x-a\right\vert
^{-\beta}\right)  $ are locally distributionally integrable over $\mathbb{R}$
for all $\alpha\in\mathbb{C}$ and $\beta>0.$ If $h\left(  x\right)  $ is $\sin
x$ or $\cos x$ we recover the previous examples; another example is
$\left\vert x\right\vert ^{\alpha}\left(  \left\{  \left\vert x\right\vert
^{-\beta}\right\}  -1/2 \right)  ,$ where $\left\{  x\right\}  $ is the
fractional part of $x.$\smallskip
\end{example}

\begin{example}
Let $\left\{  c_{n}\right\}  _{n=1}^{\infty}$ be a numerical sequence, and
define the function%
\begin{equation}
f\left(  x\right)  =\left\{
\begin{array}
[c]{cc}%
0\,, & \text{if \ }x\leq0\text{ \ or \ }x\geq1\,,\\
& \\
c_{n}\,, & \text{if \ }\frac{1}{n+1}\leq x<\frac{1}{n}\,.
\end{array}
\right.  \label{Exa.3}%
\end{equation}
Let $a_{n}=c_{n}\left(  \frac{1}{n}-\frac{1}{n+1}\right)  .$ Then $f$ is
Lebesgue integrable at $x=0$ if and only if the series $\sum_{n=1}^{\infty
}a_{n}$ is absolutely convergent, while $f$ is Denjoy-Perron-Henstock-Kurzweil
integrable if and only if the series is convergent \cite{Bartle}. We now have
that $f$ is distributionally integrable at $x=0$ if and only if $\sum
_{n=1}^{\infty}a_{n}$ is Ces\`{a}ro summable, and in that case%
\begin{equation}
\left(  \mathfrak{dist}\right)  \int_{0}^{1}f\left(  x\right)  \,\mathrm{d}%
x=\sum_{n=1}^{\infty}a_{n}\ \ \ \ \ \left(  \text{C}\right)  \,. \label{Exa.4}%
\end{equation}
In particular, if $c_{n}=\left(  -1\right)  ^{n}n(n+1),$ we obtain $\left(
\mathfrak{dist}\right)  \int_{0}^{1}f\left(  x\right)  \,\mathrm{d}%
x=-1/2 .$\smallskip
\end{example}

\begin{example}
In general the sum of a Lebesgue integrable function and a \L ojasiewicz
function is neither Lebesgue integrable nor a \L ojasiewicz function, but it
is certainly distributionally integrable. For instance, if $F$ is a closed set
with empty interior and positive Lebesgue measure in $\left[  a,b\right]  $
then $\chi_{F}\left(  x\right)  +J_{\nu}\left(  \left(  x^{2}%
-(a+b)x+ab\right)  ^{-1}\right)  $ is a distributionally integrable function
that is not Lebesgue integrable nor a distributionally regulated function in
$\left[  a,b\right]  .$ It is worth pointing out that decompositions of
distributions as sums of terms involving a \L ojasiewicz distribution have
shown to be of great importance in the study of distributional composition
operations \cite{AKS}. \smallskip
\end{example}

\begin{example}
If $a<b$ denote by $f_{a,b}$ the function
\begin{equation}
f_{a,b}\left(  x\right)  =\left[  \chi_{\left[  a,b\right]  }\left(  x\right)
\sin(\left(  x^{2}-(a+b)x+ab\right)  ^{-1})\right]  ^{\prime}, \label{Exa.4m}%
\end{equation}
Let $\left\{  \left(  a_{n},b_{n}\right)  \right\}  _{n=1}^{\infty}$ be a
sequence of mutually disjoint open intervals and let $\sum_{n=1}^{\infty}%
M_{n}$ be an absolutely convergent series. Let
\[
f=\sum_{n=1}^{\infty}M_{n}f_{a_{n},b_{n}}\,.\label{Exa.5}%
\]
Then $f$ is a locally distributionally integrable function and
\begin{equation}
\left(  \mathfrak{dist}\right)  \int_{c}^{d}f\left(  x\right)  \,\mathrm{d}%
x=F\left(  d\right)  -F\left(  c\right)  \,, \label{Exa.6}%
\end{equation}
where $F\left(  x\right)  =M_{n}\sin(\left(  x^{2}-(a_{n}+b_{n})x+a_{n}%
b_{n}\right)  ^{-1})$ if $x\in\left(  a_{n},b_{n}\right)  ,$ while $F\left(
x\right)  =0$ if $x\notin\bigcup_{n=1}^{\infty}\left(  a_{n},b_{n}\right)
.$\smallskip
\end{example}

\begin{example}
It is well known that there are continuous functions whose derivatives do not
exist at any point. Actually, one can show in many cases that the
distributional derivative does not have values at any point. For instance
\cite{VEFourier}\ if $\left\{  a_{n}\right\}  _{n\in\mathbb{Z}}$ is a lacunary
sequence such that $a_{n}\neq o(1)$ but $a_{n}=O(1),$ $\left\vert n\right\vert
\rightarrow\infty,$ then
\begin{equation}
G(x)=\sum_{n=-\infty}^{\infty}\frac{a_{n}}{n}e^{inx}, \label{Exa.7}%
\end{equation}
is continuous, but if $G\leftrightarrow\mathsf{G,}$ then $\mathsf{g=G}%
^{\prime}$ does not have distributional point values at any point. It follows
that $G$ is not the indefinite integral of a distributionally integrable
function in any interval.\smallskip
\end{example}

\begin{example}
The Heaviside function $H\left(  x\right)  $ does not have a value at the
origin, of course, and thus it is not the indefinite integral of a
distributionally integrable function. Naturally, if $H\leftrightarrow
\mathsf{H,}$ then $\mathsf{H}^{\prime}\left(  x\right)  =\delta\left(
x\right)  $ is \emph{not }a function. Similarly, the continuous increasing
Cantor function defined in $\left[  0,1\right]  $ is not the indefinite
integral of a distributionally integrable function in any interval that meets
the Cantor set, since its distributional derivative is not a function but
rather a measure concentrated on the Cantor set.\smallskip
\end{example}

\begin{example}
Let $f$ be a locally distributionally integrable function such that if
$f\leftrightarrow\mathsf{f,}$ then $\mathsf{f}\in\mathcal{D}^{\prime}\left(
\mathbb{R}\right)  $ satisfies
\begin{equation}
\mathsf{f}\left(  x\right)  =O\left(  \left\vert x\right\vert ^{\beta}\right)
\ \ \ \ \left(  \text{C}\right)  \,,\ \ \ \ \left\vert x\right\vert
\rightarrow\infty\,, \label{Exa.8}%
\end{equation}
for some $\beta<1.$ Let%
\begin{equation}
F\left(  x,y\right)  =\frac{y}{\pi}\left(  \mathfrak{dist}\right)
\int_{-\infty}^{\infty}\frac{f\left(  \xi\right)  \,\mathrm{d}\xi}{\left(
x-\xi\right)  ^{2}+y^{2}}  \ \ \ (\mathrm{C})\,, \label{Exa.9}%
\end{equation}
for $x\in\mathbb{R}$ and $y>0.$ The function $F$ is the $\phi-$transform of
$\mathsf{f}$ with respect to the function%
\begin{equation}
\phi\left(  x\right)  =\frac{1}{\pi\left(  x^{2}+1\right)  }\,. \label{BB.5}%
\end{equation}
Naturally, $F\left(  x,y\right)  $ is the Poisson integral\ of $f,$ which is
the harmonic function with $F\left(  x,0^{+}\right)  =\mathsf{f}\left(
x\right)  $ that satisfies $F\left(  x,y\right)  =O(\left\vert x\right\vert
^{\beta})$ $\left(  \text{C}\right)  ,$ $\left\vert x\right\vert
\rightarrow\infty,$ for each fixed $y>0.$

The boundary behavior of $F$ is as follows: $F\left(  x,y\right)  \rightarrow
f\left(  w\right)  $ as $\left(  x,y\right)  \rightarrow\left(  w,0\right)  $
in any sector $y\geq m\left\vert x-w\right\vert $ for $m>0,$ almost everywhere
with respect to $w\in\mathbb{R};$ this holds for all $w\in\mathbb{R}$ whenever
$f$ is a \L ojasiewicz function.\smallskip
\end{example}

\begin{example}
Let us now consider the Fourier transform of tempered locally integrable
distributions. The characterization of the Fourier series of those periodic
distributions that have a distributional point value was given in \cite{est1}:
If $\mathsf{f}\left(  \theta\right)  =\sum_{n=-\infty}^{\infty}a_{n}%
e^{in\theta}$ in the space $\mathcal{D}^{\prime}\left(  \mathbb{R}\right)  $
then%
\begin{equation}
\mathsf{f}\left(  \theta_{0}\right)  =\gamma
\,,\ \ \ \text{distributionally\ ,} \label{T.8}%
\end{equation}
if and only if there exists $k$ such that%
\begin{equation}
\lim_{x\rightarrow\infty}\sum_{-x\leq n\leq ax}a_{n}e^{in\theta_{0}}%
=\gamma\ \ \ \left(  \mathrm{C},k\right)  \,,\ \ \ \forall a>0\,. \label{T.9}%
\end{equation}
Therefore, if $f$ is a periodic locally distributionally integrable function,
of period $2\pi,$ then the coefficients
\begin{equation}
a_{n}=\frac{1}{2\pi} \left(  \mathfrak{dist}\right)  \int_{0}^{2\pi}f\left(
\theta\right)  e^{-in\theta}\,\mathrm{d}\theta\,,
\end{equation}
are well defined for all $n\in\mathbb{Z},$ and
\begin{equation}
\lim_{x\rightarrow\infty}\sum_{-x\leq n\leq ax}a_{n}e^{in\theta}=f\left(
\theta\right)  \ \ \ \left(  \mathrm{C},k\right)  \,,\ \ \ \forall a>0\,,
\label{T.9p}%
\end{equation}
almost everywhere with respect to $\theta.$ If $f$ is a \L ojasiewicz function
then (\ref{T.9p}) holds for all $\theta\in\mathbb{R}.$

The characterization of the values of general Fourier transforms
\cite{VEStudia, VEFourier} is as follows: Let $\mathsf{f}\in\mathcal{S}%
^{\prime}\left(  \mathbb{R}\right)  ,$ and let $x_{0}\in\mathbb{R},$ then%
\begin{equation}
\mathsf{f}\left(  x_{0}\right)  =\gamma\,,\ \ \ \text{distributionally\ ,}
\label{T.15}%
\end{equation}
if and only if%
\begin{equation}
\mathrm{e.v.}\left\langle \widehat{\mathsf{f}}\left(  u\right)  ,e^{-iux_{0}%
}\right\rangle =2\pi\gamma\ \ \ \ \left(  \mathrm{C}\right)  \,. \label{T.16}%
\end{equation}
We have chosen the constants in the Fourier transform in such a way that
\begin{equation}
\widehat{f}\left(  u\right)  =\int_{-\infty}^{\infty}f\left(  x\right)
e^{ixu}\mathrm{d}x\,, \label{T.16p}%
\end{equation}
if the integral makes sense. In case $\widehat{f}$ is locally distributionally
integrable this means that%
\begin{equation}
\mathrm{e.v.}\left(  \mathfrak{dist}\right)  \int_{-\infty}^{\infty}%
\widehat{f}\left(  u\right)  e^{-iux_{0}}\mathrm{d}u=2\pi\gamma\ \ \ \ \left(
\mathrm{C}\right)  \,. \label{T.17}%
\end{equation}

Suppose now that $\mathsf{f}$ is a locally integrable tempered distribution,
$f\leftrightarrow\mathsf{f}.$ Then
\begin{equation}
\mathrm{e.v.}\left\langle \widehat{\mathsf{f}}\left(  u\right)  ,e^{-iux_{0}%
}\right\rangle =2\pi f\left(  x_{0}\right)  \ \ \ \ \left(  \mathrm{C}\right)
\,, \label{t.18}%
\end{equation}
almost everywhere with respect to $x_{0},$ and actually everywhere if $f$ is a
\L ojasiewicz function. When $\widehat{\mathsf{f}}$ is also locally
integrable, $\widehat{f}\leftrightarrow\widehat{\mathsf{f}},$ then
(\ref{t.18}) becomes%
\[
\mathrm{e.v.}\left(  \mathfrak{dist}\right)  \int_{-\infty}^{\infty}%
\widehat{f}\left(  u\right)  e^{-iux_{0}}\mathrm{d}u=2\pi f\left(
x_{0}\right)  \ \ \ \ \left(  \mathrm{C}\right)  \,.
\]

\end{example}

\begin{example}
Let $\mathsf{f}$ be a distribution defined in the complement of the origin,
$\mathsf{f}\in\mathcal{D}^{\prime}\left(  \mathbb{R}\setminus\left\{
0\right\}  \right)  .$ There may or may not be a distribution $\mathsf{g}%
_{0}\in\mathcal{D}^{\prime}\left(  \mathbb{R}\right)  $ whose restriction to
$\mathbb{R}\setminus\left\{  0\right\}  $ is $\mathsf{f,}$ what we call a
regularization of $\mathsf{f,}$ but if one regularization exists then there
are infinitely many regularizations, since $\mathsf{g}\left(  x\right)
=\mathsf{g}_{0}\left(  x\right)  +\sum_{j=0}^{m}a_{j}\delta^{\left(  j\right)
}\left(  x\right)  $ is also a regularization for any constants $a_{0}%
,\ldots,a_{m}.$ It is known \cite{est2003}\ that there is no continuous way to
choose the regularization $\mathsf{g.}$ Suppose now that $\mathsf{f}$
corresponds to a function $f,$ locally distributionally integrable in
$\mathbb{R}\setminus\left\{  0\right\}  ;$ if $f$ is distributionally
integrable at $x=0$ then it has an associated distribution $\mathsf{g}_{0}%
\in\mathcal{D}^{\prime}\left(  \mathbb{R}\right)  ,$ and then $\mathsf{g}_{0}$
is the \emph{canonical} regularization of the distribution
$\mathsf{f.\smallskip}$
\end{example}

\begin{example}
Consider now an analytic function $F\left(  z\right)  $ defined in the upper
half plane $\Im m\,z>0,$ which we assume to vanish at $\infty,$ $F\left(
z\right)  \rightarrow0$ as $z\rightarrow\infty$ angularly in the half plane.
Suppose that the distributional limit $\mathsf{g}\left(  x\right)  =F\left(
x+i0\right)  $ exists in $\mathcal{D}^{\prime}\left(  \mathbb{R}\right)  $
\cite{beltrami, bremermann}, and let $\mathsf{f}$ be its restriction to
$\mathbb{R}\setminus\left\{  0\right\}  ,$ which we assume to correspond to a
locally distributionally integrable function $f$ in $\mathbb{R}\setminus
\left\{  0\right\}  .$ If $f$ is not distributionally\ integrable at $x=0$
then it does not define a canonical distribution in the whole real line, but
$\mathsf{f}$ will have at least a regularization $\mathsf{g}_{0}.$ However, if
$f$ is also distributionally integrable at $x=0$ then a corresponding
distribution $\mathsf{g}_{0}\in\mathcal{D}^{\prime}\left(  \mathbb{R}\right)
$ is defined by the association $f\leftrightarrow\mathsf{g}_{0}.$ Both
$\mathsf{g}$ and $\mathsf{g}_{0}$ are regularizations of $\mathsf{f;}$ in
general they do not coincide, but if $f$ is distributionally integrable at
$x=0$ then the results of \cite{EstradaComVar} immediately yield that
$\mathsf{g}=\mathsf{g}_{0},$ so that, if the integral converges at infinity,
we have the Cauchy representation%
\begin{equation}
F\left(  z\right)  =\left(  \frac{1}{2\pi i}\right)  \left(  \mathfrak{dist}%
\right)  \int_{-\infty}^{\infty}\frac{f\left(  \xi\right)  }{\xi
-z}\,\mathrm{d}\xi\,,\ \ \ \ \Im m\,z>0\,. \label{Exa.10}%
\end{equation}

Take $F\left(  z\right)  =1/z.$ In this case $f\left(  x\right)  =1/x$ is not
distributionally integrable at $x=0.$ The standard regularization of
$\mathsf{f}\in\mathcal{D}^{\prime}\left(  \mathbb{R}\setminus\left\{
0\right\}  \right)  ,$ $1/x\leftrightarrow\mathsf{f}$ is $\mathsf{g}%
_{0}=\left(  \ln\left\vert x\right\vert \right)  ^{\prime}=\mathrm{p.v.}%
\left(  1/x\right)  .$ However, $\mathsf{g}$ is another regularization,
namely, $\mathsf{g}\left(  x\right)  =\left(  x+i0\right)  ^{-1}%
=\mathsf{g}_{0}\left(  x\right)  -\pi i\delta\left(  x\right)  .$ Formula
(\ref{Exa.10}) does not even make sense in this case.

On the other hand, if $F\left(  z\right)  =e^{-i/z}/z,$ then $f\left(
x\right)  =e^{-i/x}/x$ is actually distributionally integrable at $x=0,$ and
thus has an associated distribution $\mathsf{g}_{0}\in\mathcal{D}^{\prime
}\left(  \mathbb{R}\right)  .$ In this case $\mathsf{g}=\mathsf{g}_{0},$ and
(\ref{Exa.10}) becomes%
\begin{equation}
\frac{e^{-i/z}}{z}=\left(  \frac{1}{2\pi i}\right)  \left(  \mathfrak{dist}%
\right)  \int_{-\infty}^{\infty}\frac{e^{-i/\xi}}{\xi(\xi-z)}\,\mathrm{d}%
\xi\,, \label{Exa.11}%
\end{equation}
in the half plane $\Im m\,z>0.$
\end{example}

%%%%% Enter the widest reference label as the first parameter%%%%%

%%% Only the title of article is italicized. No boldface numbers are used.%%%
%%% Give the issue number only if each issue of the journal is separately paginated.%%% 


\begin{thebibliography}{99}                                                                                               %


\bibitem {AKS}P. Antosik, A. Kami\'{n}ski and S. Sorek, \textit{On the
distributional compositions of functions and distributions,} Integral
Transforms Spec. Funct. 20 (2009), 247--255.

\bibitem {AMS}P. Antosik, J. Mikusi\'{n}ski and R. Sikorski, \textit{Theory of
distributions. The sequential approach,} Elsevier Scientific Publishing Co.,
Amsterdam; PWN---Polish Scientific Publishers, Warsaw, 1973.

\bibitem {Bartle}R. G. Bartle, \textit{A Modern Theory of Integration, }Amer.
Math. Soc., Providence, R.I., 2001.

\bibitem {Baumer et al}B. B\"{a}umer, G. Lumer and F. Neubrander,
\textit{Convolution kernels and generalized functions,} in: Generalized
Functions, Operator Theory, and Dynamical Systems, (Brussels, 1997), pp.
68--78, Chapman \& Hall/CRC Res. Notes Math. 399, Chapman \&
Hall/CRC, Boca Raton, FL, 1999.

\bibitem {beltrami}E. J. Beltrami and M. R. Wohlers, \textit{Distributions and
the Boundary Values of Analytic Functions,} Academic Press, New York, 1966.

\bibitem {bremermann}H. Bremermann, \textit{Distributions, Complex Variables
and Fourier Transforms,} Addison-Wesley, Reading, Massachusetts, 1965.

\bibitem {CF}J. Campos Ferreira, \textit{Introduction to the Theory of
Distributions,} Longman, Essex, 1997.

\bibitem {CF2}J. Campos Ferreira, \textit{On some general notions of superior
limit, inferior limit and value of a distribution at a point,} Portugaliae
Math. 28 (2001), 139--158.

\bibitem {den}A. Denjoy, \textit{Sur l'int\'{e}gration des coefficients
diff\'{e}rentiels d'ordre sup\'{e}rieur,} Fund. Math. 25 (1935), 237--320.

\bibitem {Dido}J. Dieudonne, \textit{Foundations of Modern Analysis,} Academic
Press, New York, 1969.

%\bibitem {Donaghue}W. F. Donoghue,\textit{ Distributions and Fourier
%Transforms,} Academic Press, New York, 1969.


\bibitem {DroZav03}Yu. N. Drozhzhinov and B.I. Zav'yalov, \textit{Tauberian
theorems for generalized functions with values in Banach spaces,} Izv. Ross.
Akad. Nauk Ser. Mat. 66 (2002), 47--118 (in Russian); translation in:
Izv. Math. 66 (2002), 701--769.

\bibitem {est1}R. Estrada, \textit{Characterization of the Fourier series of a
distribution having a value at a point,} Proc. Amer. Math. Soc. 124
(1996), 1205--1212.

\bibitem {est2}R. Estrada, \textit{The Ces\`{a}ro behavior of distributions,}
Proc. Roy. Soc. London Ser. A 454 (1998), 2425--2443.

\bibitem {est2003}R. Estrada, \textit{The non-existence of regularization
operators,} J. Math. Anal. Appl. 286 (2003), 1--10.

\bibitem {est3}R. Estrada, \textit{A distributional version of the Ferenc
Luk\'{a}cs theorem,} Sarajevo J. Math. 1(13) (2005), 75--92.

\bibitem {EstradaComVar}R. Estrada, \textit{One-sided cluster sets of
distributional boundary values of analytic functions,} Complex Var. Elliptic
Equ. 51 (2006), 661--673.

\bibitem {estPol}R. Estrada, \textit{Distributions that are functions,} in:
Linear and non-linear theory of generalized functions and its applications,
pp. 91--110, Banach Center Publ. 88, Polish Acad. Sci. Inst. Math.,
Warsaw, 2010.

\bibitem {EF}R. Estrada and S. A. Fulling, \textit{Distributional asymptotic
expansion of spectral functions and of the associated Green kernels,}
Electron. J. Differential Equations (1999), 1--37.

\bibitem {EF2001}R. Estrada and S. A. Fulling, \textit{How singular functions
define distributions,} J. Physics A 35 (2002), 3079--3089.

\bibitem {EF2007}R. Estrada and S. A. Fulling, \textit{Functions and
distributions in spaces with thick points,} Int. J. Appl. Math. Stat. 10 (2007), 25--37.

\bibitem {r43}R. Estrada and R. P. Kanwal, \textit{A distributional theory of
asymptotic expansions,} Proc.\ Roy. Soc. London Ser. A. 428 (1990), 399--430.

\bibitem {est-kan}R. Estrada and R. P. Kanwal, \textit{A Distributional
Approach to Asymptotics. Theory and Applications,} Second edition,
Birkh\"{a}user, Boston, 2002.

\bibitem{EVDiffMeans} R. Estrada and J. Vindas, \textit{Determination of jumps of distributions by differentiated means,} Acta Math. Hungar. 124 (2009), 215--241.

\bibitem {EVraex}R. Estrada and J. Vindas, \textit{On Romanovski's lemma,}
Real Anal. Exchange 35 (2010), 431--444.

\bibitem {Ford-Oconnell}G. W. Ford and R. F. O'Connell, \textit{Derivative of
the hyperbolic tangent,} Nature 380 (1996), 113.

\bibitem {Gordon}R. A. Gordon, \textit{The integrals of Lebesgue, Denjoy,
Perron, and Henstock,} Amer. Math. Soc., Providence, 1994.

\bibitem {Hake}H. Hake, \textit{\"{U}ber de la Vall\'{e}e Poussins Ober-und
Unterfunktionen einfacher Integrale und die Integraldefinitionen von Perron,}
Math. Annalen 83 (1921), 119--142.

\bibitem {Hobson}E. W. Hobson, \textit{The theory of functions of a real
variable and the theory of Fourier series, }vol.1, Dover, New York, 1956.

\bibitem {kan}R. P. Kanwal, \textit{Generalized Functions: Theory and
Technique,} Second edition, Birkh\"{a}user, Boston, 1998.

\bibitem {loj}S. \L ojasiewicz, \textit{Sur la valuer et la limite d'une
distribution en un point,} Studia Math. 16 (1957), 1--36.

\bibitem {loj2}S. \L ojasiewicz, \textit{Sur la fixation de variables dans une
distribution,} Studia Math. 17 (1958), 1--64.

\bibitem {Mikusinski-Ostas}P. Mikusi\'{n}ski and K. Ostaszewski,
\textit{Embedding Henstock integrable functions into the space of Schwartz
distributions,} Real Anal. Exchange 14 (1988-89), 24--29.

\bibitem {ML}O. P. Misra and J. Lavoine, \textit{Transform Analysis of
Generalized Functions,} North-Holland, Amsterdam, 1986.

\bibitem {Natanson2}I. P. Natanson, \textit{Theory of functions of a real
variable,} Frederick Ungar Pub., New York, 1955.

\bibitem {Peetre}J. Peetre, \textit{On the value of a distribution at a
point,} Portugaliae Math. 27 (1968), 149--159.

\bibitem {pil}S. Pilipovi\'{c}, B. Stankovi\'{c} and A. Taka\v{c}i,
\textit{Asymptotic Behavior and Stieltjes Transformation of Distributions,}
Teubner-Texte zur Mathmatik, Leipzig, 1990.

\bibitem {VPwavelet}S. Pilipovi\'{c} and J. Vindas, \textit{Multidimensional
Tauberian theorems for wavelet and non-wavelet transforms,} preprint (arXiv:1012.5090v2 [math.FA]).

\bibitem {Roma}P. Romanovski, \textit{Essai d'une exposition de l'integrale de
Denjoy sans nombres transfini,} Fund. Math. 19 (1932), 38--44.

\bibitem {Silva}J. Sebasti\~{a}o e Silva, \textit{Integrals and orders of
growth of distributions,} in: Theory of Distributions (Proc. Internat. Summer
Inst., Lisbon, 1964), pp. 327--390, Inst. Gulbenkian Ci., Lisbon, 1964.

\bibitem {Sacks}S. Saks, \textit{Theory of the Integral,} Second edition,
Dover, New York, 1964.

\bibitem {sch}L. Schwartz, \textit{Th\'{e}orie des Distributions,} Hermann,
Paris, 1966.

\bibitem {Sikorski}R. Sikorski, \textit{Integrals of distributions,} Studia
Math. 20 (1961), 119--139.

\bibitem {Strichartz}R. S. Strichartz, \textit{A guide to distribution theory
and Fourier transforms,} World Scientific Publ., Singapore, 2003.

\bibitem {Talvila}E. Talvila, \textit{The distributional Denjoy integral,}
Real Anal. Exchange 33 (2008), 51--82.

\bibitem {vindas1}J. Vindas, \textit{Structural theorems for quasiasymptotics
of distributions at infinity}, Publ. Inst. Math. (Beograd) (N.S.)
84(98), 159--174.

\bibitem {vindasthesis}J. Vindas, \textit{Local behavior of distributions and
applications,} Dissertation, Louisiana State University, Baton Rouge, 2009.

\bibitem {vindas2}J. Vindas, \textit{The structure of quasiasymptotics of
Schwartz distributions}, in: Linear and non-linear theory of generalized
functions and its applications, pp. 297--314, Banach Center Publ. 88,
Polish Acad. Sci. Inst. Math., Warsaw, 2010.

\bibitem {VEStudia}J. Vindas and R. Estrada, \textit{Distributionally
regulated functions,} Studia Math. 181 (2007), 211--236.

\bibitem {VEFourier}J. Vindas and R. Estrada, \textit{Distributional point
values and convergence of Fourier series and integrals,} J. Fourier Anal.
Appl. 13 (2007), 551--576.

\bibitem{VEJumpLog} J. Vindas and R. Estrada, \textit{On the jump behavior of distributions and logarithmic averages,} J. Math. Anal. Appl. 347 (2008), 597--606. 

\bibitem {VEmeasures}J. Vindas and R. Estrada, \textit{Measures and the
distributional $\phi$-transform,} Integral Transforms Spec. Funct. 20
(2009), 325--332.

\bibitem {VEEdin}J. Vindas and R. Estrada, \textit{On the support of tempered
distributions,} Proc. Edin Math. Soc. 53 (2010), 255--270.

\bibitem {VEChina}J. Vindas and R. Estrada, \textit{On the order of
summability of the Fourier inversion formula,} Anal. Theory Appl. 26
(2010), 13--42.

\bibitem {VP09}J. Vindas and S. Pilipovi\'{c}, \textit{Structural theorems for
quasiasymptotics of distributions at the origin}, Math. Nachr. 282
(2009), 1584--1599.

\bibitem {Vladimirovbook}V. S. Vladimirov, \textit{Methods of the theory of
generalized functions,} Taylor \& Francis, London, 2002.

\bibitem {vla}V. S. Vladimirov, Yu. N. Drozhzhinov and B. I. Zavialov,
\textit{Tauberian theorems for generalized functions,} Kluwer Academic
Publishers Group, Dordrecht, 1988.

\bibitem {wal}G. Walter, \textit{Pointwise convergence of distribution
expansions,} Studia Math. 26 (1966), 143--154.

\bibitem {Walter88}G. Walter, \textit{Abel summability for a distribution
sampling theorem,} in: Generalized functions, convergence structures, and
their applications (Dubrovnik, 1987), pp. 349--357, Plenum Press, New York, 1988.

\bibitem {Walter95}G. Walter, \textit{Pointwise convergence of wavelet
expansions,} J. Approx. Theory 80 (1995), 108--118.

\bibitem {Walter01}G. Walter and X. Shen,\textit{ Wavelets and other
orthogonal systems,} Second edition, Studies in Advanced Mathematics, Chapman
\& Hall/CRC, Boca Raton, 2001.

\bibitem {zem}A. H. Zemanian, \textit{Generalized Integral Transforms,}
Interscience, New York, 1965.

\bibitem {Zie}Z. Ziele\'{z}ny, \textit{\"{U}ber die Mengen der regul\"{a}ren
und singul\"{a}ren Punkte einer Distribution}, Studia Math. 19
(1960), 27---52.
\end{thebibliography}
\end{document}